\documentclass[reqno,10pt,a4paper]{amsart}
\usepackage[margin=2.7cm]{geometry}
\usepackage{cite}
\usepackage{newpxtext}
\usepackage{cabin}% sans serif
\usepackage[varqu,varl]{inconsolata}% sans serif typewriter
\usepackage[bigdelims,vvarbb]{newpxmath}% bb from STIX
\usepackage[cal=cm,bb=esstix,bbscaled=1.05]{mathalfa}% mathcal
\usepackage{graphicx}
\usepackage{colortbl}
\usepackage{booktabs}
\usepackage[up]{caption}
\usepackage[labelformat=simple]{subcaption}
\usepackage{xcolor}
\usepackage{hyperref}
\hypersetup{
    colorlinks,
    linkcolor={red!50!black},
    citecolor={blue!50!black},
    urlcolor={blue!80!black}
}
\AtBeginDocument{%
\def\MR#1{}
}
\newcommand{\orcid}[1]{\,\resizebox{8px}{!}{\href{https://orcid.org/#1}{\includegraphics{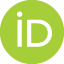}}}}
\graphicspath{{images_300/}} % 300 dpi
\newcommand{\ZZ}{\mathbb{Z}}
\newcommand{\PP}{\mathbb{P}}
\newcommand{\CC}{\mathbb{C}}
\newcommand{\QQ}{\mathbb{Q}}
\newcommand{\RR}{\mathbb{R}}
\newcommand{\FF}{\mathbb{F}}

\newcommand{\Cstar}{\CC^\times}

\newcommand{\Ghat}{\widehat{G}}
\newcommand{\tti}{\mathtt{i}}
\DeclareMathOperator{\Spec}{Spec}

\newcommand{\stderr}{s_{\text{int}}}

\DeclareMathOperator{\MLP}{MLP}
\newcommand{\doi}[1]{\href{https://doi.org/#1}{doi:#1}}

\theoremstyle{plain} 
\newtheorem{Lemma}{Lemma}
\newtheorem{Thm}[Lemma]{Theorem}
\newtheorem*{LocalThm}{Local Theorem}
\newtheorem{Prop}[Lemma]{Proposition}

\theoremstyle{definition}

\newtheorem{Ex}[Lemma]{Example}

\theoremstyle{remark}

\begin{document}
%-------------------------------------------------------------------------------
\author[T.\,Coates]{Tom Coates\orcid{0000-0003-0779-9735}}
\address{Department of Mathematics\\Imperial College London\\180 Queen's Gate\\London\\SW7 2AZ\\UK}
\email{t.coates@imperial.ac.uk}% Coates
\author[A.\,M.\,Kasprzyk]{Alexander M.\ Kasprzyk\orcid{0000-0003-2340-5257}}
\address{School of Mathematical Sciences\\University of Nottingham\\Nottingham\\NG7 2RD\\UK}
\email{a.m.kasprzyk@nottingham.ac.uk}% Kasprzyk
\author[S.\,Veneziale]{Sara Veneziale\orcid{0000-0003-2851-3820}}
\address{Department of Mathematics\\Imperial College London\\180 Queen's Gate\\London\\SW7 2AZ\\UK}
\email{s.veneziale21@imperial.ac.uk}% Veneziale
%-------------------------------------------------------------------------------
\makeatletter
\@namedef{subjclassname@2020}{%
  \textup{2020} Mathematics Subject Classification}
\makeatother
%-------------------------------------------------------------------------------
\title{Machine Learning the Dimension of a Fano~Variety}
%-------------------------------------------------------------------------------
\begin{abstract}
Fano varieties are basic building blocks in geometry -- they are `atomic pieces' of mathematical shapes. Recent progress in the classification of Fano varieties involves analysing an invariant called the quantum period. This is a sequence of integers which gives a numerical fingerprint for a Fano variety. It is conjectured that a Fano variety is uniquely determined by its quantum period. If this is true, one should be able to recover geometric properties of a Fano variety directly from its quantum period. We apply machine learning to the question: does the quantum period of~$X$ know the dimension of~$X$? Note that there is as yet no theoretical understanding of this. We show that a simple feed-forward neural network can determine the dimension of~$X$ with $98\%$ accuracy. Building on this, we establish rigorous asymptotics for the quantum periods of a class of Fano varieties. These asymptotics determine the dimension of~$X$ from its quantum period. Our results demonstrate that machine learning can pick out structure from complex mathematical data in situations where we lack theoretical understanding. They also give positive evidence for the conjecture that the quantum period of a Fano variety determines that variety. 
\end{abstract}
%-------------------------------------------------------------------------------
\maketitle
%-------------------------------------------------------------------------------
\section{Introduction}
%-------------------------------------------------------------------------------
Algebraic geometry describes shapes as the solution sets of systems of polynomial equations, and manipulates or analyses a shape~$X$ by manipulating or analysing the equations that define~$X$. This interplay between algebra and geometry has applications across mathematics and science; see e.g.~\cite{vanLintvanderGeer1988,NiederreiterXing2009,AtiyahDrinfeldHitchinManin1978,ErikssonRanestadSturmfelsSullivant2005}. Shapes defined by polynomial equations are called~\emph{algebraic varieties}. Fano varieties are a key class of algebraic varieties. They are, in a precise sense, atomic pieces of mathematical shapes~\cite{Kollar1987,KollarMori1998}. Fano varieties also play an essential role in string theory. They provide, through their `anticanonical sections', the main construction of the Calabi--Yau manifolds which give geometric models of spacetime~\cite{CandelasHorowitzStromingerWitten1985,Greene1997,Polchinski2005}.

The classification of Fano varieties is a long-standing open problem. The only one-dimensional example is a line; this is classical. The ten smooth two-dimensional Fano varieties were found by del~Pezzo in the 1880s~\cite{DelPezzo1887}. The classification of smooth Fano varieties in dimension three was a triumph of 20th century mathematics: it combines work by Fano in the~1930s, Iskovskikh in the~1970s, and Mori--Mukai in the 1980s~\cite{Fano1947,Iskovskikh1977,Iskovskikh1978,Iskovskikh1979,MoriMukai1982,MoriMukai1982Erratum}. Beyond this, little is known, particularly for the important case of Fano varieties that are not smooth.

A new approach to Fano classification centres around a set of ideas from string theory called Mirror Symmetry~\cite{CandelasdelaOssaGreenParkes1991,GreenePlesser1990,HoriVafa2000,CoxKatz1999}. From this perspective, the key invariant of a Fano variety is its~\emph{regularized quantum period}~\cite{CoatesCortiGalkinGolyshevKasprzyk2013}
\begin{equation} \label{eq:Ghat}
    \Ghat_X(t) = \sum_{d=0}^\infty c_d t^d
\end{equation}
This is a power series with coefficients~$c_0 = 1$,~$c_1 = 0$, and $c_d = r_d d!$, where~$r_d$ is a certain Gromov--Witten invariant of~$X$. Intuitively speaking,~$r_d$ is the number of rational curves in~$X$ of degree~$d$ that pass through a fixed generic point and have a certain constraint on their complex structure. In general~$r_d$ can be a rational number, because curves with a symmetry group of order~$k$ are counted with weight~$1/k$, but in all known cases the coefficients~$c_d$ in~\eqref{eq:Ghat} are integers. 

It is expected that the regularized quantum period $\Ghat_X$ uniquely determines~$X$. This is true (and proven) for smooth Fano varieties in low dimensions, but is unknown  in dimensions four and higher, and for Fano varieties that are not smooth.

In this paper we will treat the regularized quantum period as a numerical signature for the Fano variety~$X$, given by the sequence of integers~$(c_0,c_1,\ldots)$. \emph{A priori} this looks like an infinite amount of data, but in fact there is a differential operator~$L$ such that~$L \Ghat_X\equiv 0$; see e.g.~\cite[Theorem~4.3]{CoatesCortiGalkinGolyshevKasprzyk2013}. This gives a recurrence relation that determines all of the coefficients~$c_d$ from the first few terms, so the regularized quantum period~$\Ghat_X$ contains only a finite amount of information. Encoding a Fano variety~$X$ by a vector in~$\ZZ^{m+1}$ given by finitely many coefficients~$(c_0,c_1,\ldots,c_m)$ of the regularized quantum period allows us to investigate questions about Fano varieties using machine learning.

In this paper we ask whether the regularized quantum period of a Fano variety~$X$ knows the dimension of~$X$. There is currently no viable theoretical approach to this question. Instead we use machine learning methods applied to a large dataset to argue that the answer is probably yes, and then prove that the answer is yes for toric Fano varieties of low Picard rank. The use of machine learning was essential to the formulation of our rigorous results (Theorems~\ref{thm:wps} and~\ref{thm:rank_2} below). This work is therefore proof-of-concept for a larger program, demonstrating that machine learning can uncover previously unknown structure in complex mathematical datasets. Thus the Data Revolution, which has had such impact across the rest of science, also brings important new insights to pure mathematics~\cite{DaviesEtAl2021,He2022,Wagner2021,ErbinFinotello2021, LevittHajijSazdanovic2022, WuDeLoera2022}. This is particularly true for large-scale classification questions, e.g.~\cite{KreuzerSkarke2000,ConwayCurtisNortonParkerWilson1985,Cremona2016,LieAtlas,CoatesKasprzyk2022}, where these methods can potentially reveal both the classification itself and structural relationships within~it.
%-------------------------------------------------------------------------------
\section{Results}
%-------------------------------------------------------------------------------
\subsubsection*{Algebraic varieties can be smooth or have singularities}
%-------------------------------------------------------------------------------
Depending on their equations, algebraic varieties can be smooth (as in Figure~\ref{fig:non_singular}) or have singularities (as in Figure~\ref{fig:singular}). In this paper we consider algebraic varieties over the complex numbers. The equations in Figures~\ref{fig:non_singular} and~\ref{fig:singular} therefore define complex surfaces; however, for ease of visualisation, we have plotted only the points on these surfaces with co-ordinates that are real numbers.

Most of the algebraic varieties that we consider below will be singular, but they all have a class of singularities called \emph{terminal quotient singularities}. This is the most natural class of singularities to allow from the point of view of Fano classification~\cite{KollarMori1998}. Terminal quotient singularities are very mild; indeed, in dimensions one and two, an algebraic variety has terminal quotient singularities if and only if it is smooth.

\begin{figure}[tbp]
  \centering
  \begin{subfigure}{.2\textwidth}
    \includegraphics[width=\linewidth]{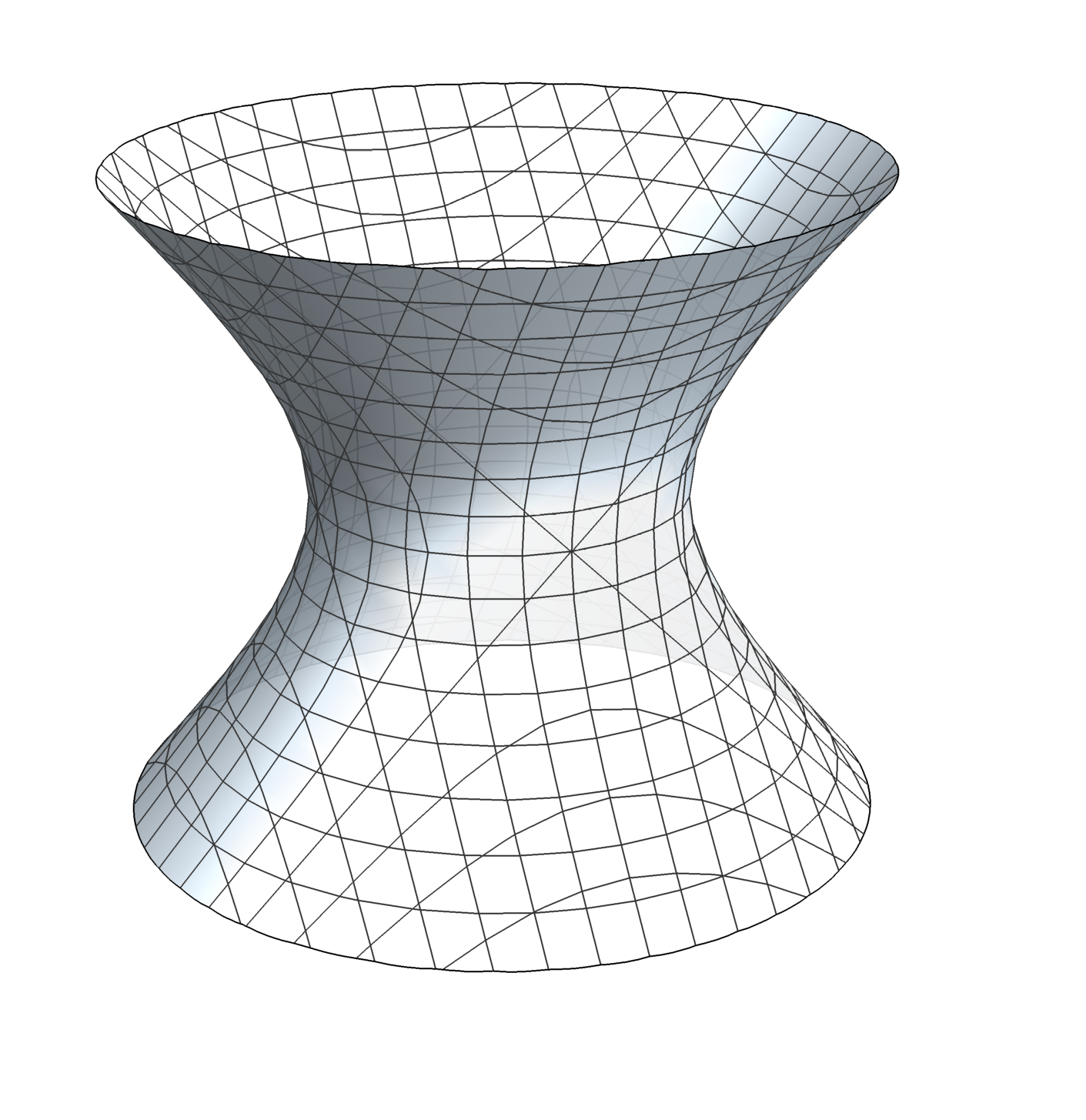}
    \caption{$x^2+y^2 = z^2 + 1$}\label{fig:non_singular}
  \end{subfigure}
  \qquad
  \begin{subfigure}{.2\textwidth}
    \includegraphics[width=\linewidth]{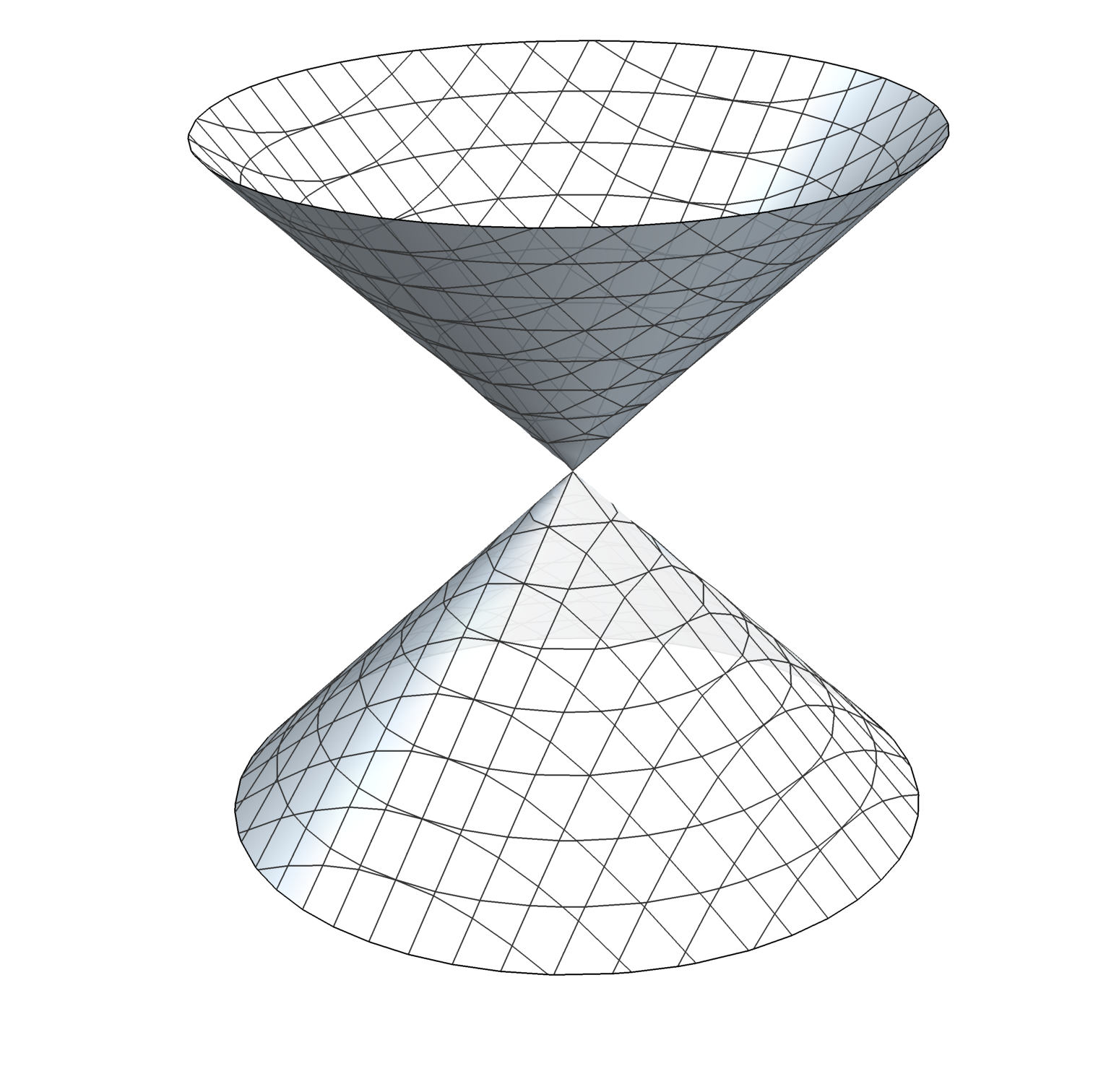}
    \caption{$x^2 + y^2 = z^2$}\label{fig:singular}
  \end{subfigure}
  \caption{Algebraic varieties and their equations:
    (a)~a smooth example; (b)~an example with a singular point.}\label{fig:singular_or_not}
\end{figure}
%-------------------------------------------------------------------------------
\subsubsection*{The Fano varieties that we consider}
%-------------------------------------------------------------------------------
The fundamental example of a Fano variety is projective space~$\PP^{N-1}$. This is a quotient of~$\CC^N \setminus \{0\}$ by the group~$\Cstar$, where the action of~$\lambda \in \Cstar$ identifies the points $(x_1,x_2,\ldots,x_N)$ and~$(\lambda x_1, \lambda x_2, \ldots, \lambda x_N)$. The resulting algebraic variety is smooth and has dimension~$N-1$. We will consider generalisations of projective spaces called~\emph{weighted projective spaces} and~\emph{toric varieties of Picard rank two}. A detailed introduction to these spaces is given in~\S\ref{sec:supp_notes}.

To define a weighted projective space, choose positive integers~$a_1,a_2,\ldots,a_N$ such that any subset of size~$N-1$ has no common factor, and consider
\[
\PP(a_1,a_2,\ldots,a_N) = (\CC^N \setminus \{0\})/\Cstar
\]
where the action of~$\lambda \in \Cstar$ identifies the points
\[
\begin{aligned}
    (x_1,x_2,\ldots,x_N) && \text{and} && (\lambda^{a_1} x_1, \lambda^{a_2} x_2, \ldots, \lambda^{a_N} x_N)
\end{aligned}
\]
\noindent in~$\CC^N \setminus \{0\}$. The quotient~$\PP(a_1,a_2,\ldots,a_N)$ is an algebraic variety of dimension~$N-1$. A general point of~$\PP(a_1,a_2,\ldots,a_N)$ is smooth, but there can be singular points. Indeed, a weighted projective space~$\PP(a_1,a_2,\ldots,a_N)$ is smooth if and only if $a_i = 1$ for all~$i$, that is, if and only if it is a projective space.

To define a toric variety of Picard rank two, choose a matrix 
\begin{equation} \label{eq:weights}
    \begin{pmatrix}
        a_1 & a_2 & \cdots & a_N \\
        b_1 & b_2 & \cdots & b_N 
    \end{pmatrix}
\end{equation}
with non-negative integer entries and no zero columns. This defines an action of~$\Cstar \times \Cstar$ on~$\CC^N$, where~$(\lambda,\mu) \in \Cstar \times \Cstar$ identifies the points
\[
\begin{aligned}
    (x_1,x_2,\ldots,x_N) && \text{and} && (\lambda^{a_1} \mu^{b_1} x_1, \lambda^{a_2} \mu^{b_2} x_2, \ldots, \lambda^{a_N} \mu^{b_N} x_N)
\end{aligned}
\]
in~$\CC^N$. Set~$a = a_1 + a_2 + \cdots + a_N$ and~$b = b_1 + b_2 + \cdots + b_N$, and suppose that~$(a,b)$ is not a scalar multiple of~$(a_i,b_i)$ for any~$i$. This determines linear subspaces
\[
\begin{aligned}
    S_+ & = \left\{ (x_1,x_2,\ldots,x_N)\mid \text{$x_i = 0$ if $b_i/a_i < b/a$} \right\} \\
    S_- & = \left\{ (x_1,x_2,\ldots,x_N)\mid \text{$x_i = 0$ if $b_i/a_i > b/a$} \right\}
\end{aligned}
\]
of~$\CC^N$, and we consider the quotient
\begin{equation} \label{eq:rank_2}
X = (\CC^N \setminus S)/(\Cstar \times \Cstar)
\end{equation}
where~$S = S_+ \cup S_-$. The quotient~$X$ is an algebraic variety of dimension~$N-2$ and second Betti number~$b_2(X) \leq 2$.  If, as we assume henceforth, the subspaces~$S_+$ and~$S_-$ both have dimension two or more then~$b_2(X) = 2$, and thus~$X$ has Picard rank two. In general~$X$ will have singular points, the precise form of which is determined by the weights in~\eqref{eq:weights}. 

There are closed formulas for the regularized quantum period of weighted projective spaces and toric varieties~\cite{CoatesCortiGalkinKasprzyk2016}. We have
\begin{equation} \label{eq:Ghat_wP}
       \Ghat_{\PP}(t) = \sum_{k = 0}^\infty \frac{(ak)!}{(a_1 k)! (a_2 k)! \cdots (a_N k)!} t^{ak}
\end{equation}
where~$\PP = \PP(a_1,\ldots,a_N)$ and~$a = a_1 + a_2 + \cdots + a_N$, and 
\begin{equation}
       \label{eq:Ghat_rank_two}
       \Ghat_X(t) =\!\!\! \sum_{(k, l) \in \ZZ^2 \cap C}\!\frac{(ak + bl)!}{(a_1 k+b_1 l)! \cdots (a_N k + b_N l)!} t^{ak + bl}
\end{equation}
where the weights for~$X$ are as in~\eqref{eq:weights}, and~$C$ is the cone in~$\RR^2$ defined by the equations~$a_i x + b_i y \geq 0$,~$i \in \{1,2,\ldots,N\}$. Formula~\eqref{eq:Ghat_wP} implies that, for weighted projective spaces, the coefficient $c_d$ from~\eqref{eq:Ghat} is zero unless~$d$ is divisible by~$a$. Formula~\eqref{eq:Ghat_rank_two} implies that, for toric varieties of Picard rank two, $c_d = 0$ unless $d$ is divisible by $\gcd\{a,b\}$.
%-------------------------------------------------------------------------------
\subsubsection*{Data generation: weighted projective spaces}
%-------------------------------------------------------------------------------
The following result characterises weighted projective spaces with terminal quotient singularities; this is~\cite[Proposition~2.3]{Kasprzyk2013}. 

\begin{Prop}\label{prop:wps_terminal}
Let~$X=\PP(a_1,a_2,\ldots,a_N)$ be a weighted projective space of dimension at least three. Then~$X$ has terminal quotient singularities if and only if
\[
\sum_{i=1}^N\{ka_i/a\}\in\{2,\ldots,N-2\}
\]
for each~$k\in\{2,\ldots,a-2\}$. Here~$a=a_1+a_2+\cdots+a_N$ and~$\{q\}$ denotes the fractional part~$q - \lfloor q \rfloor$ of~$q \in \QQ$.
\end{Prop}

\noindent  A simpler necessary condition is given by~\cite[Theorem~3.5]{Kasprzyk2009}:

\begin{Prop} \label{prop:wps_bound}
Let~$X=\PP(a_1,a_2,\ldots,a_N)$ be a weighted projective space of dimension at least two, with weights ordered~$a_1\leq a_2\leq\ldots\leq a_N$. If~$X$ has terminal quotient singularities then
$a_i/a<1/(N-i+2)$ for each~$i\in\{3,\ldots,N\}$.
\end{Prop}

Weighted projective spaces with terminal quotient singularities have been classified in dimensions up to four~\cite{Kasprzyk2006,Kasprzyk2013}. Classifications in higher dimensions are hindered by the lack of an effective upper bound on~$a$.

We randomly generated~$150\,000$ distinct weighted projective spaces with terminal quotient singularities, and with dimension up to~$10$, as follows. We generated random sequences of weights $a_1\leq a_2\leq\ldots\leq a_N$ with~$a_N \leq 10N$ and discarded them if they failed to satisfy any one of the following:
\begin{enumerate}
    \item\label{item:wellformed} for each $i \in \{1,\ldots,N\}$, $\gcd\{a_1,\ldots,\widehat{a}_i,\ldots,a_N\}=1$, where $\widehat{a}_i$ indicates that $a_i$ is omitted;
    \item\label{item:terminal_bound} $a_i/a<1/(N-i+2)$ for each $i\in\{3,\ldots,N\}$;
    \item\label{item:terminal_iff} $\sum_{i=1}^N\{ka_i/a\}\in\{2,\ldots,N-2\}$ for each $k\in\{2,\ldots,a-2\}$.
\end{enumerate}
Condition~\eqref{item:wellformed} here was part of our definition of weighted projective spaces above; it ensures that the set of singular points in~$\PP(a_1,a_2,\ldots,a_N)$ has dimension at most~$N-2$, and also that weighted projective spaces are isomorphic as algebraic varieties if and only if they have the same weights. Condition~\eqref{item:terminal_bound} is from Proposition~\ref{prop:wps_bound}; it efficiently rules out many non-terminal examples. Condition~\eqref{item:terminal_iff} is the necessary and sufficient condition from Proposition~\ref{prop:wps_terminal}. We then deduplicated the sequences. The resulting sample sizes are summarised in Table~\ref{tab:both_samples}. 
%-------------------------------------------------------------------------------
\subsubsection*{Data generation: toric varieties}\label{sec:data_generation}
%-------------------------------------------------------------------------------
Deduplicating randomly-generated toric varieties of Picard rank two is harder than deduplicating randomly generated weighted projective spaces, because different weight matrices in~\eqref{eq:weights} can give rise to the same toric variety. Toric varieties are uniquely determined, up to isomorphism, by a combinatorial object called a~\emph{fan}~\cite{Fulton1993}. A~fan is a collection of cones, and one can determine the singularities of a toric variety~$X$ from the geometry of the cones in the corresponding fan. 

We randomly generated~$200\,000$ distinct toric varieties of Picard rank two with terminal quotient singularities, and with dimension up to~$10$, as follows. We randomly generated weight matrices, as in~\eqref{eq:weights}, such that~$0 \leq a_i, b_j \leq 5$. We then discarded the weight matrix if any column was zero, and otherwise formed the corresponding fan~$F$. We discarded the weight matrix unless:
\begin{enumerate}
    \item\label{item:non_degenerate} $F$ had~$N$ rays;
    \item\label{item:simplicial} each cone in~$F$ was simplicial (i.e.~has number of rays equal to its dimension);
    \item\label{item:terminal_fan} the convex hull of the primitive generators of the rays of~$F$ contained no lattice points other than the rays and the origin.
\end{enumerate}
Conditions~\eqref{item:non_degenerate} and~\eqref{item:simplicial} together guarantee that~$X$ has Picard rank two, and are equivalent to the conditions on the weight matrix in~\eqref{eq:weights} given in our definition. Conditions~\eqref{item:simplicial} and~\eqref{item:terminal_fan} guarantee that~$X$ has terminal quotient singularities. We then deduplicated the weight matrices according to the isomorphism type of~$F$, by putting~$F$ in normal form~\cite{KreuzerSkarke2004,GrinisKasprzyk2013}. See Table~\ref{tab:both_samples} for a summary of the dataset.

\begin{table*}[tbp]
    \small
    \centering
    \begin{tabular}{crr@{}p{4mm}@{}crr}
    \toprule
    \multicolumn{3}{c}{Weighted projective spaces} && 
    \multicolumn{3}{c}{Rank-two toric varieties} \\
    \cmidrule{1-3} \cmidrule{5-7} 
    Dimension&Sample size&Percentage&&Dimension&Sample size&Percentage\\
    \midrule
    1&1&0.001\\
    2&1&0.001&&2&2&0.001\\
    3&7&0.005&&3&17&0.009\\
    4&8\,936&5.957&&4&758&0.379\\
    5&23\,584&15.723&&5&6\,050&3.025\\
    6&23\,640&15.760&&6&19\,690&9.845\\
    7&23\,700&15.800&&7&35\,395&17.698\\
    8&23\,469&15.646&&8&42\,866&21.433\\
    9&23\,225&15.483&&9&47\,206&23.603\\
    10&23\,437&15.625&&10&48\,016&24.008\\
    \midrule
    Total&150\,000&&&Total&200\,000&\\
    \bottomrule
    \end{tabular}
    \caption{The distribution by dimension in our datasets.}
    \label{tab:both_samples}
\end{table*}
%-------------------------------------------------------------------------------
\subsubsection*{Data analysis: weighted projective spaces}
%-------------------------------------------------------------------------------
We computed an initial segment~$(c_0,c_1,\ldots,c_m)$ of the regularized quantum period for all the examples in the sample of~$150\,000$ terminal weighted projective spaces, with~$m \approx 100\,000$. The non-zero coefficients~$c_d$ appeared to grow exponentially with~$d$, and so we considered~$\{\log c_d\}_{d\in S}$ where $S= \{d \in \ZZ_{\geq 0}\mid c_d \neq 0\}$. To reduce dimension we fitted a linear model to the set~$\{(d, \log c_d)\mid d \in S\}$ and used the slope and intercept of this model as features; see Figure~\ref{fig:lin_approximation_pr1} for a typical example. Plotting the slope against the $y$\nobreakdash-intercept and colouring datapoints according to the dimension we obtain Figure~\ref{fig:colour_triangle_all_pr1}: note the clear separation by dimension. A Support Vector Machine~(SVM) trained on~10\% of the slope and $y$\nobreakdash-intercept data predicted the dimension of the weighted projective space with an accuracy of~99.99\%. Full details are given in~\S\S\ref{sec:supp_methods_1}--\ref{sec:supp_methods_2}.

\begin{figure}[tbp]
  \centering
  \begin{subfigure}{0.45\textwidth}
    \includegraphics[width=\linewidth]{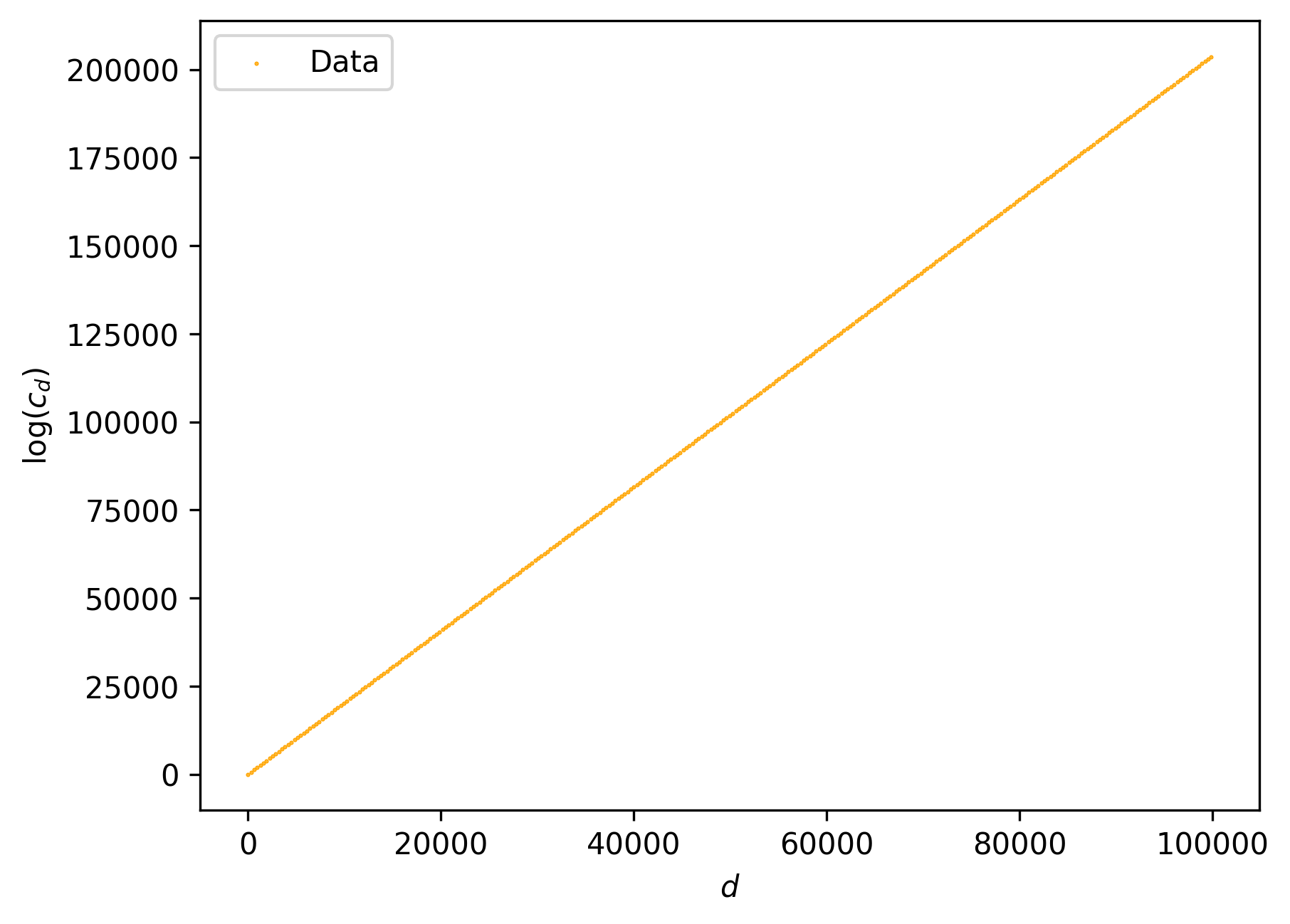}
    \caption{$\log c_d$ for~$\PP(5,5,11,23,28,29,33,44,66,76)$.}\label{fig:lin_approximation_pr1}
  \end{subfigure}
  \qquad
  \begin{subfigure}{0.45\textwidth}
    \includegraphics[width=\linewidth]{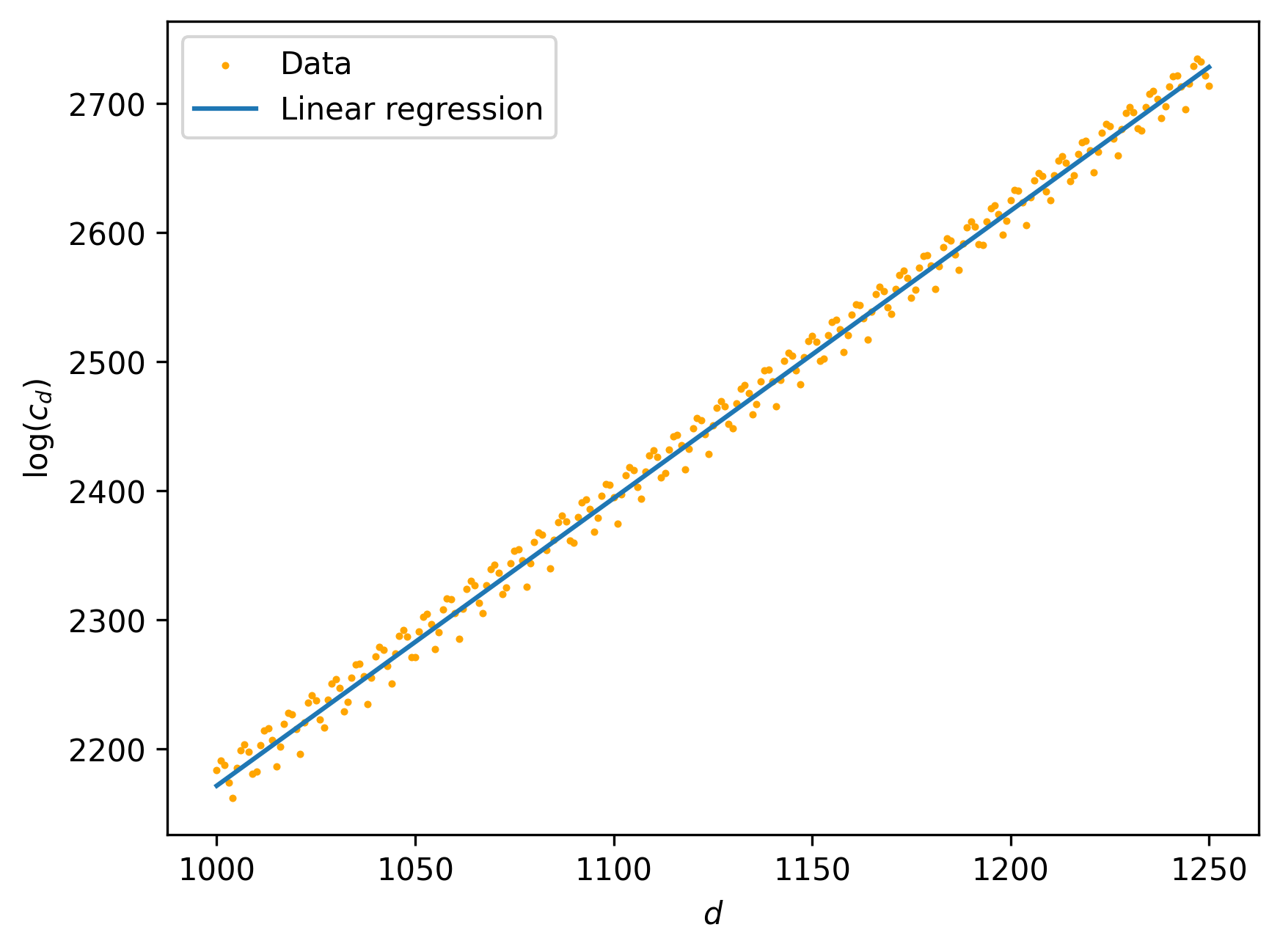}
    \caption{$\log c_d$ for~Example~\ref{eg:rank_2_growth}.}\label{fig:lin_approximation_pr2}
  \end{subfigure}
  \caption{The logarithm of the non-zero period coefficients~$c_d$: (a)~for a typical weighted projective space; (b)~for the toric variety of Picard rank two from Example~\ref{eg:rank_2_growth}.}
  \label{fig:linear_approximations}
\end{figure}

\begin{figure*}[tbp]
  \centering
  \begin{subfigure}{0.45\linewidth}
    \includegraphics[width=\textwidth]{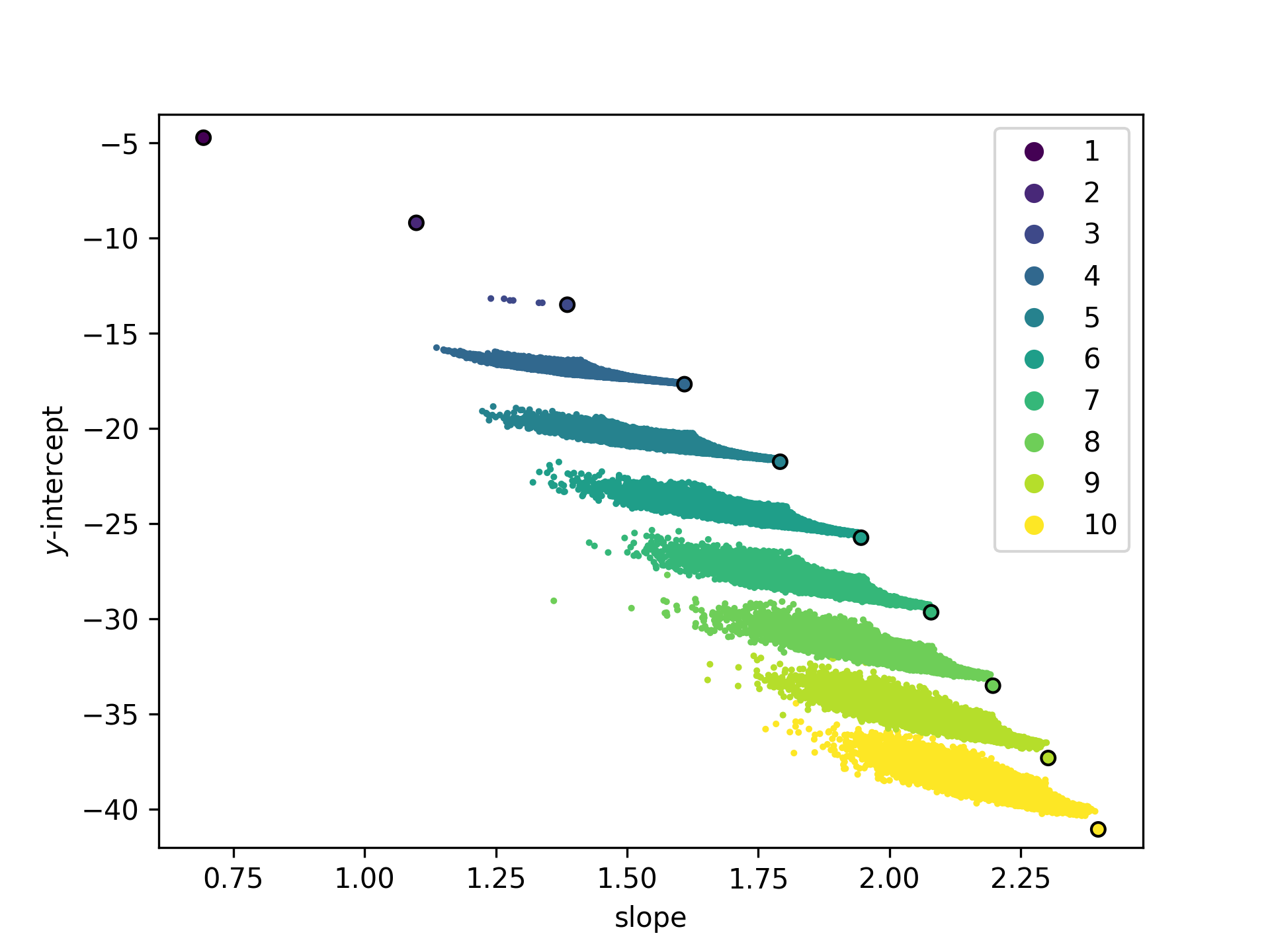}
    \caption{}
    \label{fig:colour_triangle_all_pr1}
  \end{subfigure}
  \qquad
  \begin{subfigure}{0.45\linewidth}
    \includegraphics[width=\textwidth]{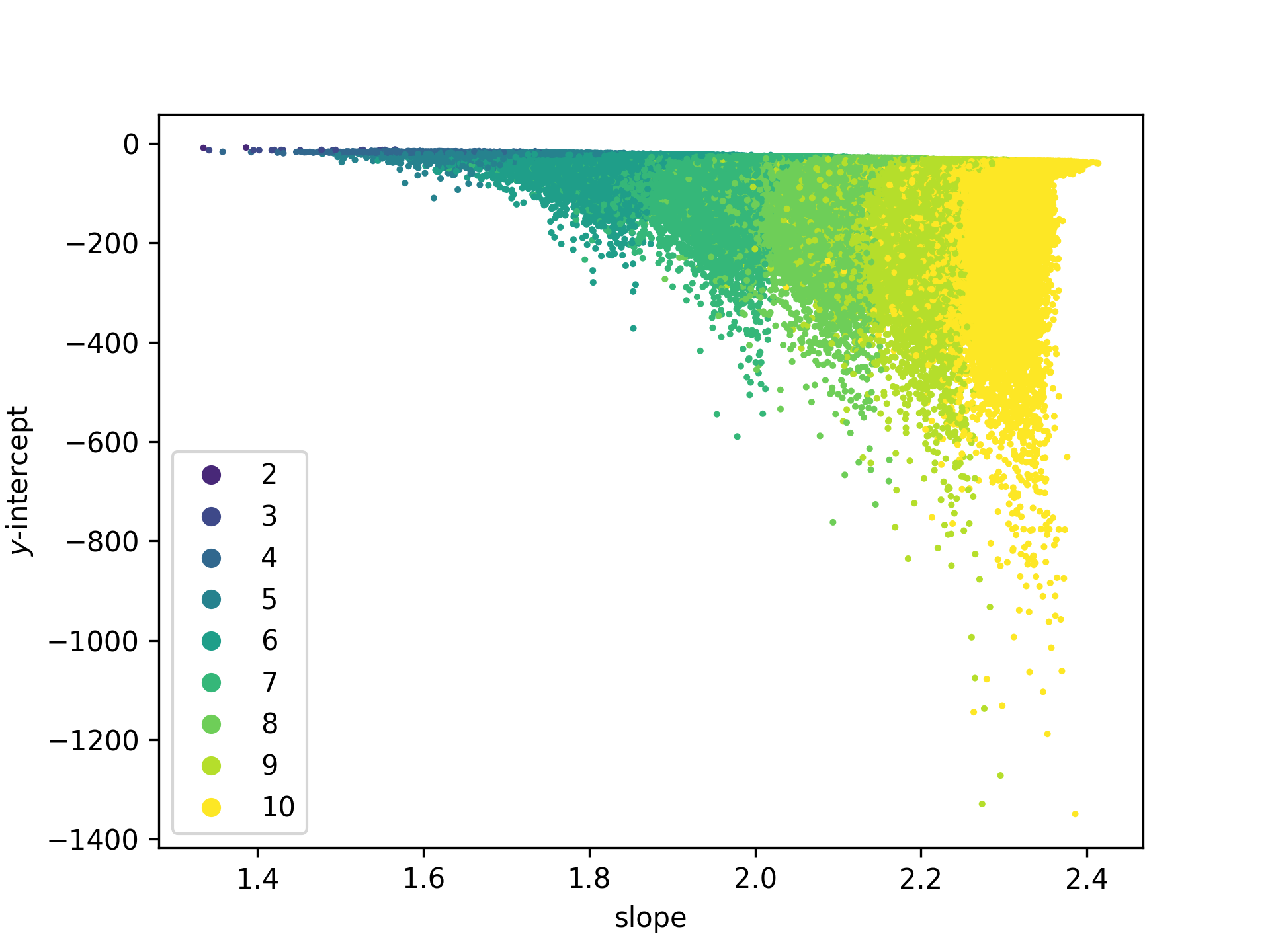}
    \caption{}
    \label{fig:colour_triangle_all}
  \end{subfigure}
  \caption{The slopes and $y$-intercepts from the linear models: (a)~for weighted projective spaces with terminal quotient singularities. The colour records the dimension of the weighted projective space and the circled points indicate projective spaces. (b)~for toric varieties of Picard rank two with terminal quotient singularities. The colour records the dimension of the toric variety.}
\end{figure*}
%-------------------------------------------------------------------------------
\subsubsection*{Data analysis: toric varieties}
%-------------------------------------------------------------------------------
As before, the non-zero coefficients~$c_d$ appeared to grow exponentially with~$d$, so we fitted a linear model to the set~$\{(d, \log c_d)\mid d \in S\}$ where~$S= \{d \in \ZZ_{\geq 0} \mid c_d \neq 0\}$. We used the slope and intercept of this linear model as features. 

\begin{Ex}\label{eg:rank_2_growth}
In Figure~\ref{fig:lin_approximation_pr2} we plot a typical example: the logarithm of the regularized quantum period sequence for the nine\nobreakdash-dimensional toric variety with weight matrix
\[
% Note: too many columns for pmatrix
\left(\begin{array}{ccccccccccc}
1 & 2 & 5 & 3 & 3 & 3 & 0 & 0 & 0 & 0 & 0 \\
0 & 0 & 0 & 3 & 4 & 4 & 1 & 2 & 2 & 3 & 4
\end{array}\right)
\]
along with the linear approximation. We see a periodic deviation from the linear approximation; the magnitude of this deviation decreases as~$d$ increases (not shown). 
\end{Ex}

To reduce computational costs, we computed pairs~$(d,\log{c_d})$ for~$1000\leq d\leq 20\,000$ by sampling every~$100$th term. We discarded the beginning of the period sequence because of the noise it introduces to the linear regression. In cases where the sampled coefficient~$c_d$ is zero, we considered instead the next non-zero coefficient. The resulting plot of slope against $y$\nobreakdash-intercept, with datapoints coloured according to dimension, is shown in Figure~\ref{fig:colour_triangle_all}.

We analysed the standard errors for the slope and $y$\nobreakdash-intercept of the linear model. The standard errors for the slope are small compared to the range of slopes, but in many cases the standard error~$\stderr$ for the $y$\nobreakdash-intercept is relatively large. As Figure~\ref{fig:less_0.3_standard_error} illustrates, discarding data points where the standard error~$\stderr$ for the $y$\nobreakdash-intercept exceeds some threshold reduces apparent noise. This suggests that the underlying structure is being obscured by inaccuracies in the linear regression caused by oscillatory behaviour in the initial terms of the quantum period sequence; these inaccuracies are concentrated in the $y$\nobreakdash-intercept of the linear model. Note that restricting attention to those data points where~$\stderr$ is small also greatly decreases the range of $y$\nobreakdash-intercepts that occur. As Example~\ref{ex:outlier} and Figure~\ref{fig:outlier_intercept} suggest, this reflects both transient oscillatory behaviour and also the presence of a subleading term in the asymptotics of~$\log c_d$ which is missing from our feature set. We discuss this further below.

\begin{Ex} \label{ex:outlier}
Consider the toric variety with Picard rank two and weight matrix
\[
  \begin{pmatrix}
    1 & 10 & 5 & 13 & 8 & 12 & 0 \\
    0 & 0 & 3 & 8 & 5 & 14 & 1
  \end{pmatrix}
\]
% 30000 terms: Time: 1057.380s
% 40000 terms: Time: 2905.620s
This is one of the outliers in Figure~\ref{fig:colour_triangle_all}. The toric variety is five\nobreakdash-dimensional, and has slope~$1.637$ and $y$\nobreakdash-intercept $-62.64$.  The standard errors are~$4.246 \times 10^{-4}$ for the slope and~$5.021$ for the $y$\nobreakdash-intercept. We computed the first~$40\,000$ coefficients~$c_d$ in~\eqref{eq:Ghat}. As Figure~\ref{fig:outlier_intercept} shows, as~$d$ increases the $y$\nobreakdash-intercept of the linear model increases to~$-28.96$ and~$\stderr$ decreases to~$0.7877$. At the same time, the slope of the linear model remains more or less unchanged, decreasing to~$1.635$. This supports the idea that computing (many) more coefficients~$c_d$ would significantly reduce noise in Figure~\ref{fig:colour_triangle_all}. In this example, even~$40\,000$ coefficients may not be enough.
\end{Ex}

\begin{figure}[tbp]
       \centering
       \includegraphics[width=0.4\textwidth]{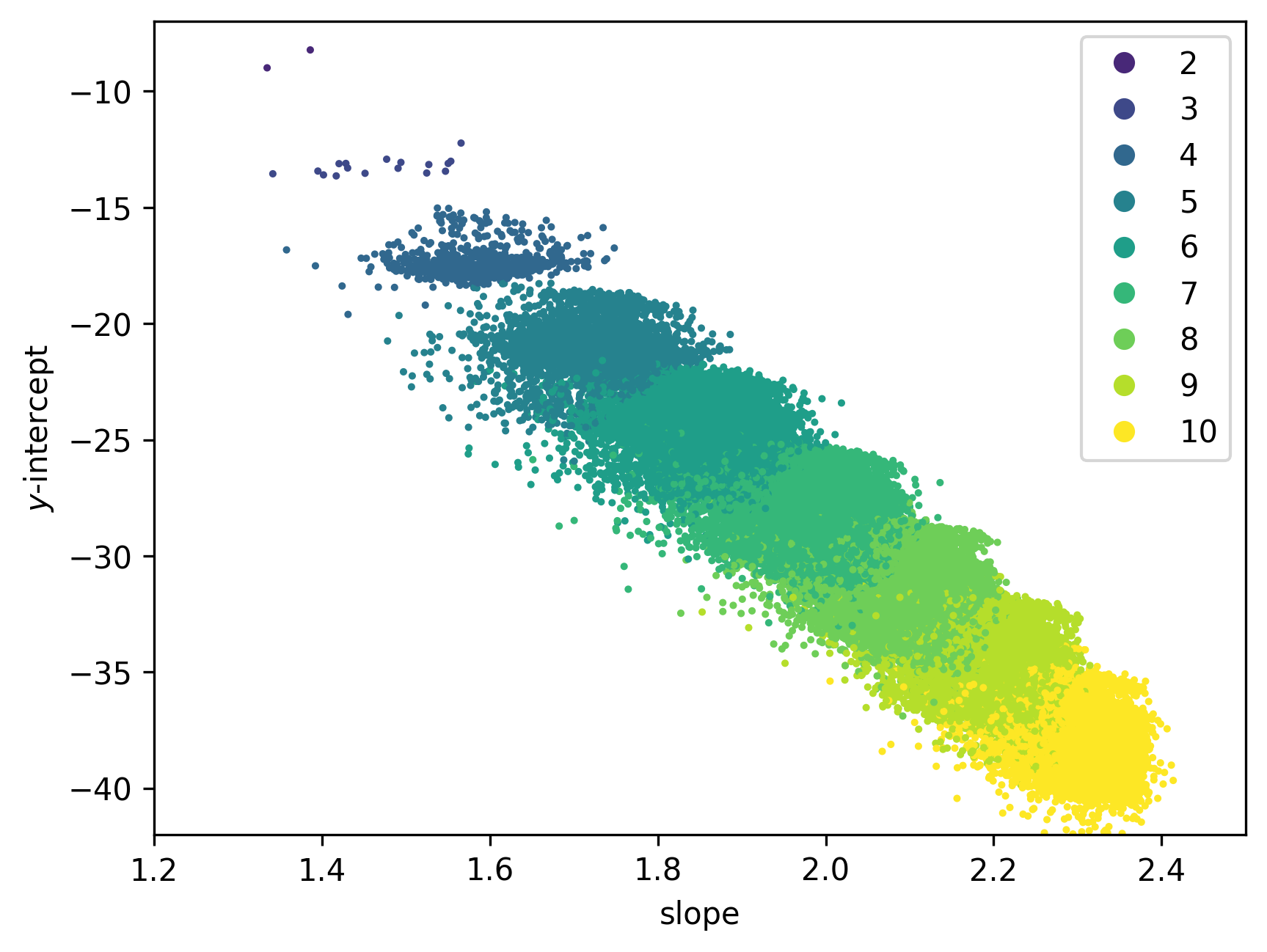}
       \caption{The slopes and $y$-intercepts from the linear model. This is as in Figure~\ref{fig:colour_triangle_all}, but plotting only data points for which the standard error~$\stderr$ for the $y$-intercept satisfies~$\stderr < 0.3$. The colour records the dimension of the toric variety.}
       \label{fig:less_0.3_standard_error}
\end{figure}

\begin{figure}[tbp]
       \centering
       \includegraphics[width=0.4\textwidth]{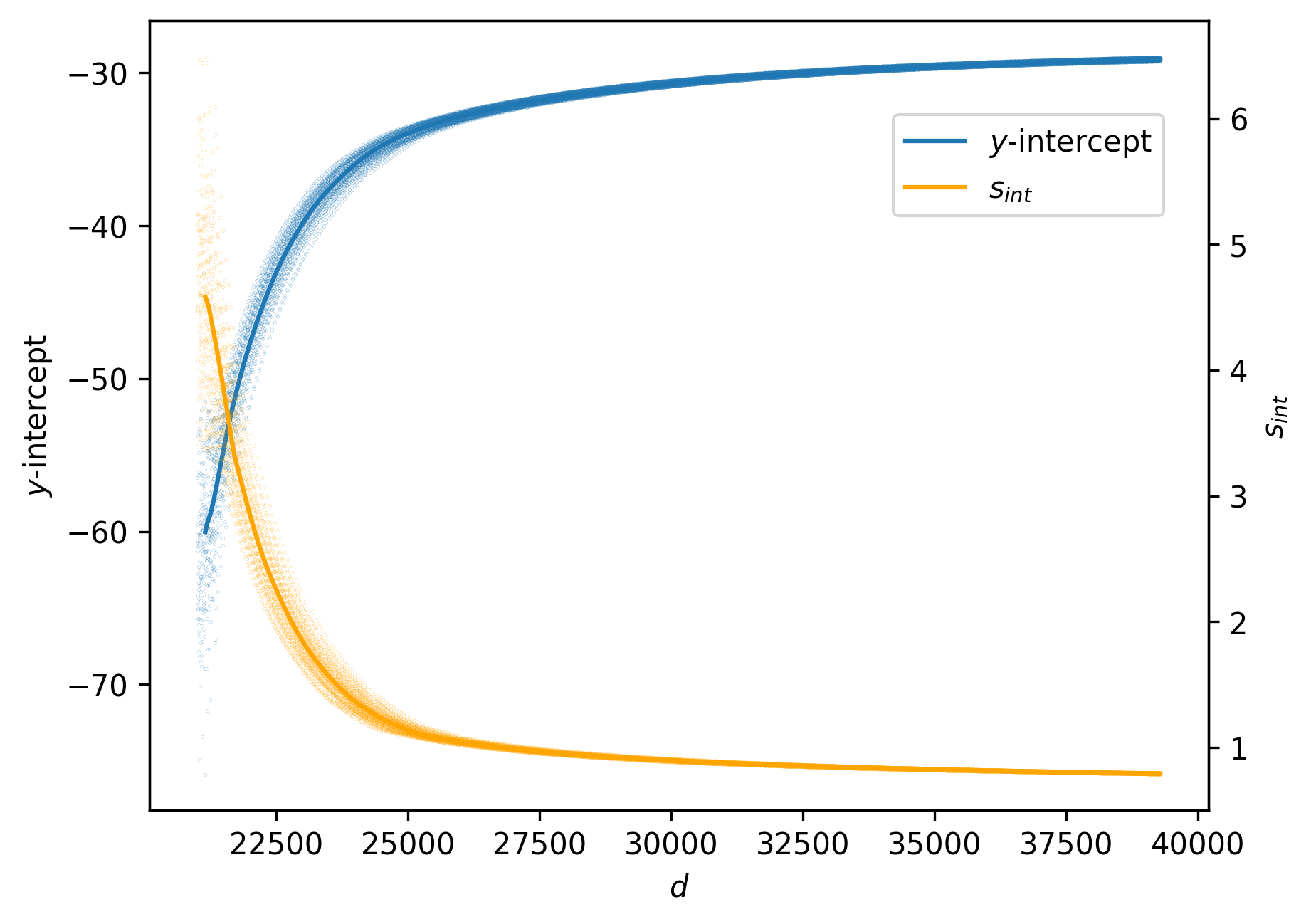}
       \caption{Variation as we move deeper into the period sequence. The $y$-intercept and its standard error $\stderr$ for the toric variety from Example~\ref{ex:outlier}, as computed from pairs~$(k,\log{c_k})$ such that~$d - 20\,000\leq k \leq d$ by sampling every~$100$th term. We also show LOWESS-smoothed trend lines.}
       \label{fig:outlier_intercept}
\end{figure}

Computing many more coefficients~$c_d$ across the whole dataset would require impractical amounts of computation time. In the example above, which is typical in this regard, increasing the number of coefficients computed from~$20\,000$ to~$40\,000$ increased the computation time by a factor of more than~$10$. Instead we restrict to those  toric varieties of Picard rank two such that the $y$\nobreakdash-intercept standard error~$\stderr$ is less than~$0.3$; this retains~$67\,443$ of the~$200\,000$ datapoints. We used~70\% of the slope and $y$\nobreakdash-intercept data in the restricted dataset for model training, and the rest for validation. An~SVM model predicted the dimension of the toric variety with an accuracy of~87.7\%, and a Random Forest Classifier~(RFC) predicted the dimension with an accuracy of~88.6\%.
%-------------------------------------------------------------------------------
\subsubsection*{Neural networks}
%-------------------------------------------------------------------------------
Neural networks do not handle unbalanced datasets well. We therefore removed the toric varieties of dimensions~3,~4, and~5 from our data, leaving~$61\,164$ toric varieties of Picard rank two with terminal quotient singularities and $\stderr < 0.3$. This dataset is approximately balanced by dimension.

A Multilayer Perceptron (MLP) with three hidden layers of sizes $(10,30,10)$ using the slope and intercept as features predicted the dimension with 89.0\% accuracy. Since the slope and intercept give good control over $\log c_d$ for $d \gg 0$, but not for small~$d$, it is likely that the coefficients $c_d$ with $d$ small contain extra information that the slope and intercept do not see. Supplementing the feature set by including the first $100$ coefficients $c_d$ as well as the slope and intercept increased the accuracy of the prediction to 97.7\%. Full details can be found in~\S\S\ref{sec:supp_methods_1}--\ref{sec:supp_methods_2}. 
%-------------------------------------------------------------------------------
\subsubsection*{From machine learning to rigorous analysis}
%-------------------------------------------------------------------------------
Elementary ``out of the box'' models (SVM, RFC, and MLP) trained on the slope and intercept data alone already gave a highly accurate prediction for the dimension. Furthermore even for the many-feature MLP, which was the most accurate, sensitivity analysis using SHAP values~\cite{LundbergLee2017} showed that the slope and intercept were substantially more important to the prediction than any of the coefficients $c_d$: see Figure~\ref{fig:SHAP_plot}. This suggested that the dimension of $X$ might be visible from a rigorous estimate of the growth rate of $\log c_d$. 

In~\S\ref{sec:methods} we establish asymptotic results for the regularized quantum period of toric varieties with low Picard rank, as follows. These results apply to any weighted projective space or toric variety of Picard rank two: they do not require a terminality hypothesis. Note, in each case, the presence of a subleading logarithmic term in the asymptotics for~$\log c_d$. 

\begin{Thm} \label{thm:wps}
    Let~$X$ denote the weighted projective space~$\PP(a_1,\ldots,a_N)$, so that the dimension of~$X$ is~$N-1$. Let~$c_d$ denote the coefficient of~$t^d$ in the regularized quantum period~$\Ghat_X(t)$ given in~\eqref{eq:Ghat_wP}. Let~$a = a_1 + \cdots + a_N$ and~$p_i = a_i/a$. Then~$c_d = 0$ unless~$d$ is divisible by~$a$, and non-zero coefficients~$c_d$ satisfy
    \[
        \log c_d \sim Ad - \frac{\dim{X}}{2} \log d + B
    \]
    as $d \to \infty$, where 
    \begin{align*}
       A &= - \sum_{i=1}^N p_i \log p_i \\
       B &= - \frac{\dim{X}}{2} \log(2 \pi) - \frac{1}{2} \sum_{i=1}^N \log p_i
    \end{align*}
\end{Thm}

Note, although it plays no role in what follows, that~$A$ is the Shannon entropy of the discrete random variable~$Z$ with distribution~$(p_1,p_2,\ldots,p_N)$, and that~$B$ is a constant plus half the total self-information of~$Z$.

\begin{Thm} \label{thm:rank_2} 
    Let~$X$ denote the toric variety of Picard rank two with weight matrix 
    \[
    \begin{pmatrix}
        a_1 & a_2 & a_3 &  \cdots & a_N \\
        b_1 & b_2 & b_3 & \cdots & b_N
    \end{pmatrix}
    \]
    so that the dimension of~$X$ is~$N-2$. Let~$a = a_1 + \cdots + a_N$, $b = b_1 + \cdots + b_N$, and $\ell = \gcd\{a,b\}$. Let~$[\mu:\nu] \in \PP^1$ be the unique root of the homogeneous polynomial 
    \[
        \prod_{i=1}^N (a_i \mu + b_i \nu)^{a_i b} -
        \prod_{i=1}^N (a_i \mu + b_i \nu)^{b_i a}
    \]     
    such that $a_i \mu + b_i \nu \geq 0$ for all~$i \in \{1,2,\ldots,N\}$, and set
    \[
        p_i = \frac{\mu a_i + \nu b_i}{\mu a + \nu b}
    \]    
    Let~$c_d$ denote the coefficient of~$t^d$ in the regularized quantum period~$\Ghat_X(t)$ given in~\eqref{eq:Ghat_rank_two}. Then non-zero coefficients~$c_d$ satisfy
    \[
        \log c_d \sim Ad - \frac{\dim{X}}{2} \log d + B
    \]
    as~$d \to \infty$, where 
    \begin{align*}
        A &= -\sum_{i=1}^N p_i \log p_i \\
        B &= - \frac{\dim{X}}{2} \log(2 \pi)\!-\!\frac{1}{2} \sum_{i=1}^N\log p_i\!-\!\frac{1}{2} \log\!\left(\sum_{i=1}^N\!\frac{(a_i b - b_i a)^2}{ \ell^2 p_i} \right) 
    \end{align*}
\end{Thm}

Theorem~\ref{thm:wps} is a straightforward application of Stirling's formula. Theorem~\ref{thm:rank_2} is more involved, and relies on a Central Limit-type theorem that generalises the De Moivre--Laplace theorem.

\begin{figure}[tbp]
  \centering
  \includegraphics[width=0.4\textwidth]{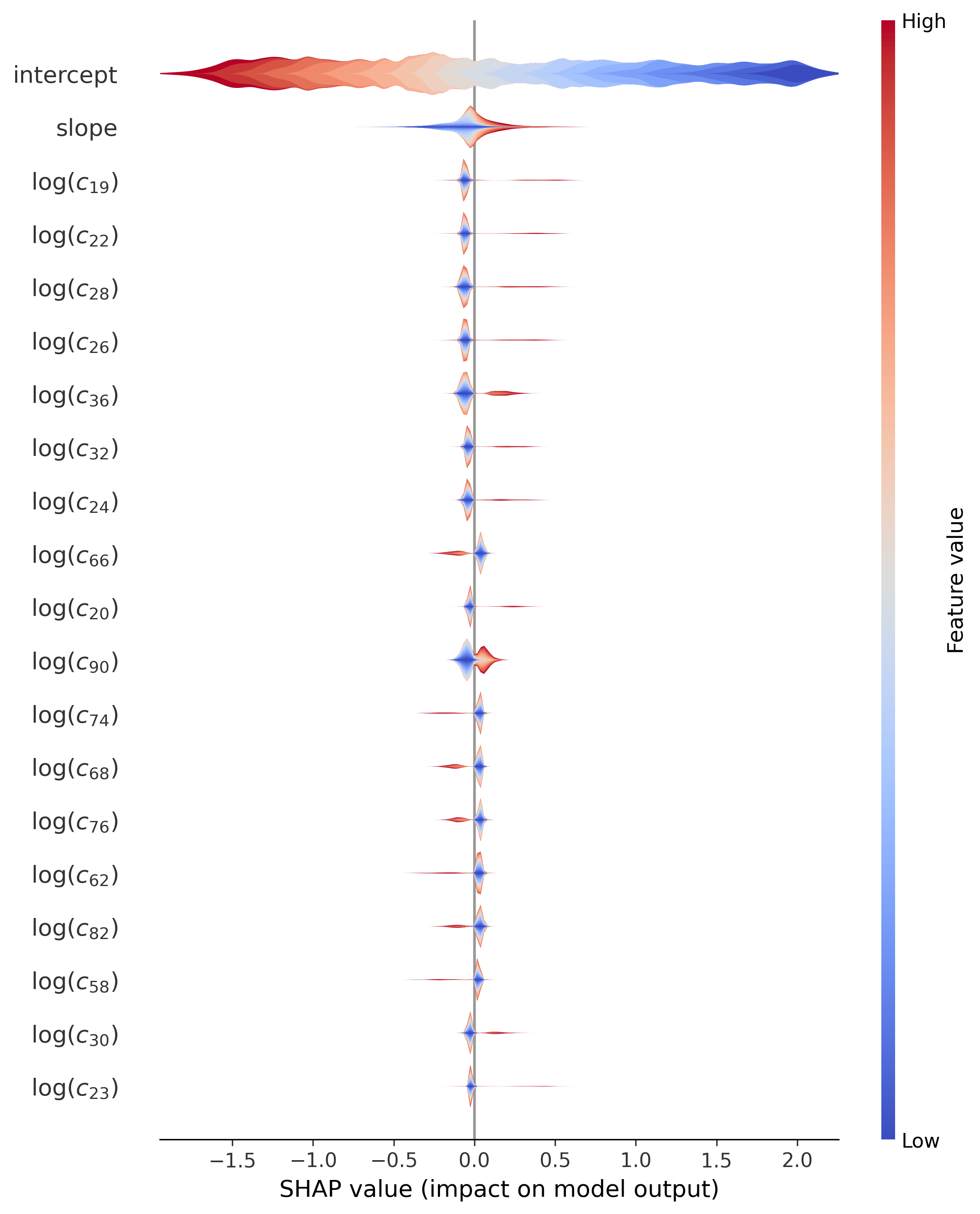}
  \caption{Model sensitivity analysis using SHAP values. The model is an MLP with three hidden layers of sizes (10,30,10) applied to toric varieties of Picard rank two with terminal quotient singularities. It is trained on the slope, $y$-intercept, and the first 100 coefficients $c_d$ as features, and predicts the dimension with 97.7\% accuracy.}
  \label{fig:SHAP_plot}
\end{figure}
%-------------------------------------------------------------------------------
\subsubsection*{Theoretical analysis}
%-------------------------------------------------------------------------------
\begin{figure*}[tbp]
  \centering
  \begin{subfigure}{0.45\linewidth}
    \includegraphics[width=\linewidth]{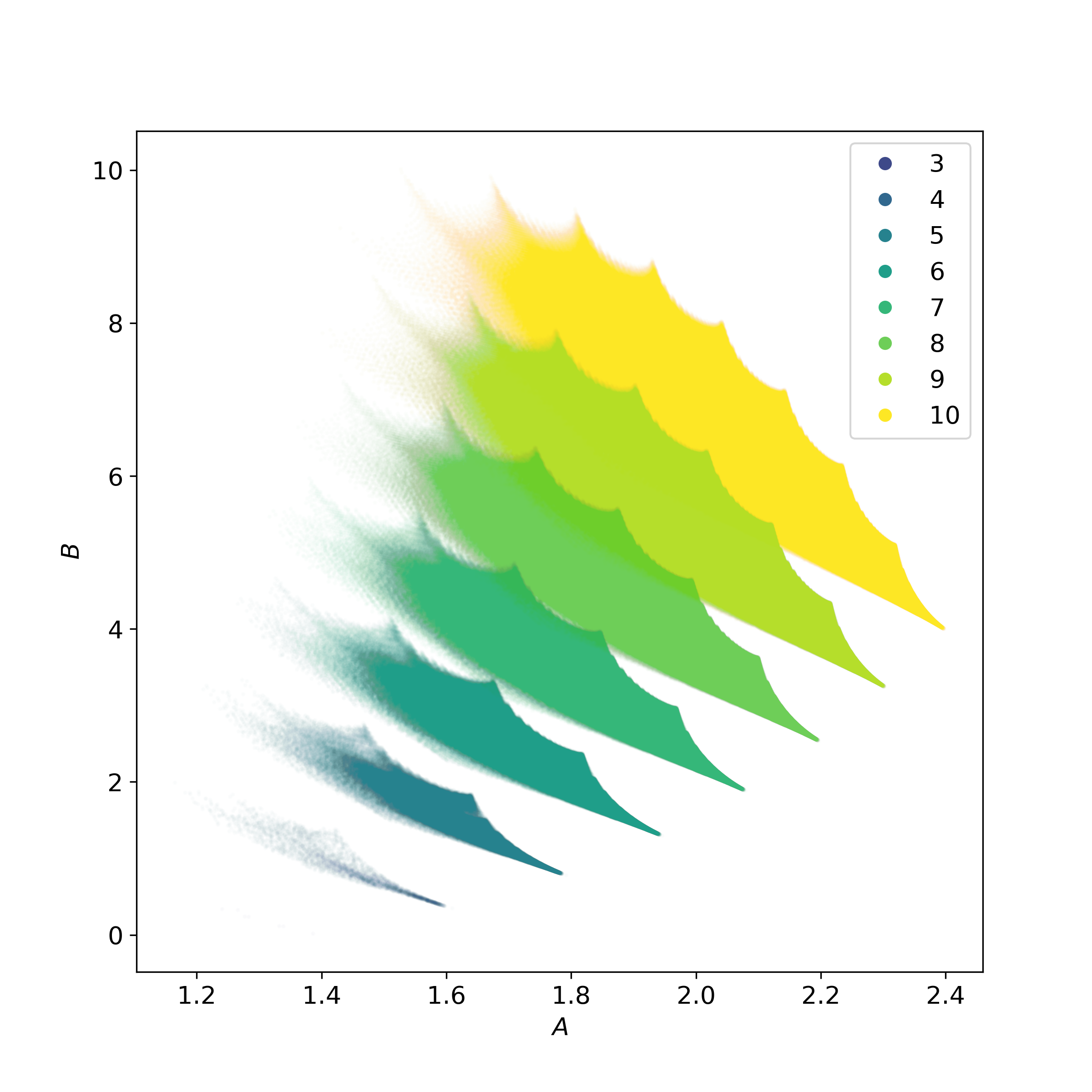}
    \caption{}
    \label{fig:hedgehogs}
  \end{subfigure}
  \qquad
  \begin{subfigure}{0.45\linewidth}
    \includegraphics[width=\linewidth]{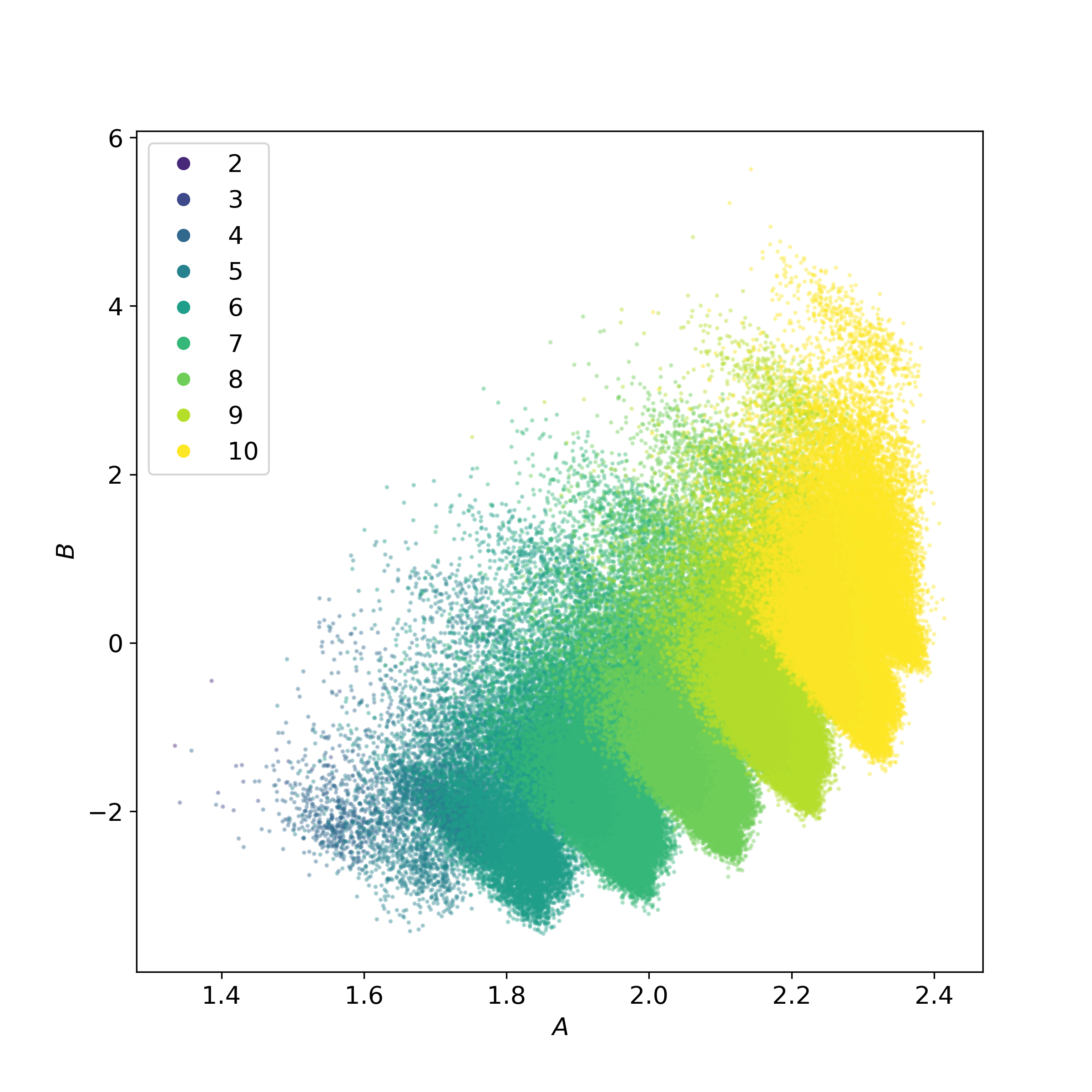}
    \caption{}
    \label{fig:hedgehogs_pr2}
  \end{subfigure}
  \caption{The values of the asymptotic coefficients~$A$ and~$B$: (a)~for all weighted projective spaces~$\PP(a_1,\ldots,a_N)$ with terminal quotient singularities and~$a_i \leq 25$ for all~$i$. The colour records the dimension of the weighted projective space. (b)~for toric varieties of Picard rank two in our dataset. The colour records the dimension of the toric variety.}
\end{figure*}

\begin{figure*}[tbp]
  \centering
  \begin{subfigure}{0.45\linewidth}
    \includegraphics[width=\linewidth]{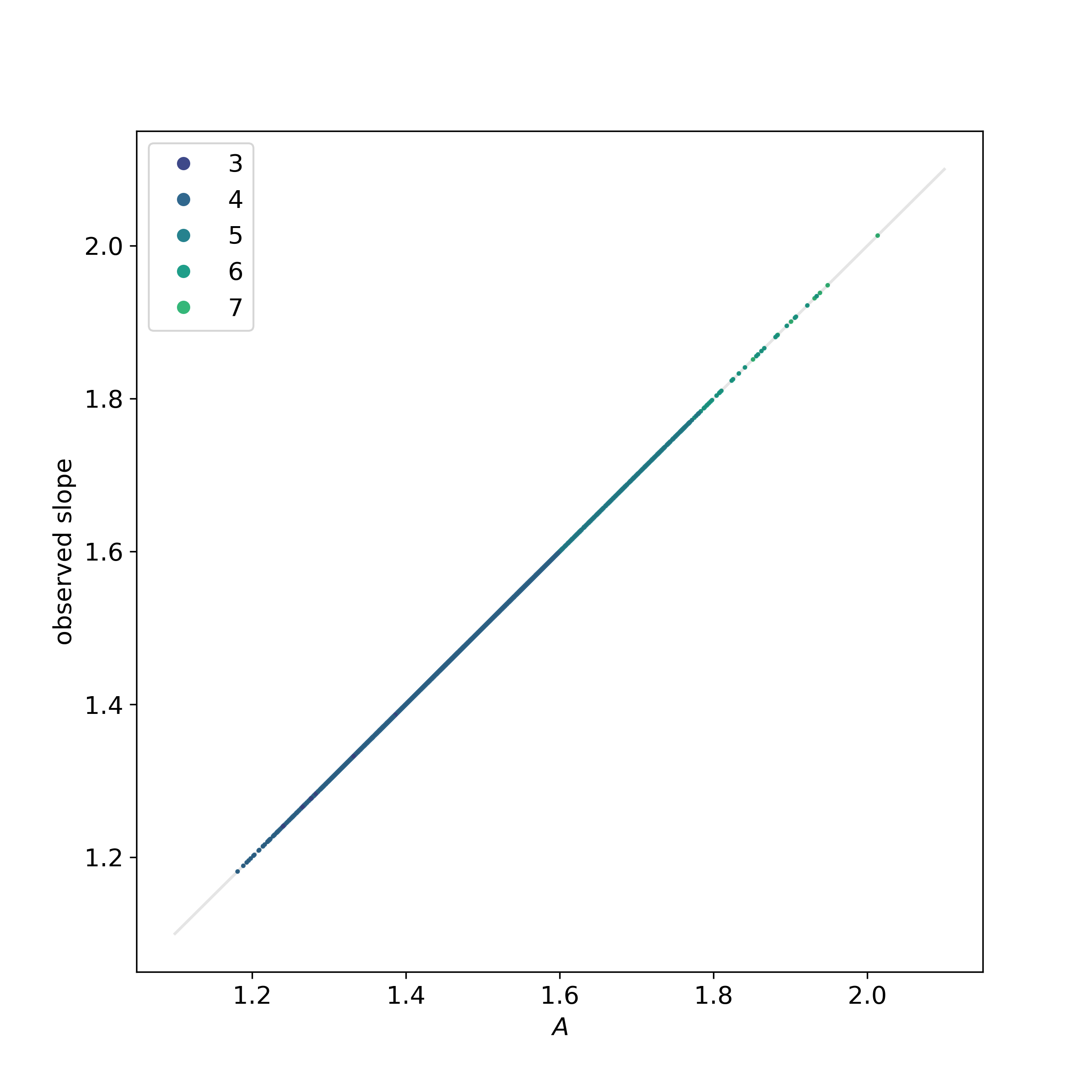}
    \caption{}
    \label{fig:slope_comparison}
  \end{subfigure}
  \qquad
  \begin{subfigure}{0.45\linewidth}
    \includegraphics[width=\linewidth]{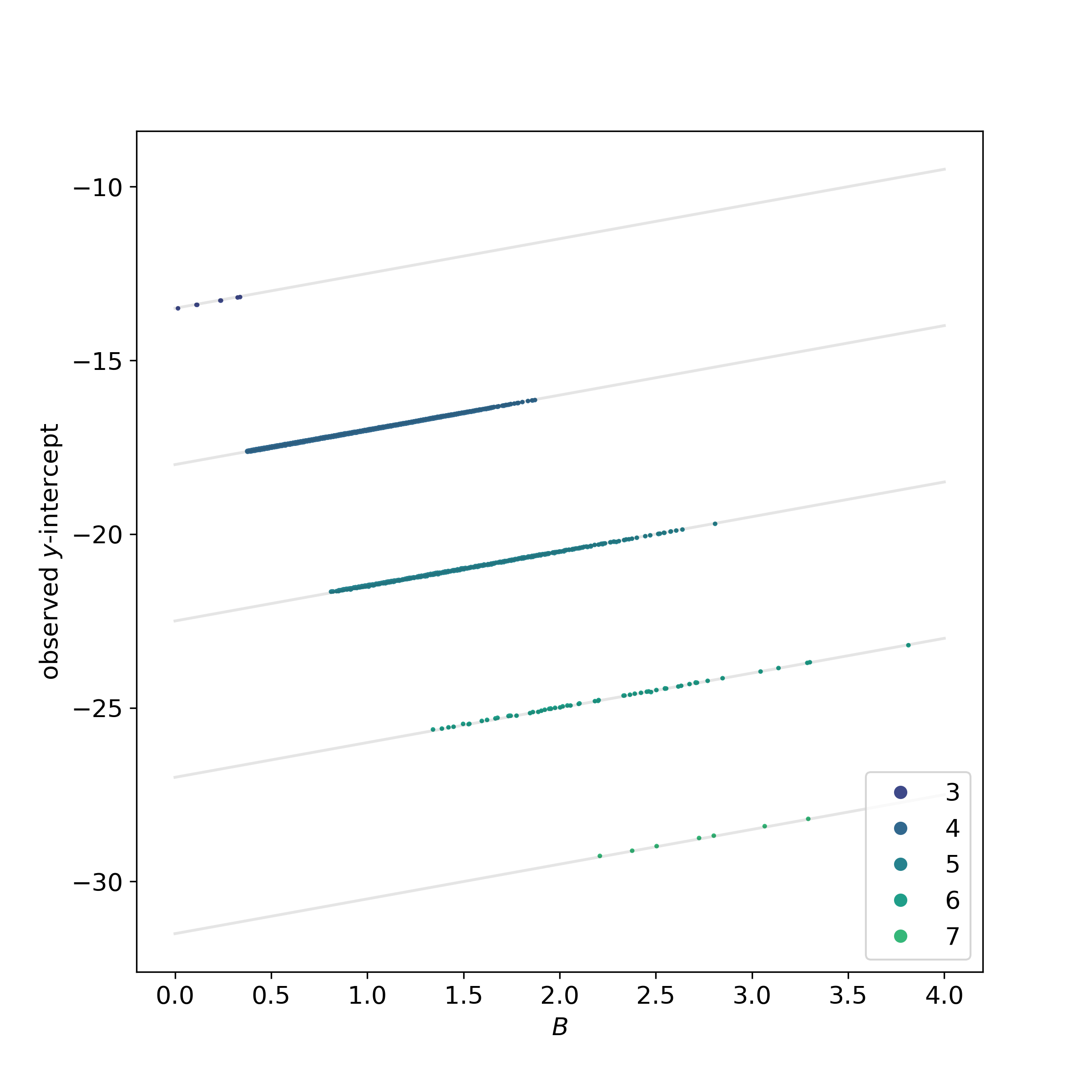}
    \caption{}
    \label{fig:intercept_comparison}
  \end{subfigure}
  \caption{For weighted projective spaces, the asymptotic coefficients $A$ and $B$ are closely related to the slope and $y$-intercept. (a)~Comparison between~$A$ and the slope from the linear model, for weighted projective spaces that occur in both Figure~\ref{fig:colour_triangle_all_pr1} and Figure~\ref{fig:hedgehogs}, coloured by dimension. The line $\text{slope} = A$ is indicated. (b)~Comparison between~$B$ and the $y$-intercept from the linear model, for weighted projective spaces that occur in both Figure~\ref{fig:colour_triangle_all_pr1} and Figure~\ref{fig:hedgehogs}, coloured by dimension. In each case the line $\text{$y$-intercept} = B - \frac{9}{2} \dim{X}$ is shown.}
  \label{fig:comparison}
\end{figure*}

The asymptotics in Theorems~\ref{thm:wps} and~\ref{thm:rank_2} imply that, for~$X$ a weighted projective space or toric variety of Picard rank two, the quantum period determines the dimension of~$X$. Let us revisit the clustering analysis from this perspective. Recall the asymptotic expression~$\log c_d\sim A d - \frac{\dim{X}}{2} \log d + B$ and the formulae for~$A$ and~$B$ from Theorem~\ref{thm:wps}. Figure~\ref{fig:hedgehogs} shows the values of~$A$ and~$B$ for a sample of weighted projective spaces, coloured by dimension. Note the clusters, which overlap. Broadly speaking, the values of~$B$ increase as the dimension of the weighted projective space increases, whereas in Figure~\ref{fig:colour_triangle_all_pr1} the $y$\nobreakdash-intercepts decrease as the dimension increases. This reflects the fact that we fitted a linear model to~$\log c_d$, omitting the subleading~$\log d$ term in the asymptotics. As Figure~\ref{fig:comparison} shows, the linear model assigns the omitted term to the $y$-intercept rather than the slope. The slope of the linear model is approximately equal to $A$. The $y$-intercept, however, differs from $B$ by a dimension-dependent factor.
The omitted~$\log$ term does not vary too much over the range of degrees ($d<100\,000$) that we considered, and has the effect of reducing the observed $y$\nobreakdash-intercept from~$B$ to approximately~$B - \frac{9}{2} \dim X$, distorting the clusters slightly and translating them downwards by a dimension\nobreakdash-dependent factor. This separates the clusters. We expect that the same mechanism applies in Picard rank two as well: see Figure~\ref{fig:hedgehogs_pr2}.

We can show that each cluster in Figure~\ref{fig:hedgehogs} is linearly bounded using constrained optimisation techniques. Consider for example the cluster for weighted projective spaces of dimension five, as in Figure~\ref{fig:bounded_hedgehog}.

\begin{Prop}\label{pro:bound_5}
  Let~$X$ be the five\nobreakdash-dimensional weighted projective space~$\PP(a_1,\ldots,a_6)$, and let~$A$,~$B$ be as in Theorem~\ref{thm:wps}. Then~$B+\frac{5}{2}A \geq \frac{41}{8}$. If in addition~$a_i \leq 25$ for all~$i$ then~$B + 5 A \leq \frac{41}{40}$.
\end{Prop}

\noindent Fix a suitable $\theta \geq 0$ and consider 
\[
  B + \theta A = - \frac{\dim{X}}{2} \log(2 \pi) - \frac{1}{2} \sum_{i=1}^N \log p_i - \theta \sum_{i=1}^N p_i \log p_i
\]
with~$\dim{X}=N-1=5$. Solving
\[
\begin{aligned}
    \min(B + \theta A) && \text{subject to} &&& p_1 + \cdots + p_6 = 1 \\
    && &&& p_1,\ldots,p_6 \geq 0
\end{aligned}
\]
on the five\nobreakdash-simplex gives a linear lower bound for the cluster. This bound does not use terminality: it applies to any weighted projective space of dimension five. The expression $B + \theta A$ is unbounded above on the five\nobreakdash-simplex (because $B$ is) so we cannot obtain an upper bound this way. Instead, consider  
\[
\begin{aligned}
    \max(B + \theta A)&& \text{subject to} &&& p_1 + \cdots + p_6 = 1 \\
    && &&& \epsilon \leq p_1 \leq p_2 \leq \cdots \leq p_6
\end{aligned}
\]
for an appropriate small positive $\epsilon$, which we can take to be $1/a$ where $a$ is the maximum sum of the weights. For Figure~\ref{fig:bounded_hedgehog}, for example, we can take $a = 124$, and in general such an $a$ exists because there are only finitely many terminal weighted projective spaces. This gives a linear upper bound for the cluster. 

\begin{figure}[tbp]
    \centering
    \includegraphics[width=0.4\textwidth]{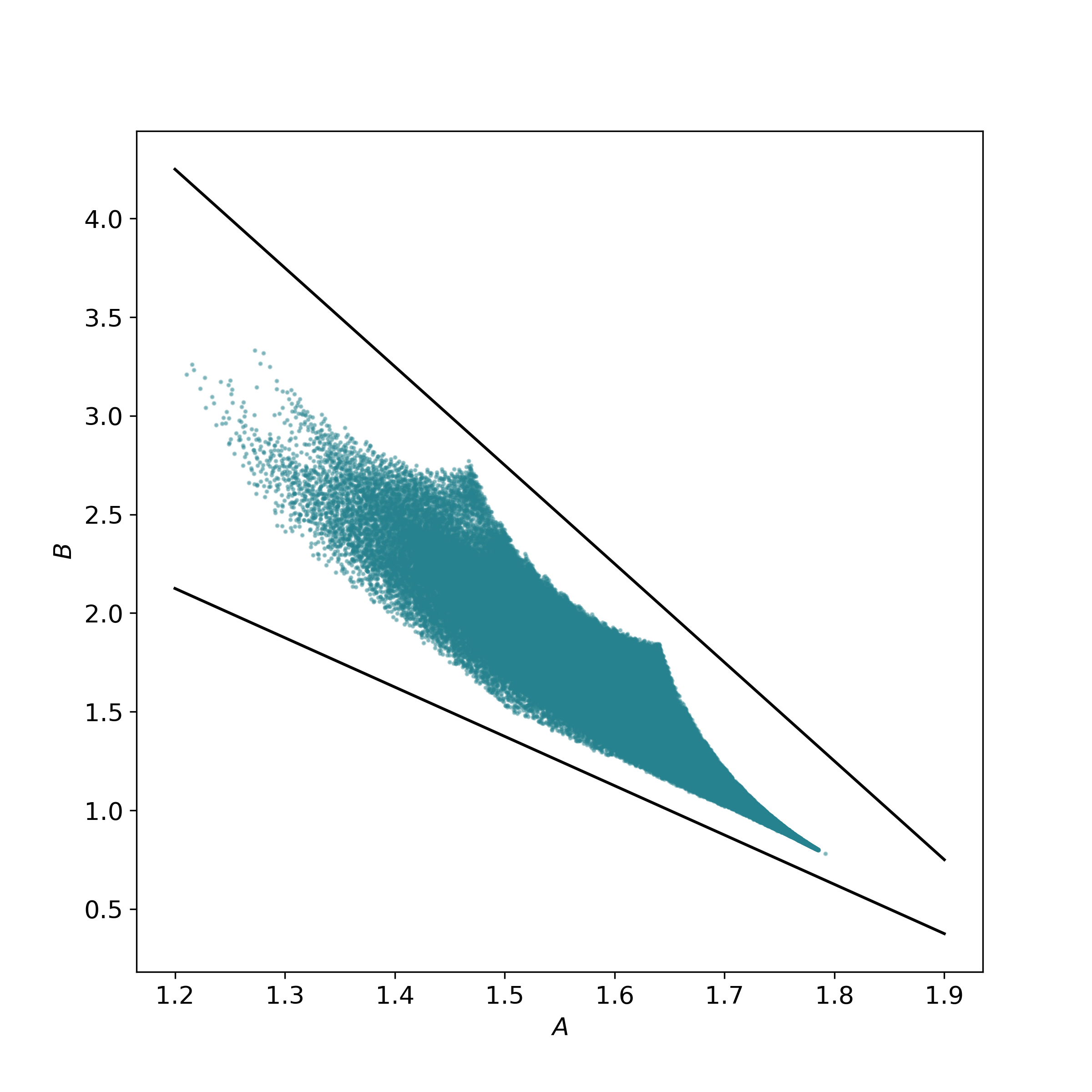}
    \caption{Linear bounds for the cluster of five-dimensional weighted projective spaces in Figure~\ref{fig:hedgehogs}. The bounds are given by Proposition~\ref{pro:bound_5}.}
    \label{fig:bounded_hedgehog}
\end{figure}

The same methods yield linear bounds on each of the clusters in Figure~\ref{fig:hedgehogs}. As the Figure shows however, the clusters are not linearly separable.
%-------------------------------------------------------------------------------
\subsection*{Discussion}
%-------------------------------------------------------------------------------
We developed machine learning models that predict, with high accuracy, the dimension of a Fano variety from its regularized quantum period. These models apply to weighted projective spaces and toric varieties of Picard rank two with terminal quotient singularities. We then established rigorous asymptotics for the regularized quantum period of these Fano varieties. The form of the asymptotics implies that, in these cases, the regularized quantum period of a Fano variety~$X$ determines the dimension of~$X$. The asymptotics also give a theoretical underpinning for the success of the machine learning models. 

Perversely, because the series involved converge extremely slowly, reading the dimension of a Fano variety directly from the asymptotics of the regularized quantum period is not practical. For the same reason, enhancing the feature set of our machine learning models by including a~$\log d$ term in the linear regression results in less accurate predictions. So although the asymptotics in Theorems~\ref{thm:wps} and~\ref{thm:rank_2} determine the dimension in theory, in practice the most effective way to determine the dimension of an unknown Fano variety from its quantum period is to apply a machine learning model. 

The insights gained from machine learning were the key to our formulation of the rigorous results in Theorems~\ref{thm:wps} and~\ref{thm:rank_2}. Indeed, it might be hard to discover these results without a machine learning approach. It is notable that the techniques in the proof of Theorem~\ref{thm:rank_2} -- the identification of generating functions for Gromov--Witten invariants of toric varieties with certain hypergeometric functions -- have been known since the late 1990s and have been studied by many experts in hypergeometric functions since then. For us, the essential step in the discovery of the results was 
the feature extraction that we performed as part of our ML pipeline.

This work demonstrates that machine learning can uncover previously unknown structure in complex mathematical data, and is a powerful tool for developing rigorous mathematical results; cf.~\cite{DaviesEtAl2021}. It also provides evidence for a fundamental conjecture in the Fano classification program~\cite{CoatesCortiGalkinGolyshevKasprzyk2013}: that the regularized quantum period of a Fano variety determines that variety. 
%-------------------------------------------------------------------------------
\section{Methods}\label{sec:methods}
%-------------------------------------------------------------------------------
In this section we prove Theorem~\ref{thm:wps} and Theorem~\ref{thm:rank_2}. The following result implies Theorem~\ref{thm:wps}.

\begin{Thm} \label{thm:wps_methods}
    Let $X$ denote the weighted projective space $\PP(a_1,\ldots,a_N)$, so that the dimension of $X$ is $N-1$. Let $c_d$ denote the coefficient of $t^d$ in the regularized quantum period $\Ghat_X(t)$ given in~\eqref{eq:Ghat_wP}. Let $a = a_1 + \ldots + a_N$. Then $c_d = 0$ unless $d$ is divisible by $a$, and
    \[
        \log c_{ka} \sim ka \left[\log a - \frac{1}{a} \sum_{i=1}^N a_i \log a_i \right] -\frac{\dim{X}}{2} \log(ka) + \frac{1+\dim{X}}{2} \log a - \frac{\dim{X}}{2} \log(2 \pi) - \frac{1}{2} \sum_{i=1}^N \log a_i
    \] 
    That is, non-zero coefficients $c_d$ satisfy
    \[
        \log c_d \sim Ad - \frac{\dim{X}}{2} \log d + B
    \]
    as $d \to \infty$, where
    \[
    \begin{aligned}
      A = - \sum_{i=1}^N p_i \log p_i &&
      B = - \frac{\dim{X}}{2} \log(2 \pi) - \frac{1}{2} \sum_{i=1}^N \log p_i
    \end{aligned}
    \]
    and $p_i = a_i/a$.
\end{Thm}

\begin{proof}
    Combine Stirling's formula
    \[
        n! \sim \sqrt{2 \pi n } \left(\frac{n}{e} \right)^n
    \]
    with the closed formula~\eqref{eq:Ghat_wP} for $c_{ka}$.
\end{proof}
%-------------------------------------------------------------------------------
\subsubsection*{Toric varieties of Picard rank 2}
%-------------------------------------------------------------------------------
Consider a toric variety $X$ of Picard rank two and dimension~$N-2$ with weight matrix
\[
  \begin{pmatrix}
    a_1 & a_2 & a_3 & \cdots & a_N \\
    b_1 & b_2 & b_3 & \cdots & b_N
  \end{pmatrix}
\]
as in~\eqref{eq:weights}. Let us move to more invariant notation, writing $\alpha_i$ for the linear form on $\RR^2$ defined by the transpose of the $i$th column of the weight matrix, and $\alpha = \alpha_1 + \cdots + \alpha_N$. Equation~\ref{eq:Ghat_rank_two} becomes 
\[
  \Ghat_X(t) = \sum_{k \in \ZZ^2 \cap C} \frac{(\alpha\cdot k)!}{\prod_{i=1}^N (\alpha_i\cdot k)!} t^{\alpha \cdot k}
\]
where $C$ is the cone $C = \{ x \in \RR^2\mid\text{$\alpha_i \cdot x \geq 0$ for $i =1,2,\ldots,N$} \}$. As we will see, for $d \gg 0$ the coefficients
\[
\begin{aligned}
  \frac{(\alpha\cdot k)!}{\prod_{i=1}^N (\alpha_i\cdot k)!}
  && \text{where $k \in \ZZ^2 \cap C$ and $\alpha \cdot k = d$}
\end{aligned}
\]
are approximated by a rescaled Gaussian. We begin by finding the mean of that Gaussian, that is, by minimising
\[
\begin{aligned}
  \prod_{i=1}^N (\alpha_i\cdot k)!
  && \text{where $k \in \ZZ^2 \cap C$ and $\alpha \cdot k = d$.}
\end{aligned}
\]
For $k$ in the strict interior of $C$ with $\alpha \cdot k = d$, we have that
\[
  (\alpha_i \cdot k)! \sim \left( \frac{\alpha_i \cdot k}{e}\right)^{\alpha_i \cdot k}
\]
as $d \to \infty$.

\begin{Prop}
  \label{pro:minimum}
  The constrained optimisation problem
  \[
  \begin{aligned}
    \min & \prod_{i=1}^N (\alpha_i\cdot x)^{\alpha_i \cdot x} &
    \text{subject to\ } & 
    \begin{cases}
      x \in C \\
      \alpha \cdot x = d
    \end{cases}
  \end{aligned}
  \]  
  has a unique solution $x = x^*$. Furthermore, setting $p_i = (\alpha_i \cdot x^*)/(\alpha \cdot x^*)$ we have that the monomial
  \[    
    \prod_{i=1}^N p_i^{\alpha_i \cdot k}
  \]
  depends on $k \in \ZZ^2$ only via $\alpha \cdot k$.
\end{Prop}

\begin{proof}
  Taking logarithms gives the equivalent problem
  \begin{align}
    \label{eq:min_problem}
    \min & \sum_{i=1}^N (\alpha_i \cdot x) \log (\alpha_i \cdot x) &
    \text{subject to\ } & 
    \begin{cases}
      x \in C \\
      \alpha \cdot x = d
    \end{cases}
  \end{align}  
  The objective function $\sum_{i=1}^N (\alpha_i \cdot x) \log (\alpha_i \cdot x)$ here is the pullback to $\RR^2$ of the function
  \[
    f(x_1,\ldots,x_N) = \sum_{i=1}^N x_i \log x_i
  \]
  along the linear embedding $\varphi : \RR^2 \to \RR^N$ given by $(\alpha_1,\ldots,\alpha_N)$. Note that $C$ is the preimage under $\varphi$ of the positive orthant~$\RR^N_+$, so we need to minimise $f$ on the intersection of the simplex $x_1 + \cdots + x_N = d$, $(x_1,\ldots,x_N) \in \RR^N_+$ with the image of $\varphi$. The function $f$ is convex and decreases as we move away from the boundary of the simplex, so the minimisation problem in~\eqref{eq:min_problem} has a unique solution $x^*$ and this lies in the strict interior of $C$. We can therefore find the minimum $x^*$ using the method of Lagrange multipliers, by solving
  \begin{equation}
    \label{eq:Lagrange_multiplier}
    \sum_{i=1}^N \alpha_i \log (\alpha_i \cdot x) + \alpha
    =
    \lambda \alpha
  \end{equation}
  for $\lambda \in \RR$ and $x$ in the interior of $C$ with $\alpha \cdot x = d$. Thus
  \[
    \sum_{i=1}^N \alpha_i \log (\alpha_i \cdot x^*)
    =
    (\lambda-1) \alpha
  \]
  and, evaluating on $k \in \ZZ^2$ and exponentiating, we see that
  \[
    \prod_{i=1}^N (\alpha_i \cdot x^*)^{\alpha_i \cdot k}
  \]
  depends only on $\alpha \cdot k$. The result follows.
\end{proof}

Given a solution $x^*$ to~\eqref{eq:Lagrange_multiplier}, any positive scalar multiple of $x^*$ also satisfies~\eqref{eq:Lagrange_multiplier}, with a different value of $\lambda$ and a different value of $d$. Thus the solutions $x^*$, as $d$ varies, lie on a half-line through the origin. The direction vector $[\mu:\nu] \in \PP^1$ of this half-line is the unique solution to the system
\begin{equation}
  \label{eq:polynomial_minimum}
  \begin{aligned}
    \prod_{i=1}^N (a_i \mu  +b_i \nu)^{a_i b} &=
    \prod_{i=1}^N (a_i \mu + b_i \nu)^{b_i a}
    \\
    \begin{pmatrix} \mu \\ \nu \end{pmatrix} &\in C
  \end{aligned}
\end{equation}
Note that the first equation here is homogeneous in $\mu$ and $\nu$; it is equivalent to~\eqref{eq:Lagrange_multiplier}, by exponentiating and then eliminating~$\lambda$. Any two solutions $x^*$, for different values of $d$, differ by rescaling, and the quantities $p_i$ in Proposition~\ref{pro:minimum} are invariant under this rescaling. They also satisfy $p_1 + \cdots + p_N = 1$.

We use the following result, known in the literature as the ``Local Theorem''~\cite{GnedenkoUshakov2018}, to approximate multinomial coefficients.

\begin{LocalThm}
    For $p_1, \dots, p_n \in [0,1]$ such that $p_1 + \cdots + p_n = 1$, the ratio 
    \[
        d^{\frac{n-1}{2}} \binom{d}{k_1 \cdots k_n} \prod_{i=1}^n p_i^{k_i} : \frac{\exp(-\frac{1}{2} \sum_{i=1}^n q_i x_i^2)}{(2 \pi)^{\frac{n-1}{2}} \sqrt{p_1 \cdots p_n}} \rightarrow 1
    \]
    as $d \rightarrow \infty$, uniformly in all $k_i$'s, where 
    \begin{align*}
        q_i = 1- p_i && 
        x_i = \frac{k_i - d p_i}{\sqrt{dp_iq_i}} 
    \end{align*}
    and the $x_i$ lie in bounded intervals. 
\end{LocalThm}

Let $B_r$ denote the ball of radius $r$ about $x^* \in \RR^2$. Fix $R > 0$. We apply the Local Theorem with $k_i = \alpha_i \cdot k$ and $p_i = (\alpha_i \cdot x^*)/(\alpha \cdot x^*)$, where $k \in \ZZ^2 \cap C$ satisfies $\alpha \cdot k = d$ and $k \in B_{R \sqrt{d}}$. Since
\[
  x_i = \frac{\alpha_i \cdot (k - x^*)}{\sqrt{d p_i q_i}}
\]
the assumption that $k \in B_{R \sqrt{d}}$ ensures that the $x_i$ remain bounded as $d \to \infty$.
Note that, by Proposition~\ref{pro:minimum}, the monomial $\prod_{i=1}^N p_i^{k_i}$ depends on $k$ only via $\alpha \cdot k$, and hence here is independent of $k$:
\[
  \prod_{i=1}^N p_i^{k_i} = \prod_{i=1}^N p_i^{\alpha_i \cdot x^*} = \prod_{i=1}^N p_i^{d p_i}
\]
Furthermore 
\[
  \sum_{i=1}^N q_i x_i^2 = \frac{(k-x^*)^T A \, (k - x^*)}{d}
\]
where $A$ is the positive-definite $2 \times 2$ matrix given by
\[
  A = \sum_{i=1}^N \frac{1}{p_i} \alpha_i^T \alpha_i
\]
Thus as $d \to \infty$, the ratio
\begin{equation} \label{eq:limit_of_multinomial}
  \frac{(\alpha\cdot k)!}{\prod_{i=1}^N (\alpha_i\cdot k)!}
  :
  \frac{\exp\big({-\frac{1}{2d}} (k - x^*)^T A \, (k-x^*)\big)}{(2 \pi d)^{\frac{N-1}{2}} \prod_{i=1}^N p_i^{d p_i + \frac{1}{2}}}
  \to 1
\end{equation} 
for all $k \in \ZZ^2 \cap C \cap B_{R\sqrt{d}}$ such that $\alpha \cdot k = d$.

\bigskip
\noindent\textbf{Theorem~\ref{thm:rank_2}.}
\emph{
  Let $X$ be a toric variety of Picard rank two and dimension~$N-2$ with weight matrix
  \[
    \begin{pmatrix}
      a_1 & a_2 & a_3 & \cdots & a_N \\
      b_1 & b_2 & b_3 & \cdots & b_N
    \end{pmatrix}
  \]
  Let $a = a_1 + \cdots + a_N$ and $b = b_1 + \cdots + b_N$, let $\ell = \gcd(a,b)$, and let $[\mu:\nu] \in \PP^1$ be the unique solution to~\eqref{eq:polynomial_minimum}. Let $c_d$ denote the coefficient of $t^d$ in the regularized quantum period $\Ghat_X(t)$.  Then non-zero coefficients $c_d$ satisfy
  \[
      \log c_d \sim Ad - \frac{\dim{X}}{2} \log d + B
  \]
  as $d \to \infty$, where 
  \begin{align*}
    A &= -\sum_{i=1}^N p_i \log p_i \\
    B &= - \frac{\dim{X}}{2} \log(2 \pi) - \frac{1}{2} \sum_{i=1}^N \log p_i - \frac{1}{2} \log \left( \sum_{i=1}^N \frac{(a_i b - b_i a)^2}{ \ell^2 p_i} \right) 
  \end{align*}
  and $p_i = \displaystyle \frac{\mu a_i + \nu b_i}{\mu a + \nu b}$.
}

\begin{proof}
  We need to estimate
  \[
    c_d = \sum_{\substack{k \in \ZZ^2 \cap C \\ \text{with $\alpha \cdot k = d$}}} \frac{(\alpha\cdot k)!}{\prod_{i=1}^N (\alpha_i\cdot k)!} 
  \]
  Consider first the summands with $k \in \ZZ^2 \cap C$ such that $\alpha \cdot k = d$ and $k \not \in B_{R \sqrt{d}}$. For $d$ sufficiently large, each such summand is bounded by $c d^{-\frac{1+\dim{X}}{2}}$ for some constant $c$ -- see~\eqref{eq:limit_of_multinomial}. Since the number of such summands grows linearly with $d$, in the limit $d \to \infty$ the contribution to $c_d$ from $k \not \in B_{R \sqrt{d}}$ vanishes.

  As $d \to \infty$, therefore
  \[
    c_d \sim
    \frac{1}{(2 \pi d)^{\frac{N-1}{2}} \prod_{i=1}^N p_i^{d p_i + \frac{1}{2}}}
    \sum_{\substack{k \in \ZZ^2 \cap C \cap B_{R \sqrt{d}}\\ \text{with $\alpha \cdot k = d$}}} \exp\left(-\frac{(k - x^*)^T A \, (k-x^*)}{2d}\right)
  \]
  Writing $y_k = (k - x^*)/\sqrt{d}$, considering the sum here as a Riemann sum, and letting $R \to \infty$, we see that
  \[
  \begin{aligned}
     c_d \sim
    \frac{1}{(2 \pi d)^{\frac{N-1}{2}} \prod_{i=1}^N p_i^{d p_i+ \frac{1}{2}}}
    \sqrt{d}
    \int_{L_\alpha} \exp \big( {-\textstyle \frac{1}{2}}y^T A y \big) \, dy
  \end{aligned}
  \]
  where $L_\alpha$ is the line through the origin given by $\ker \alpha$ and $dy$ is the measure on $L_\alpha$ given by the integer lattice $\ZZ^2 \cap L_\alpha \subset L_\alpha$.

  To evaluate the integral, let 
  \[
  \begin{aligned}
    \alpha^\perp = \frac{1}{\ell}\begin{pmatrix} b\\-a \end{pmatrix}
    && \text{where $\ell = \gcd\{a,b\}$}
  \end{aligned}
  \]
and observe that the pullback of $dy$ along the map $\RR \to L_\alpha$ given by $t \mapsto t \alpha^\perp$ is the standard measure on $\RR$. Thus
  \[
  \int_{L_\alpha} \exp \big( {-\textstyle \frac{1}{2}}y^T A y \big) \, dy = 
  \int_{-\infty}^\infty \exp \big( {-\textstyle \frac{1}{2}} \theta x^2 \big) \, dx =
  \sqrt{\frac{2 \pi}{\theta}}
  \]
  where $\theta = \sum_{i=1}^N \frac{1}{\ell^2 p_i} ( \alpha_i \cdot \alpha^\perp)^2$, and
  \[
  c_d \sim \frac{1}{(2 \pi d)^{\frac{\dim{X}}{2}} \prod_{i=1}^N p_i^{d p_i + \frac{1}{2}} \sqrt{\theta}}
  \]
  Taking logarithms gives the result.
\end{proof}
%-------------------------------------------------------------------------------
\appendix
%-------------------------------------------------------------------------------
\section{Supplementary Notes}\label{sec:supp_notes}
%-------------------------------------------------------------------------------
We begin with an introduction to weighted projective spaces and toric varieties, aimed at non-specialists.
%-------------------------------------------------------------------------------
\subsubsection*{Projective spaces and weighted projective spaces}
%-------------------------------------------------------------------------------
The fundamental example of a Fano variety is two\nobreakdash-dim\-ensional projective space~$\PP^2$. This is a quotient of~$\CC^3 \setminus \{0\}$ by the group~$\Cstar$, where the action of~$\lambda \in \Cstar$ identifies the points~$(x,y,z)$ and~$(\lambda x, \lambda y, \lambda z)$ in~$\CC^3\setminus \{0\}$. The variety~$\PP^2$ is smooth: we can see this by covering it with three open sets~$U_x$, $U_y$, $U_z$ that are each isomorphic to the plane~$\CC^2$:
\[
    \begin{aligned}
        U_x &= \{(1,Y,Z)\} & \text{given by rescaling $x$ to 1} \\
        U_y &= \{(X,1,Z)\} & \text{given by rescaling $y$ to 1} \\
        U_z &= \{(X,Y,1)\} & \text{given by rescaling $z$ to 1}
    \end{aligned}
\]
Here, for example, in the case~$U_x$ we take~$x \ne 0$ and set~$Y = y/x$, $Z = z/x$.

Although the projective space~$\PP^2$ is smooth, there are closely related Fano varieties called weighted projective spaces~\cite{Dolgachev1982,IanoFletcher2000} that have singularities. For example, consider the weighted projective plane~$\PP(1,2,3)$: this is the quotient of~$\CC^3 \setminus \{0\}$ by~$\Cstar$, where the action of~$\lambda \in \Cstar$ identifies the points~$(x,y,z)$ and~$(\lambda x, \lambda^2 y, \lambda^3 z)$. Let us write
\[ 
\mu_n = \{e^{2 \pi k \tti /n}\mid k \in \ZZ\}
\]
for the group of~$n$th roots of unity. The variety~$\PP(1,2,3)$ is once again covered by open sets
\[
    \begin{aligned}
        U_x &= \{(1,Y,Z)\} & \text{given by rescaling $x$ to 1} \\
        U_y &= \{(X,1,Z)\} & \text{given by rescaling $y$ to 1} \\
        U_z &= \{(X,Y,1)\} & \text{given by rescaling $z$ to 1}
    \end{aligned}
\]
but this time we have~$U_x \cong \CC^2$, $U_y \cong \CC^2/\mu_2$, and $U_z = \CC^2/\mu_3$. This is because, for example, when we choose~$\lambda \in \Cstar$ to rescale~$(x,y,z)$ with~$z \ne 0$ to~$(X,Y,1)$, there are three possible choices for~$\lambda$ and they differ by the action of~$\mu_3$. In particular this lets us see that~$\PP(1,2,3)$ is singular. For example, functions on the chart~$U_y \cong \CC^2/\mu_2$ are polynomials in~$X$ and~$Z$ that are invariant under~$X \mapsto -X$, $Z \mapsto -Z$, or in other words 
\[
\begin{aligned}
    U_y &= \Spec \CC[X^2, XZ, Z^2]\\
        &= \Spec \CC[a,b,c]/(ac-b^2)
\end{aligned}
\]
Thus the chart~$U_y$ is the solution set for the equation~$ac-b^2=0$, as pictured in Figure~\ref{fig:A1}. Similarly, the chart~$U_z \cong \CC^2/\mu_3$ can be written as 
\[
\begin{aligned}
    U_z &= \Spec \CC[X^3, XY, Y^3]\\
        &= \Spec \CC[a,b,c]/(ac-b^3)
\end{aligned}
\]
\noindent and is the solution set to the equation~$ac-b^3=0$, as pictured in Figure~\ref{fig:A2}. The variety~$\PP(1,2,3)$ has singular points at~$(0,1,0) \in U_y$ and~$(0,0,1) \in U_z$, and away from these points it is smooth.

\begin{figure}[tbp]
  \centering
  \begin{subfigure}{.2\textwidth}
    \includegraphics[width=\linewidth]{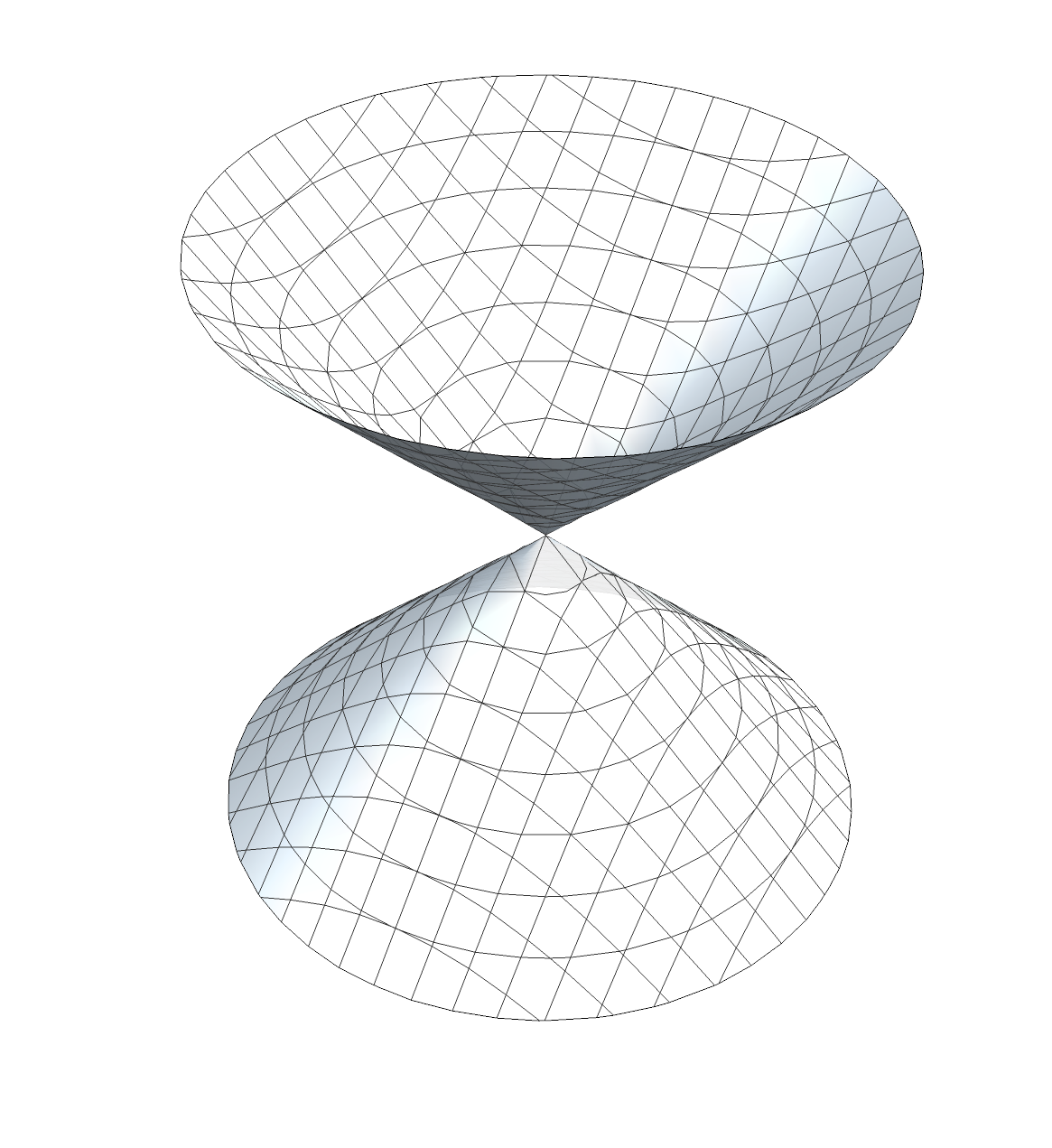}
    \caption{}\label{fig:A1}
  \end{subfigure}
  \qquad
  \begin{subfigure}{.2\textwidth}
    \includegraphics[width=\linewidth]{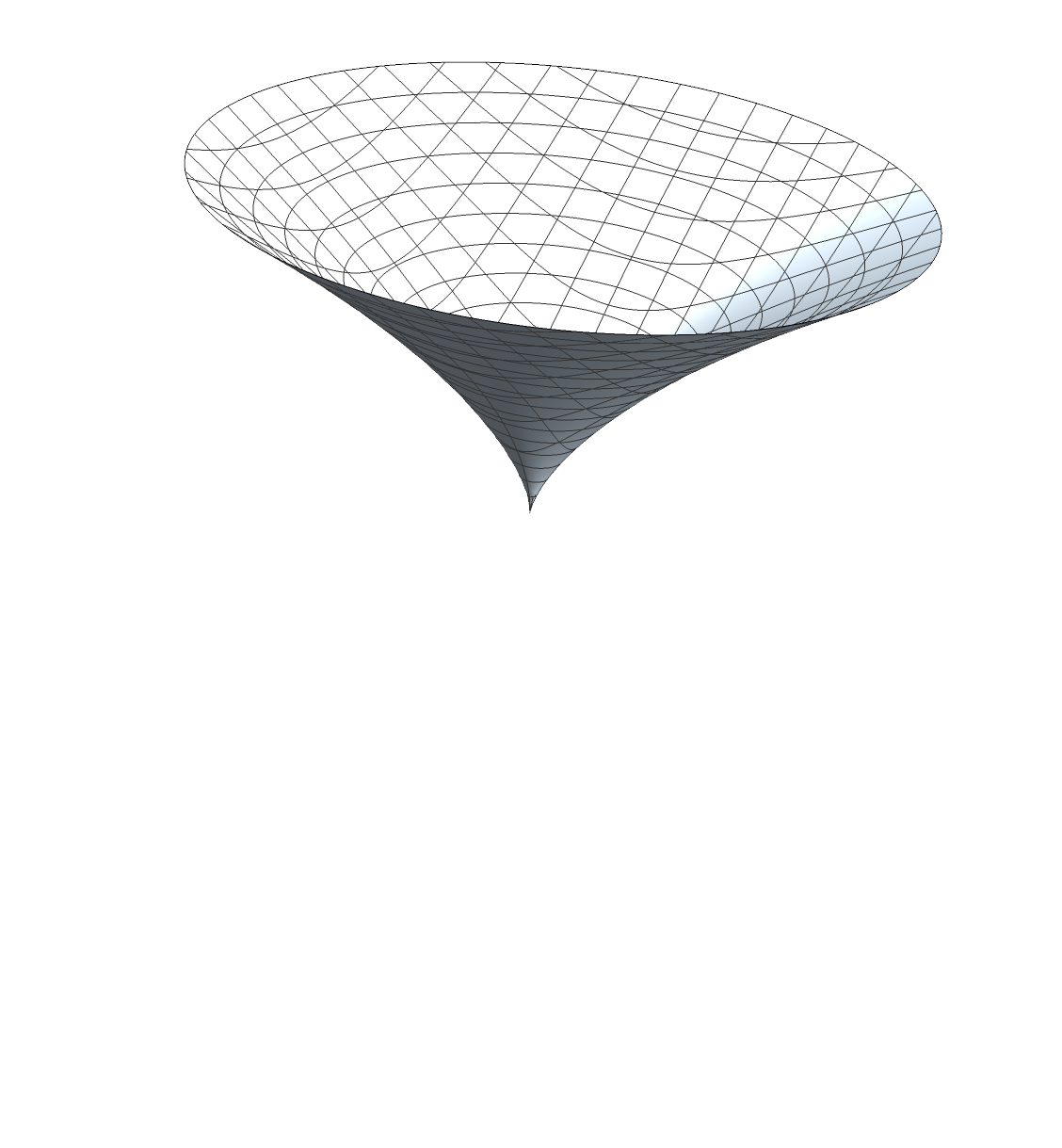}
    \caption{}\label{fig:A2}
  \end{subfigure}
  \caption{Singular charts on the weighted projective space $\PP(1,2,3)$: (a) the real-valued points in the chart $U_y$. (b) the real-valued points in the chart $U_z$.}
  \label{fig:wP_charts}
\end{figure}

There are weighted projective spaces of any dimension. Let~$a_1,a_2,\ldots,a_N$ be positive integers such that any subset of size~$N-1$ has no common factor, and consider
\[
\PP(a_1,a_2,\ldots,a_N) = (\CC^N \setminus \{0\})/\Cstar
\]
where the action of~$\lambda \in \Cstar$ identifies the points
\[
    (x_1,x_2,\ldots,x_N) \qquad \text{and} \qquad (\lambda^{a_1} x_1, \lambda^{a_2} x_2, \ldots, \lambda^{a_N} x_N)
\]
in~$\CC^N \setminus \{0\}$. The quotient~$\PP(a_1,a_2,\ldots,a_N)$ is an algebraic variety of dimension~$N-1$. A general point of~$\PP(a_1,a_2,\ldots,a_N)$ is smooth, but there can be singular points. Indeed,~$\PP(a_1,a_2,\ldots,a_N)$ is covered by~$N$ open sets
\[
    U_i = \{(X_1,\ldots,X_{i-1},1,X_{i+1},\ldots,X_N)\} \qquad i \in \{1,2,\ldots,N\}
\]
given by rescaling~$x_i$ to 1; here we take~$x_i \ne 0$ and set~$X_j = x_j/x_i$. The chart~$U_i$ is isomorphic to~$\CC^{N-1}/\mu_{a_i}$, where~$\mu_{a_i}$ acts on~$\CC^{N-1}$ with weights~$a_j$, $j \ne i$. In Reid's notation, this is the cyclic quotient singularity~$\frac{1}{a_i}(a_1,\ldots,\widehat{a}_i,\ldots,a_N)$; it is smooth if and only if~$a_i = 1$.

The topology of weighted projective space is very simple, with
\[
H^k\big(\PP(a_1,a_2,\ldots,a_N); \QQ) = 
\begin{cases}
    \QQ & \text{if $0 \leq k \leq 2N-2$ and $k$ is even;} \\
    0 & \text{otherwise.}
\end{cases}
\]
Hence every weighted projective space has second Betti number~$b_2 = 1$. There is a closed formula~\cite[Proposition~D.9]{CoatesCortiGalkinKasprzyk2016} for the regularized quantum period of~$X = \PP(a_1,a_2,\ldots,a_N)$:
\begin{equation} \label{eq:Ghat_wP_appendix}
    \Ghat_X(t) = \sum_{k = 0}^\infty \frac{(ak)!}{(a_1 k)! (a_2 k)! \cdots (a_N k)!} t^{ak}
\end{equation}
where~$a = a_1 + a_2 + \cdots + a_N$.
%-------------------------------------------------------------------------------
\subsubsection*{Toric varieties of Picard rank 2}
%-------------------------------------------------------------------------------
As well as weighted projective spaces, which are quotients of~$\CC^N \setminus \{0\}$ by an action of~$\Cstar$, we will consider varieties that arise as quotients of~$\CC^N \setminus S$ by~$\Cstar  \times \Cstar$, where~$S$ is a union of linear subspaces. These are examples of~\emph{toric varieties}~\cite{Fulton1993,CoxLittleSchenk2011}. Specifically, consider a matrix 
\begin{equation} \label{eq:weights_appendix}
    \begin{pmatrix}
        a_1 & a_2 & \cdots & a_N \\
        b_1 & b_2 & \cdots & b_N 
    \end{pmatrix}
\end{equation}
with non-negative integer entries and no zero columns. This defines an action of~$\Cstar  \times \Cstar$ on~$\CC^N$, where~$(\lambda,\mu) \in \Cstar  \times \Cstar$ identifies the points
\[
\begin{aligned}
    (x_1,x_2,\ldots,x_N) && \text{and} && (\lambda^{a_1} \mu^{b_1} x_1, \lambda^{a_2} \mu^{b_2} x_2, \ldots, \lambda^{a_N} \mu^{b_N} x_N)
\end{aligned}
\]
in~$\CC^N$. Set~$a = a_1 + a_2 + \ldots + a_N$ and~$b = b_1 + b_2 + \cdots + b_N$, and suppose that~$(a,b)$ is not a scalar multiple of~$(a_i,b_i)$ for any~$i$. This determines linear subspaces
\[
\begin{aligned}
    S_+ & = \left\{ (x_1,x_2,\ldots,x_N)\mid \text{$x_i = 0$ if $b_i/a_i < b/a$} \right\} \\
    S_- & = \left\{ (x_1,x_2,\ldots,x_N)\mid \text{$x_i = 0$ if $b_i/a_i > b/a$} \right\}
\end{aligned}
\]
of~$\CC^N$, and we consider the quotient
\begin{equation} \label{eq:rank_2_appendix}
X = (\CC^N \setminus S)/(\Cstar  \times \Cstar)
\end{equation}
where~$S = S_+ \cup S_-$. See e.g.~\cite[\S A.5]{BrownCortiZucconi2004}.

These quotients behave in many ways like weighted projective spaces. Indeed, if we take the weight matrix~\eqref{eq:weights_appendix} to be
\[
\begin{pmatrix}
    a_1 & a_2 & \cdots & a_N & 0 \\
    0 & 0 & \cdots & 0 & 1
\end{pmatrix}
\]
then~$X$ coincides with~$\PP(a_1,a_2,\ldots,a_N)$. We will consider only weight matrices such that the subspaces~$S_+$ and~$S_-$ both have dimension two or more; this implies that the second Betti number~$b_2(X) = 2$, and hence~$X$ is not a weighted projective space. We will refer to such quotients~\eqref{eq:rank_2_appendix} as~\emph{toric varieties of Picard rank two}, because general theory implies that the Picard lattice of~$X$ has rank two. The dimension of $X$ is $N-2$. As for weighted projective spaces, toric varieties of Picard rank two can have singular points, the precise form of which is determined by the weights~\eqref{eq:weights_appendix}. There is also a closed formula~\cite[Proposition~C.2]{CoatesCortiGalkinKasprzyk2016} for the regularized quantum period. Let~$C$ denote the cone in~$\RR^2$ defined by the equations~$a_i x + b_i y \geq 0$,~$i \in \{1,2,\ldots,N\}$. Then
\begin{equation}
    \label{eq:Ghat_rank_two_appendix}
    \Ghat_X(t) =\!\! \sum_{(k, l) \in \ZZ^2 \cap C} \frac{(ak + bl)!}{(a_1 k+b_1 l)! (a_2 k + b_2 l)! \cdots (a_N k + b_N l)!} t^{ak + bl}
\end{equation}
%-------------------------------------------------------------------------------
\subsubsection*{Classification results}
%-------------------------------------------------------------------------------
Weighted projective spaces with terminal quotient singularities have been classified in dimensions up to four; see Table~\ref{tab:wps_low_dim} for a summary. There are $35$ three-dimensional Fano toric varieties with terminal quotient singularities and Picard rank two~\cite{Kasprzyk2006}. There is no known classification of Fano toric varieties with terminal quotient singularities in higher dimension, even when the Picard rank is two.

\begin{table}[tbp]
  \centering
  \begin{tabular}{cccc}
      \toprule
      \multicolumn{4}{c}{Dimension} \\
      1 & 2 & 3 & 4 \\
      \midrule
      $\PP^1$ & $\PP^2$ & \text{7 cases} & \text{28\,686 cases} \\
      \phantom{\text{7 cases}} & \phantom{\text{7 cases}} & \text see~\cite{Kasprzyk2006} & see~\cite{Kasprzyk2013} \\
      \bottomrule
  \end{tabular}
    \caption{The classification of low-dimensional weighted projective spaces with terminal quotient singularities.}
  \label{tab:wps_low_dim}
\end{table}
%-------------------------------------------------------------------------------
\section{Supplementary Methods 1}\label{sec:supp_methods_1}
%-------------------------------------------------------------------------------
\subsubsection*{Data analysis: weighted projective spaces}
%-------------------------------------------------------------------------------
We computed an initial segment $(c_0,c_1,\ldots,c_m)$ of the regularized quantum period, with~$m \approx 100\,000$, for all the examples in the sample of $150\,000$ weighted projective spaces with terminal quotient singularities. We then considered~$\{\log c_d\}_{d\in S}$ where~$S= \{d \in \ZZ_{\geq 0}\mid c_d \neq 0\}$. To reduce dimension we fitted a linear model to the set~$\{(d, \log c_d)\mid d \in S\}$ and used the slope and intercept of this model as features. The linear fit produces a close approximation of the data. Figure~\ref{fig:error_histogram_pr1} shows the distribution of the standard errors for the slope and the $y$-intercept: the errors for the slope are between~$3.9\times 10^{-8}$ and~$1.4\times 10^{-5}$, and the errors for the~$y$-intercept are between~$0.0022$ and~$0.82$. As we will see below, the standard error for the $y$-intercept is a good proxy for the accuracy of the linear model. This accuracy decreases as the dimension grows -- see Figure~\ref{fig:intercept_histogram_by_dimension_pr1} --  but we will see below that this does not affect the accuracy of the machine learning classification.

\begin{figure}[tbp]
  \centering
  \begin{subfigure}{0.45\linewidth}
    \includegraphics[width=\linewidth]{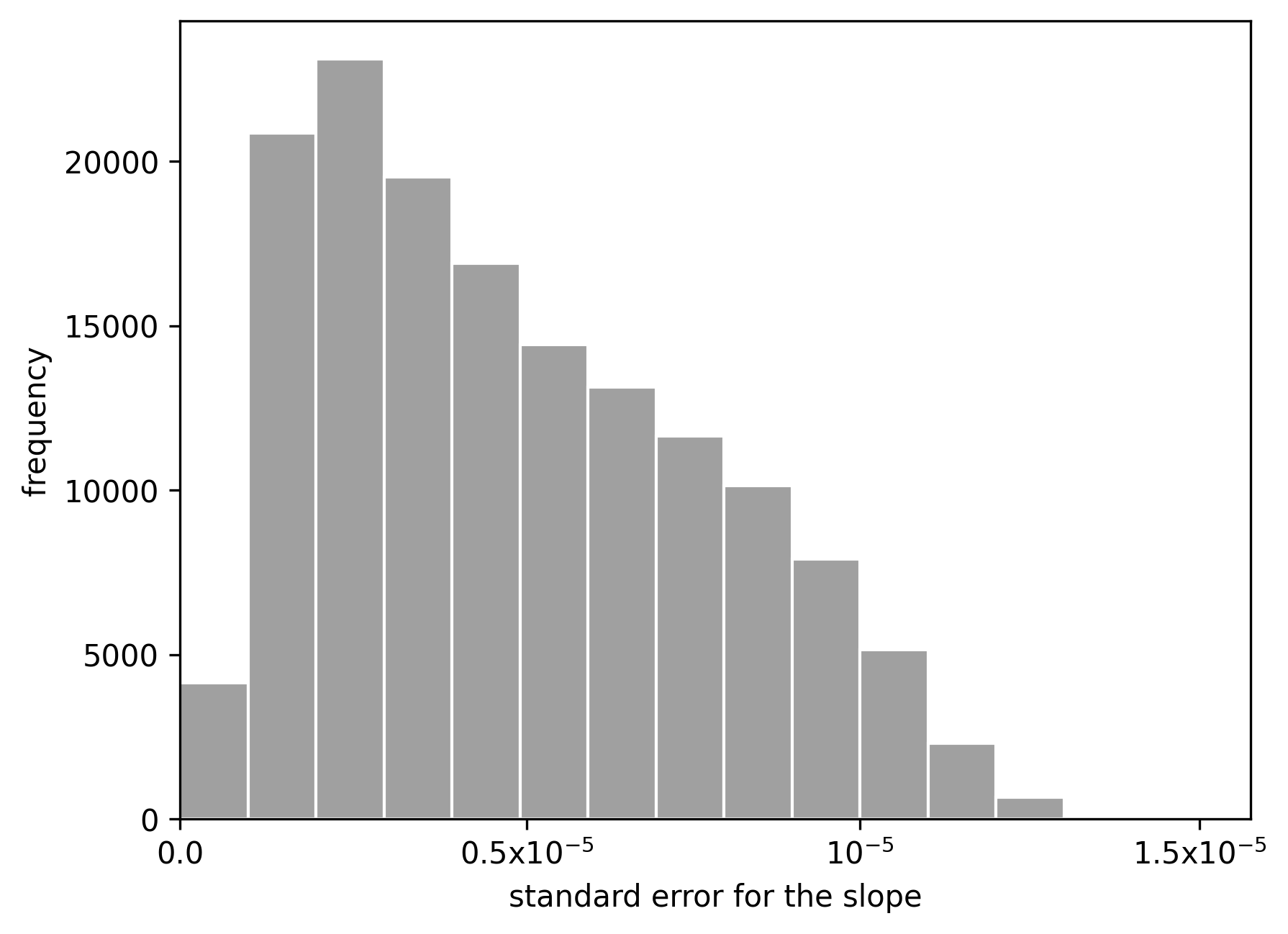}
    \caption{}\label{fig:slope_histogram_pr1}
  \end{subfigure}
  \qquad
  \begin{subfigure}{0.45\linewidth}
    \includegraphics[width=\linewidth]{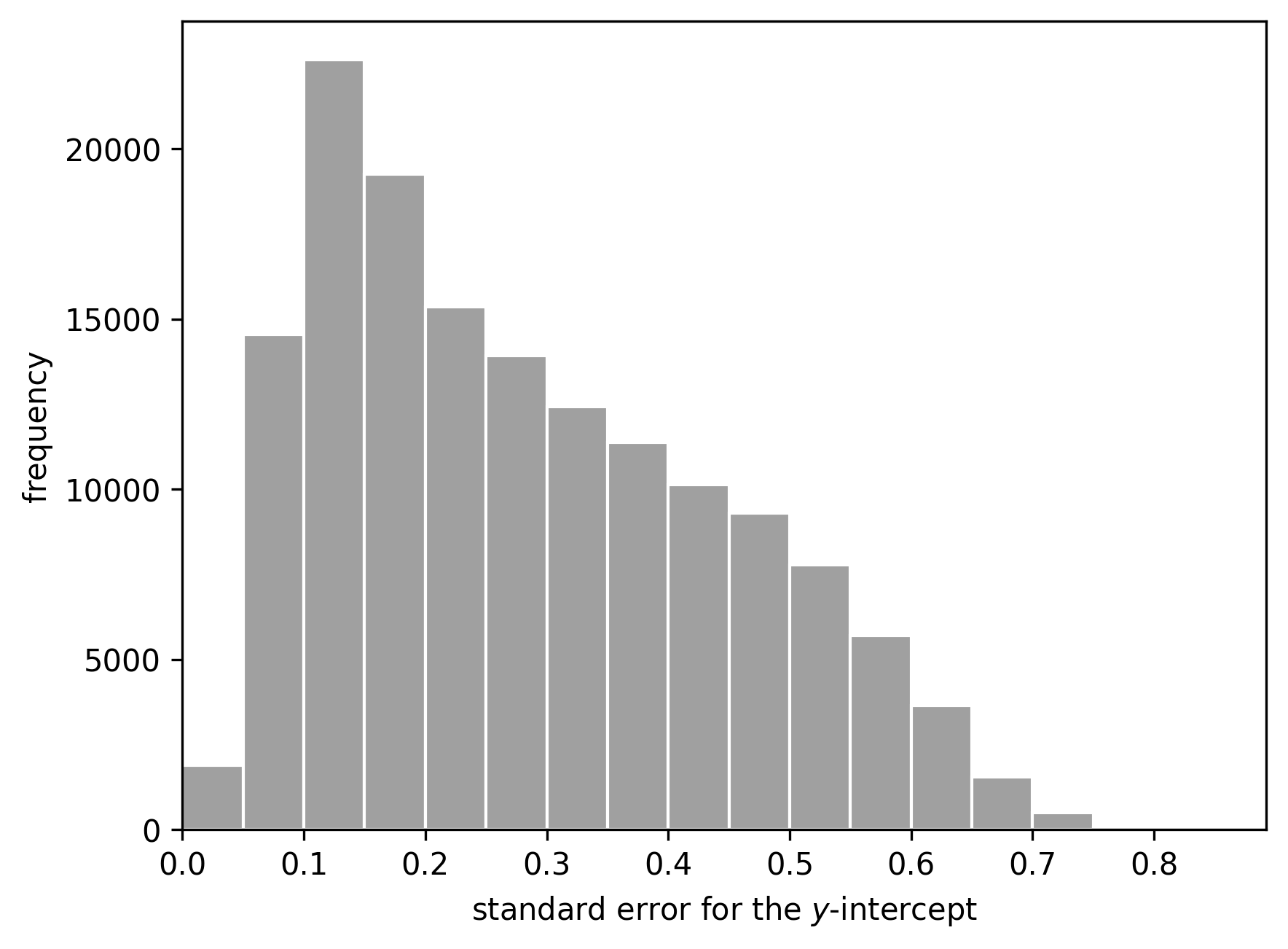}
    \caption{}\label{fig:intercept_histogram_pr1}
  \end{subfigure} \\[2ex]

  \begin{subfigure}{\textwidth}
    \includegraphics[width=\linewidth]{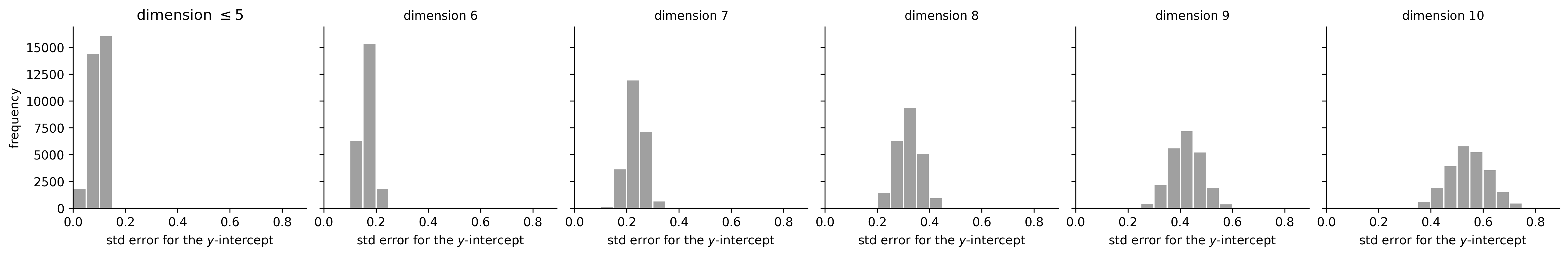}
    \caption{}\label{fig:intercept_histogram_by_dimension_pr1}
  \end{subfigure}
  \caption{Standard errors for the slope and $y$-intercept. The distribution of standard errors for the slope and $y$-intercept from the linear model applied to weighted projective spaces~$X$ with terminal quotient singularities: (a) standard error for the slope. (b) standard error for the~$y$-intercept. (c) standard error for the~$y$-intercept by dimension.}
  \label{fig:error_histogram_pr1}
\end{figure}
%-------------------------------------------------------------------------------
\subsubsection*{Data analysis: toric varieties of Picard rank 2}
%-------------------------------------------------------------------------------
We fitted a linear model to the set~$\{(d, \log c_d)\mid d \in S\}$ where~$S= \{d \in \ZZ_{\geq 0} \mid c_d \neq 0\}$, and used the slope and intercept of this linear model as features. The distribution of standard errors for the slope and~$y$-intercept of the linear model are shown in  Figure~\ref{fig:error_histogram_pr2}. The standard errors for the slope are small compared to the range of slopes, but in many cases the standard error for the $y$-intercept is relatively large. As Figure~\ref{fig:comparison_errors} illustrates, discarding data points where the standard error $\stderr$ for the $y$-intercept exceeds some threshold reduces apparent noise. As discussed above, we believe that this reflects inaccuracies in the linear regression caused by oscillatory behaviour in the initial terms of the quantum period sequence.

\begin{figure}[tbp]
  \centering
  \begin{subfigure}{.45\textwidth}
    \includegraphics[width=\linewidth]{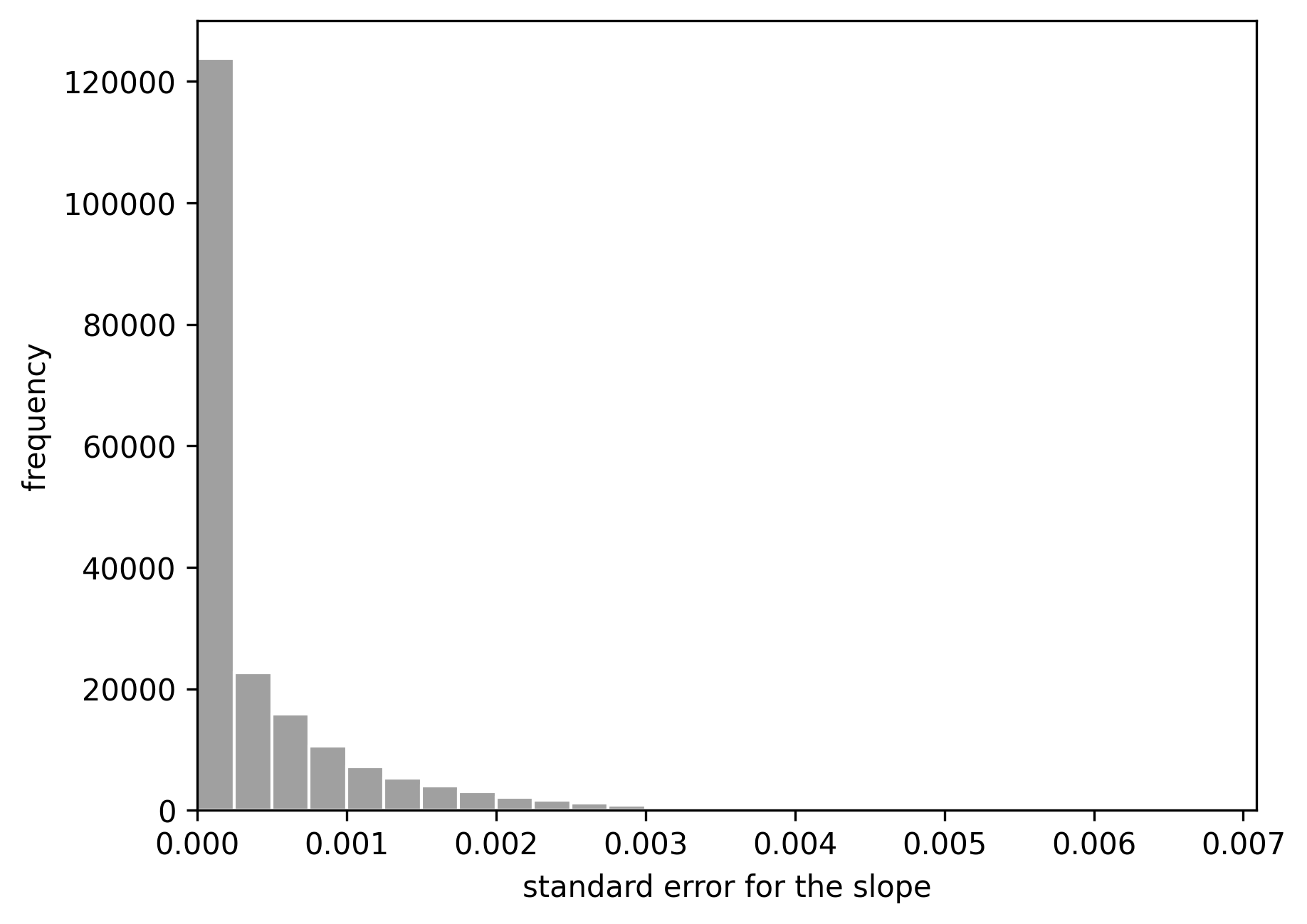}
    \caption{}\label{fig:slope_histogram}
  \end{subfigure}
  \qquad
  \begin{subfigure}{.45\textwidth}
    \includegraphics[width=\linewidth]{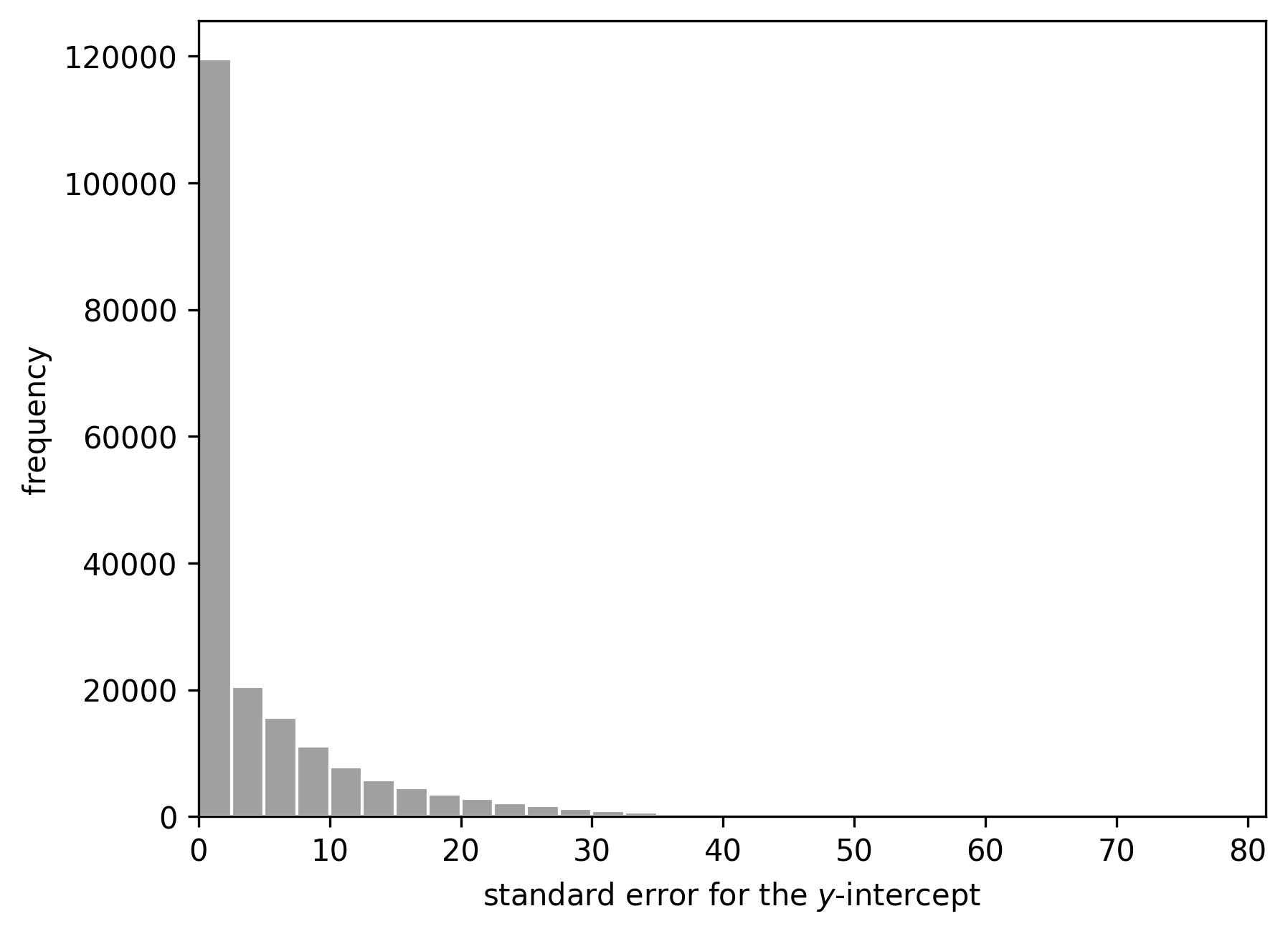}
    \caption{}\label{fig:intercept_histogram}
  \end{subfigure}
  \caption{Standard errors for the slope and $y$-intercept. The distribution of standard errors for the slope and $y$-intercept from the linear model applied to toric varieties of Picard rank two with terminal quotient singularities: (a) standard error for the slope. (b) standard error for the~$y$-intercept.}
  \label{fig:error_histogram_pr2}
\end{figure}

\begin{figure}[tbp]
  \centering
  \begin{subfigure}[t]{.45\textwidth}
    \includegraphics[width=\linewidth]{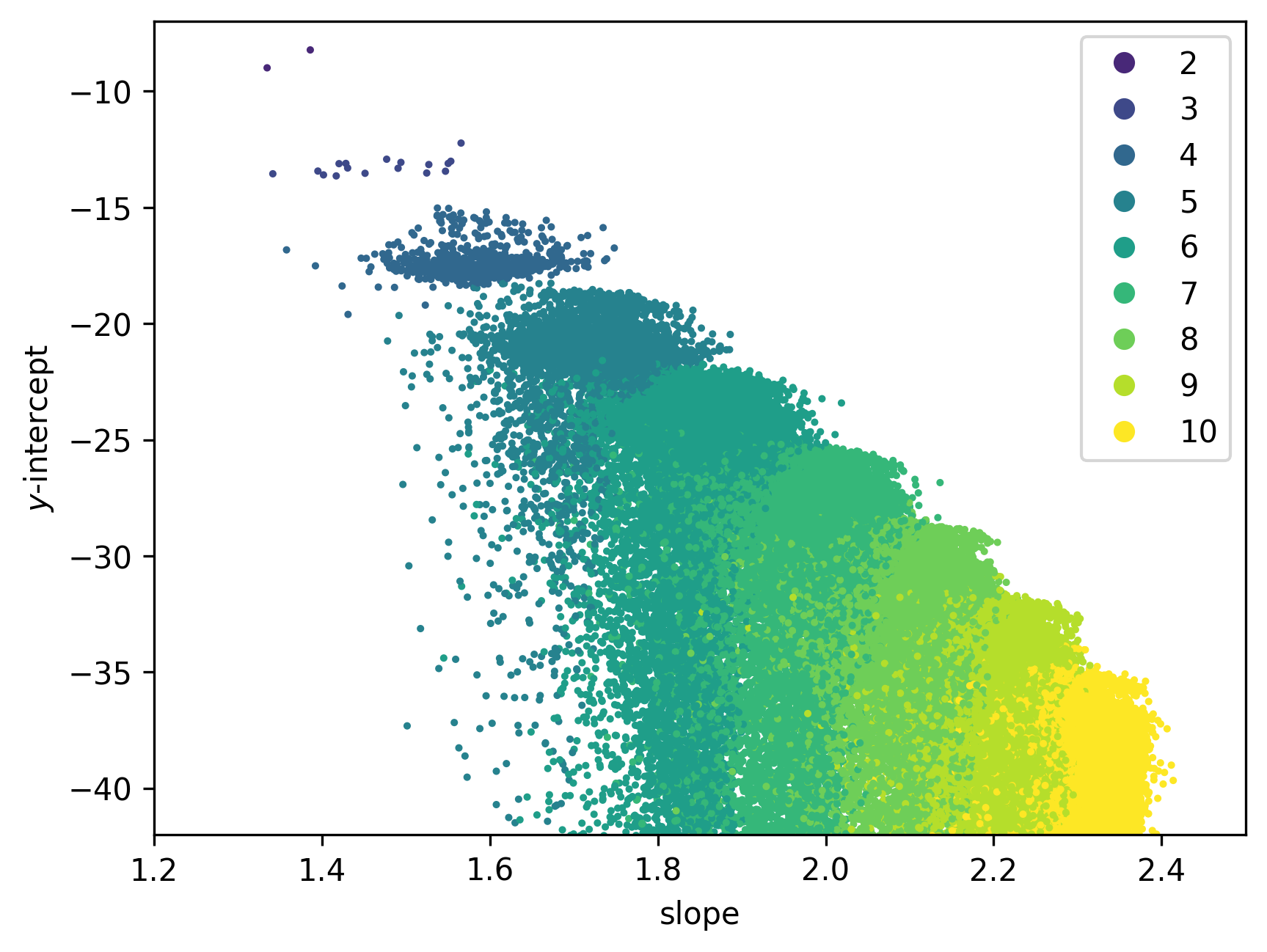}
    \caption{}\label{fig:all_standard_error}
  \end{subfigure}
  \qquad
  \begin{subfigure}[t]{.45\textwidth}
    \includegraphics[width=\linewidth]{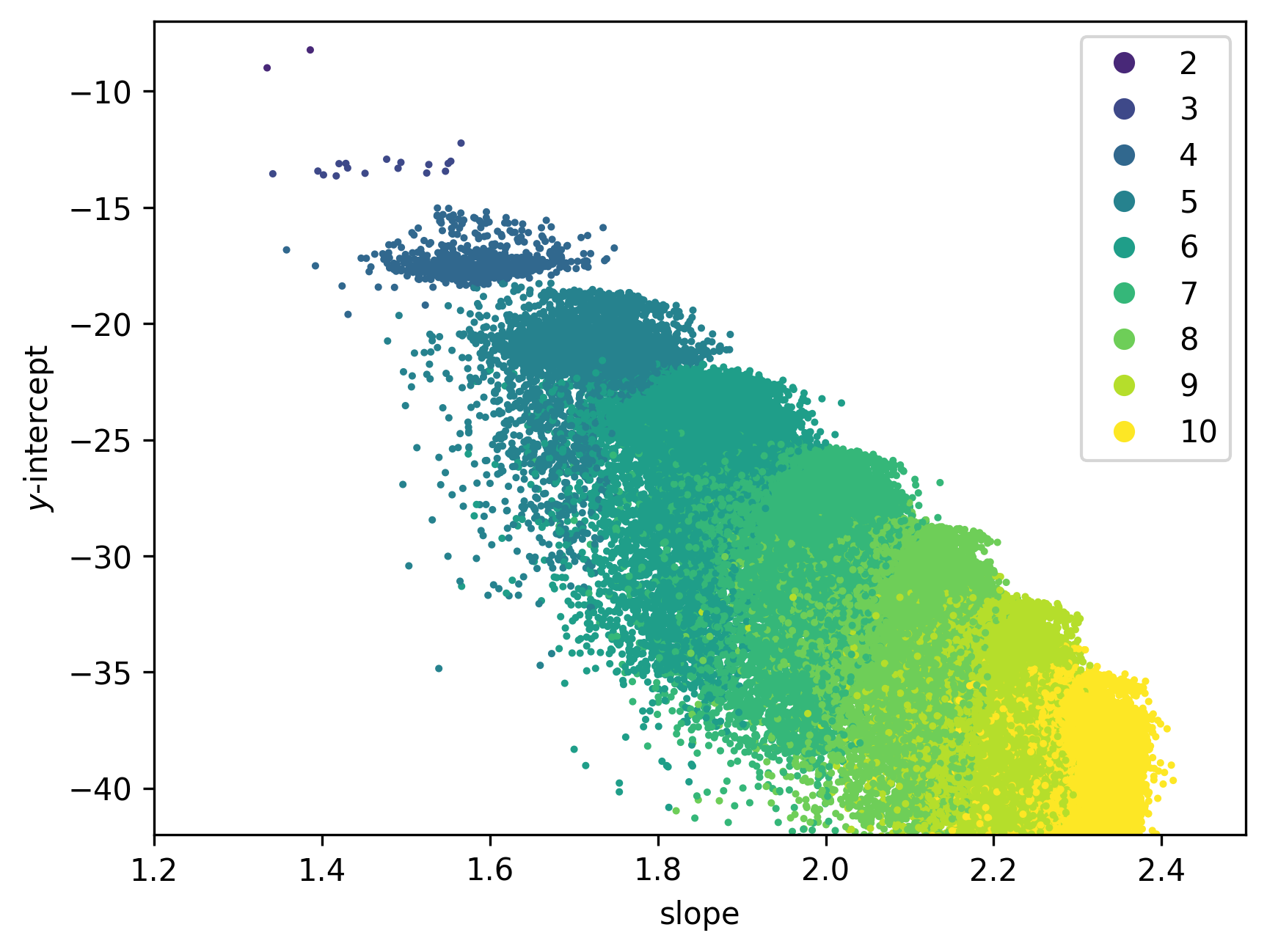}
    \caption{}\label{fig:less_1_standard_error}
  \end{subfigure}
  \begin{subfigure}[t]{.45\textwidth}
    \includegraphics[width=\linewidth]{colour_triangle_zoom_0-3.png}
    \caption{}\label{fig:less_0.2_standard_error}
  \end{subfigure}
  \qquad
  \begin{subfigure}[t]{.45\textwidth}
    \ 
  \end{subfigure}
  \caption{The slopes and $y$-intercepts from the linear model applied to toric varieties of Picard rank two with terminal quotient singularities. Data points are selected according to the standard error $\stderr$ for the $y$-intercept. The colour records the dimension of the toric variety. (a) All data points. (b) Points with $\stderr < 1$: 101\,183/200000 points. (c) Points with $\stderr < 0.3$: 67\,445/200000 points. }
  \label{fig:comparison_errors}
\end{figure}

\begin{Ex}\label{eg:rank_2_growth_appendix}
  Let us consider in more detail the toric variety from Example~\ref{eg:rank_2_growth}. In Figure~\ref{fig:linear_approximations_sm} we plot $\log c_d$ along with its linear approximation. Figure~\ref{fig:linear_approximation_0_250} shows only the first~$250$ terms, whilst Figure~\ref{fig:linear_approximation_1000_1250} shows the interval between the~$1000$th and the~$1250$th term. We see considerable deviation from the linear approximation among the first 250 terms; the deviation reduces for larger~$d$. 
\end{Ex}

\begin{figure}[tbp]
  \centering
  \begin{subfigure}{.45\textwidth}
    \includegraphics[width=\linewidth]{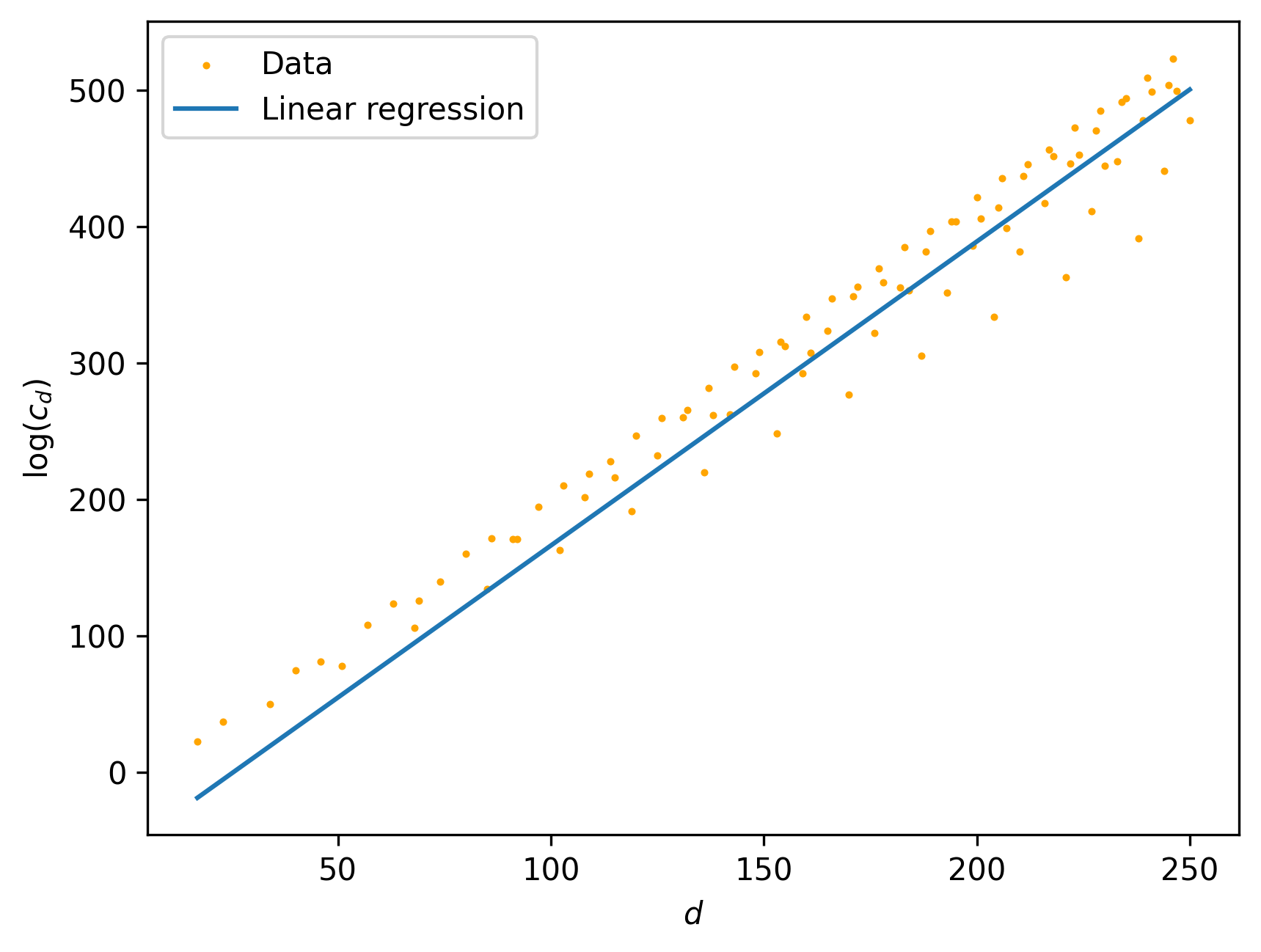}
    \caption{}\label{fig:linear_approximation_0_250}
  \end{subfigure}
  \qquad
  \begin{subfigure}{.45\textwidth}
    \includegraphics[width=\linewidth]{linear_approx_last250.png}
    \caption{}\label{fig:linear_approximation_1000_1250}
  \end{subfigure}
  \caption{The logarithm of the non-zero coefficients $c_d$ for Example~\ref{eg:rank_2_growth}: (a) the first 250 terms. (b) terms between $d=1000$ and $d=1250$. In each case, the linear approximation is also shown.}
  \label{fig:linear_approximations_sm}
\end{figure}
%-------------------------------------------------------------------------------
\section{Supplementary Methods 2}\label{sec:supp_methods_2}
%-------------------------------------------------------------------------------
We performed our experiments using scikit-learn~\cite{scikit-learn}, a standard machine learning library for Python. The computations that produced the data shown in Figure~\ref{fig:hedgehogs} were performed using Mathematica~\cite{Mathematica}. All code required to replicate the results in this paper is available from Bitbucket under an MIT license~\cite{supporting-code}.
%-------------------------------------------------------------------------------
\subsubsection*{Weighted projective spaces}
%-------------------------------------------------------------------------------
We excluded dimensions one and two from the analysis, since there is only one weighted projective space in each case (namely~$\PP^1$ and~$\PP^2$). Therefore we have a dataset of~$149\,998$ slope-intercept pairs, labelled by the dimension which varies between three and ten. We standardised the features, by translating the means to zero and scaling to unit variance, and applied a Support Vector Machine~(SVM) with linear kernel and regularisation parameter~$C=10$. By looking at different train--test splits we obtained the learning curves shown in Figure~\ref{fig:learning_curve_sv_pr1}. The figure displays the mean accuracies for both training and validation data obtained by performing five random test-train splits each time: the shaded areas around the lines correspond to the~$1\sigma$ region, where $\sigma$ denotes the standard deviation. Using~$10\%$ (or more) of the data for training we obtained an accuracy of~$99.99\%$. In Figure~\ref{fig:decision_boundaries_pr1} we plot the decision boundaries computed by the~SVM between neighbouring dimension classes.

\begin{figure}[tbp]
  \centering
  \includegraphics[width=0.4\textwidth]{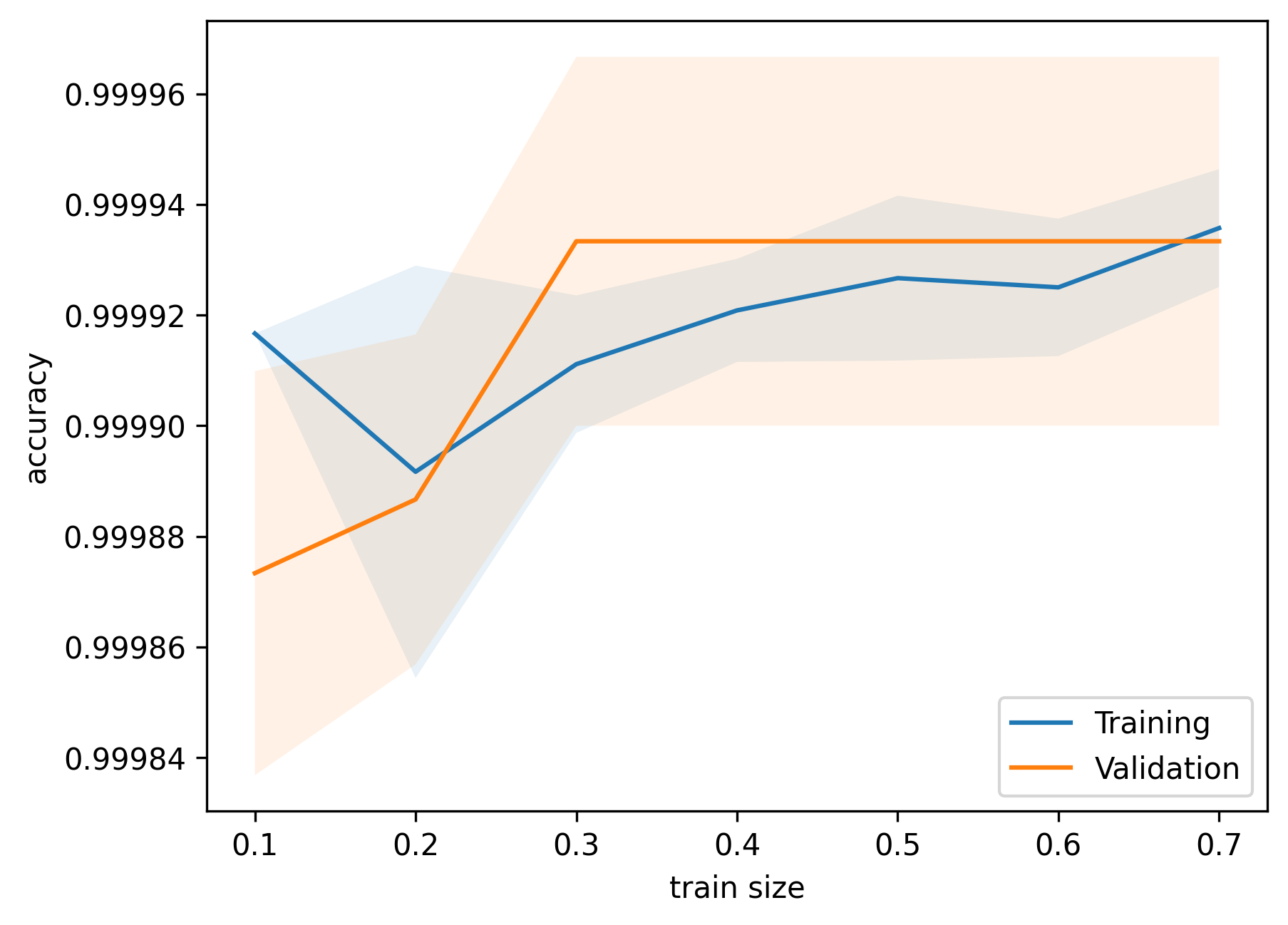}
  \caption{Learning curves for a Support Vector Machine with linear kernel applied to the dataset of weighted projective spaces. The plot shows the means of the training and validation accuracies for five different random train--test splits. The shaded regions show the~$1\sigma$ interval, where~$\sigma$ is the standard deviation.}\label{fig:learning_curve_sv_pr1}
\end{figure}

\begin{figure}[tbp]
  \centering
  \includegraphics[width=0.4\textwidth]{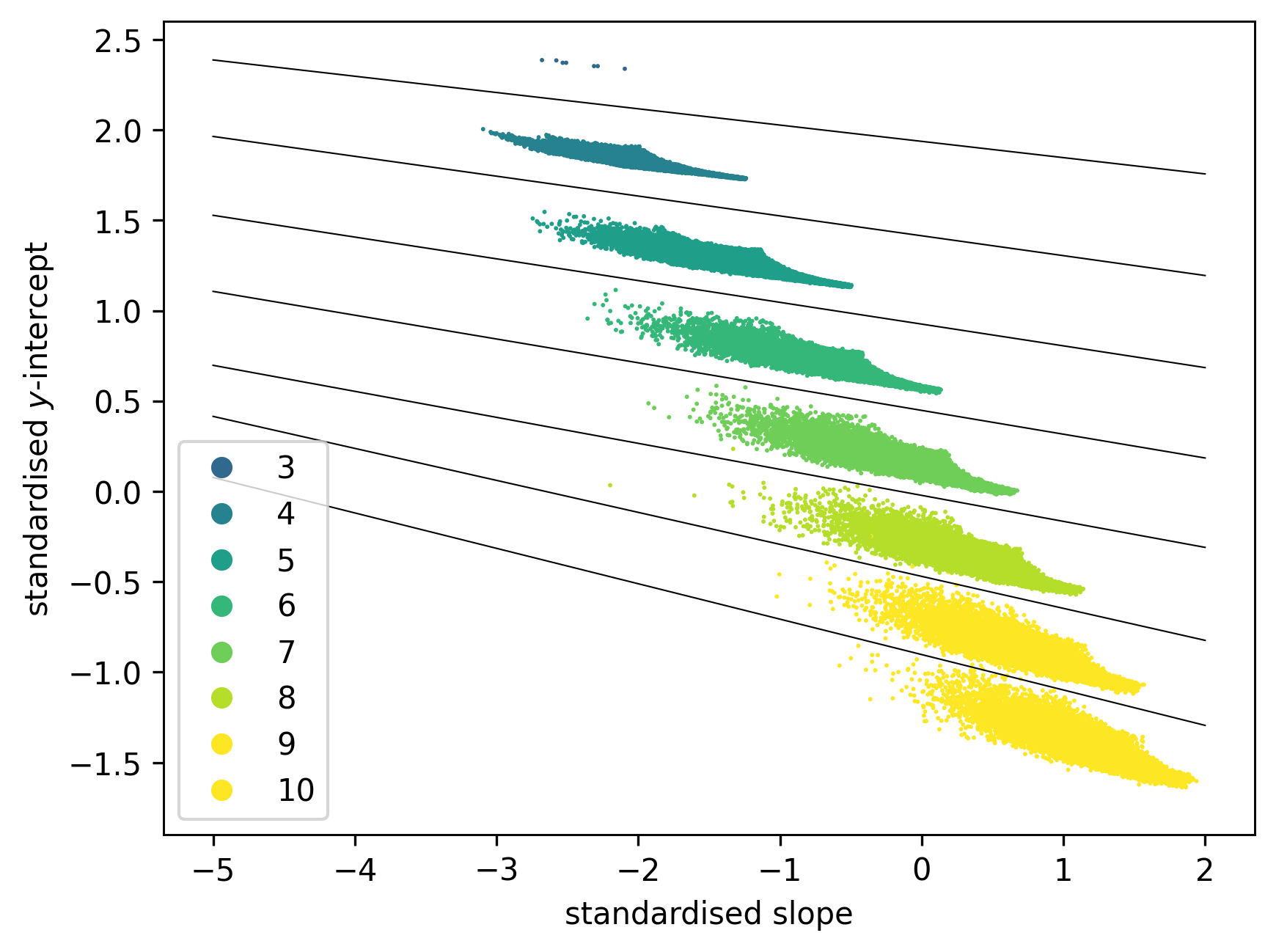}
  \caption{Decision boundaries computed from a Support Vector Machine with linear kernel trained on 70\% of the dataset of weighted projective spaces. Note that the data has been standardised.}\label{fig:decision_boundaries_pr1}
\end{figure}
%-------------------------------------------------------------------------------
\subsubsection*{Toric varieties of Picard rank 2}
%-------------------------------------------------------------------------------
In light of the discussion above, we restricted attention to toric varieties with Picard rank two such that the $y$-intercept standard error $\stderr$ is less than $0.3$. We also excluded dimension two from the analysis, since in this case there are only two varieties (namely,~$\PP^1\times\PP^1$ and the Hirzebruch surface~$\FF_1$). The resulting dataset contains $67\,443$ slope-intercept pairs, labelled by dimension; the dimension varies between three and ten, as shown in Table~\ref{tab:filtered_rank_2}. 

\begin{table}[tbp]
  \centering
  \begin{tabular}{crr}
  \toprule
  \multicolumn{3}{c}{Rank-two toric varieties with $\stderr < 0.3$} \\
  \midrule
  Dimension&Sample size&Percentage\\
  \midrule
  3&17&0.025\\
  4&758&1.124\\
  5&5\,504&8.161\\
  6&12\,497&18.530\\
  7&16\,084&23.848\\
  8&13\,701&20.315\\
  9&10\,638&15.773\\
  10&8\,244&12.224\\
  \midrule
  Total&67\,443&\\
  \bottomrule
  \end{tabular}
 \caption{The distribution by dimension among toric varieties of Picard rank two in our dataset with $\stderr < 0.3$.}
  \label{tab:filtered_rank_2}
\end{table}
%-------------------------------------------------------------------------------
\subsubsection*{Support Vector Machine}
%-------------------------------------------------------------------------------
We used a linear~SVM with regularisation parameter~$C=50$. By considering different train--test splits we obtained the learning curves shown in Figure~\ref{fig:learning_curve_svm_pr2}, where the means and the standard deviations were obtained by performing five random samples for each split. Note that the model did not overfit. We obtained a validation accuracy of~$88.2\%$ using~$70\%$ of the data for training. Figure~\ref{fig:decision_boundaries_pr2} shows the decision boundaries computed by the~SVM between neighbouring dimension classes. Figure~\ref{fig:confusion_matrices_svm} shows the confusion matrices for the same train--test split.

\begin{figure}[tbp]
  \centering
  \includegraphics[width=0.4\textwidth]{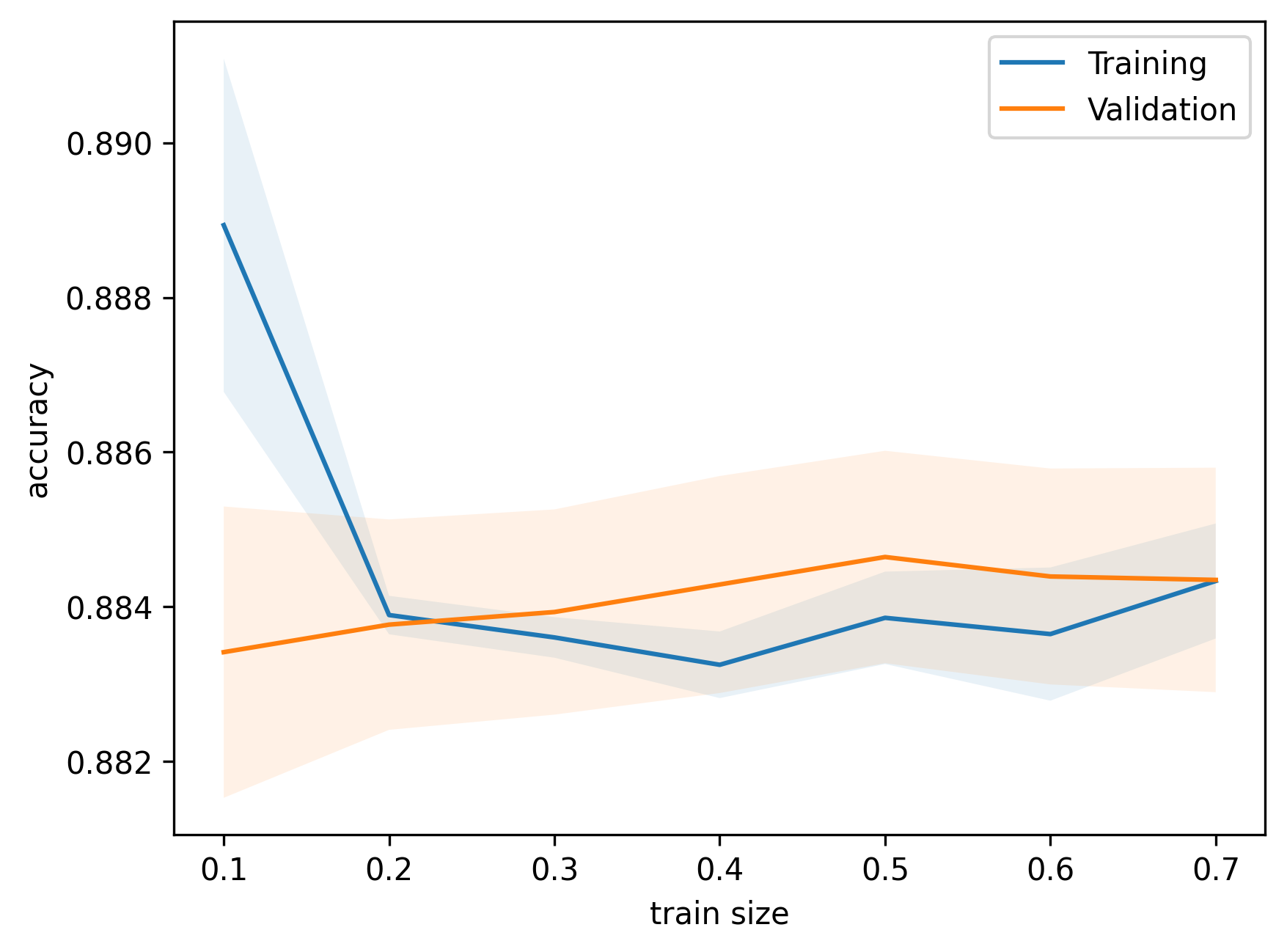}
  \caption{Learning curves for a Support Vector Machine with linear kernel applied to the dataset of toric varieties of Picard rank two. The plot shows the means of the training and validation accuracies for five different random train--test splits. The shaded regions show the~$1\sigma$ interval, where~$\sigma$ is the standard deviation.}\label{fig:learning_curve_svm_pr2}
\end{figure}

\begin{figure}[tbp]
    \centering
      \includegraphics[width=0.4\textwidth]{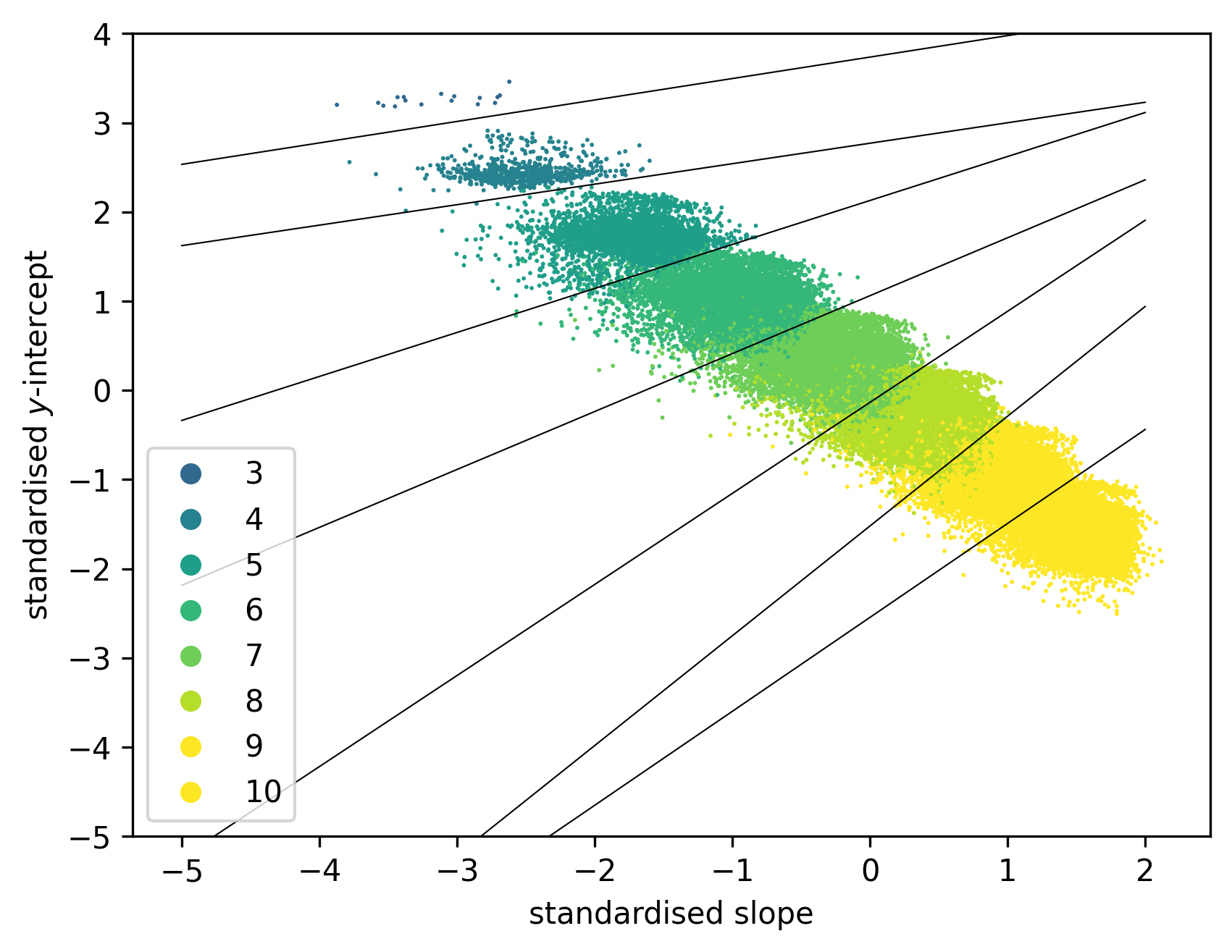}
    \caption{Decision boundaries computed from a Support Vector Machine with linear kernel trained on $70\%$ of the dataset of toric varieties of Picard rank two. Note that the data has been standardised.}
      \label{fig:decision_boundaries_pr2}
\end{figure}

\begin{figure}[tbp]
    \centering
    \begin{subfigure}{.45\textwidth}
      \includegraphics[width=\linewidth]{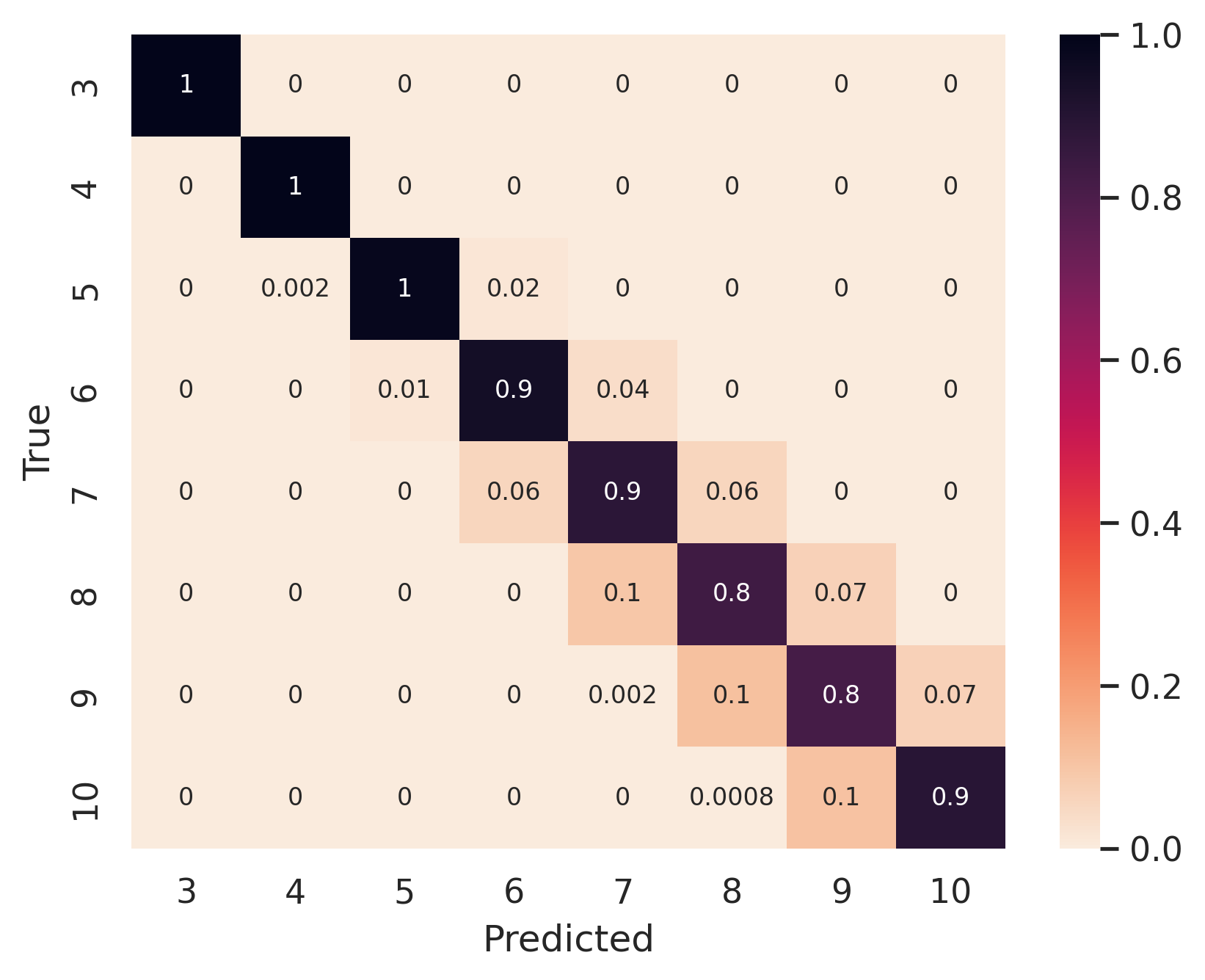}
      \caption{Confusion matrix normalised with respect to the true values.}
    \end{subfigure}
    \qquad
    \begin{subfigure}{.45\textwidth}
      \includegraphics[width=\linewidth]{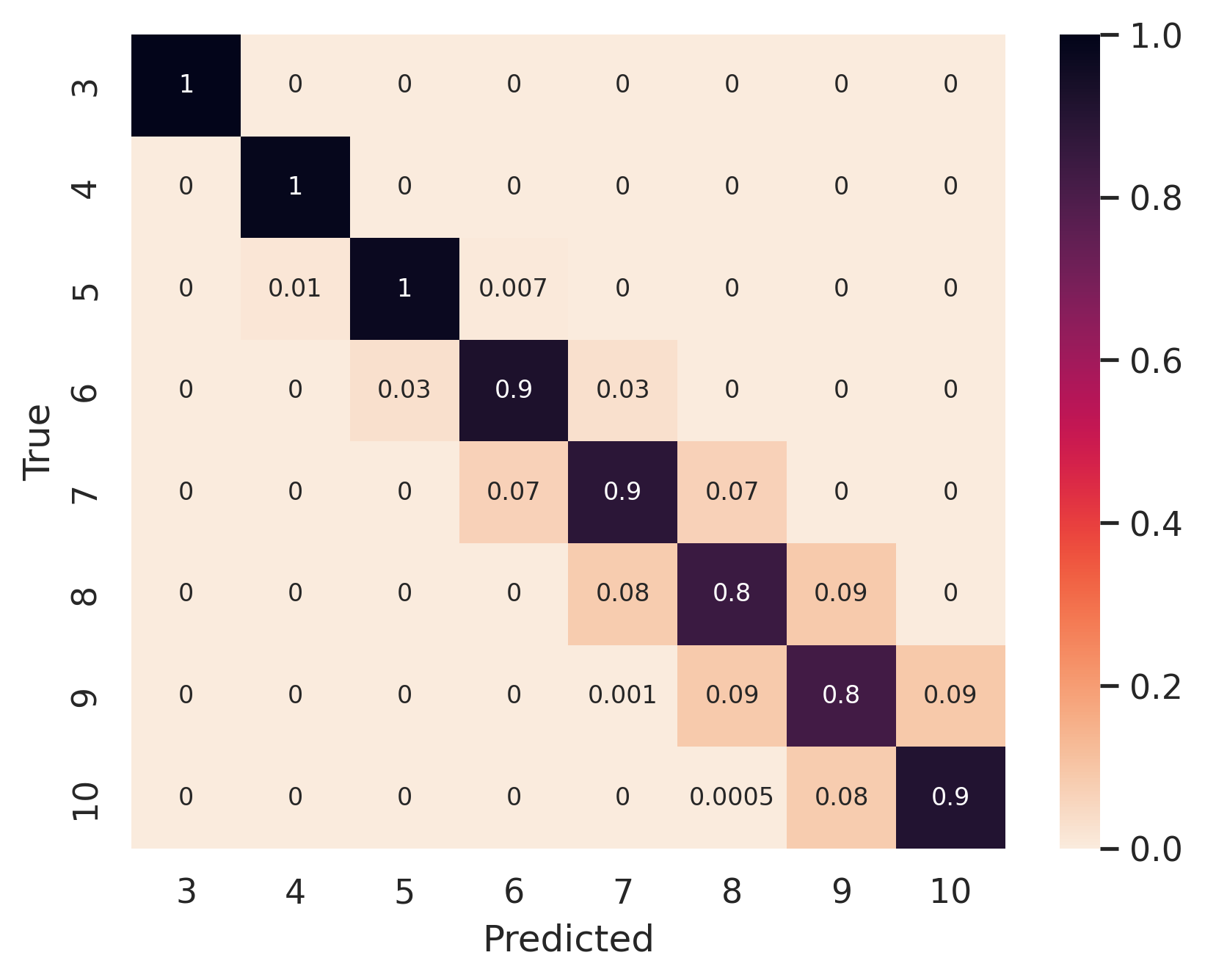}
      \caption{Confusion matrix normalised with respect to the predicted values.}
    \end{subfigure}
    \caption{Confusion matrices for a Support Vector Machine with linear kernel trained on $70\%$ of the dataset of toric varieties of Picard rank two.}\label{fig:confusion_matrices_svm}
\end{figure}
%-------------------------------------------------------------------------------
\subsubsection*{Random Forest Classifier}
%-------------------------------------------------------------------------------
We used a Random Forest Classifier~(RFC) with~$1500$ estimators and the same features (slope and $y$-intercept for the linear model). By considering different train--test splits we obtained the learning curves shown in Figure~\ref{fig:learning_curve_rfc_pr2}; note again that the model did not overfit. Using 70\% of the data for training, the~RFC gave a validation accuracy of $89.4\%$. Figure~\ref{fig:confusion_matrices_rfc} on page~\pageref{fig:confusion_matrices_rfc} shows confusion matrices for the same train--test split.    

\begin{figure}[tbp]
  \centering
  \includegraphics[width=0.4\textwidth]{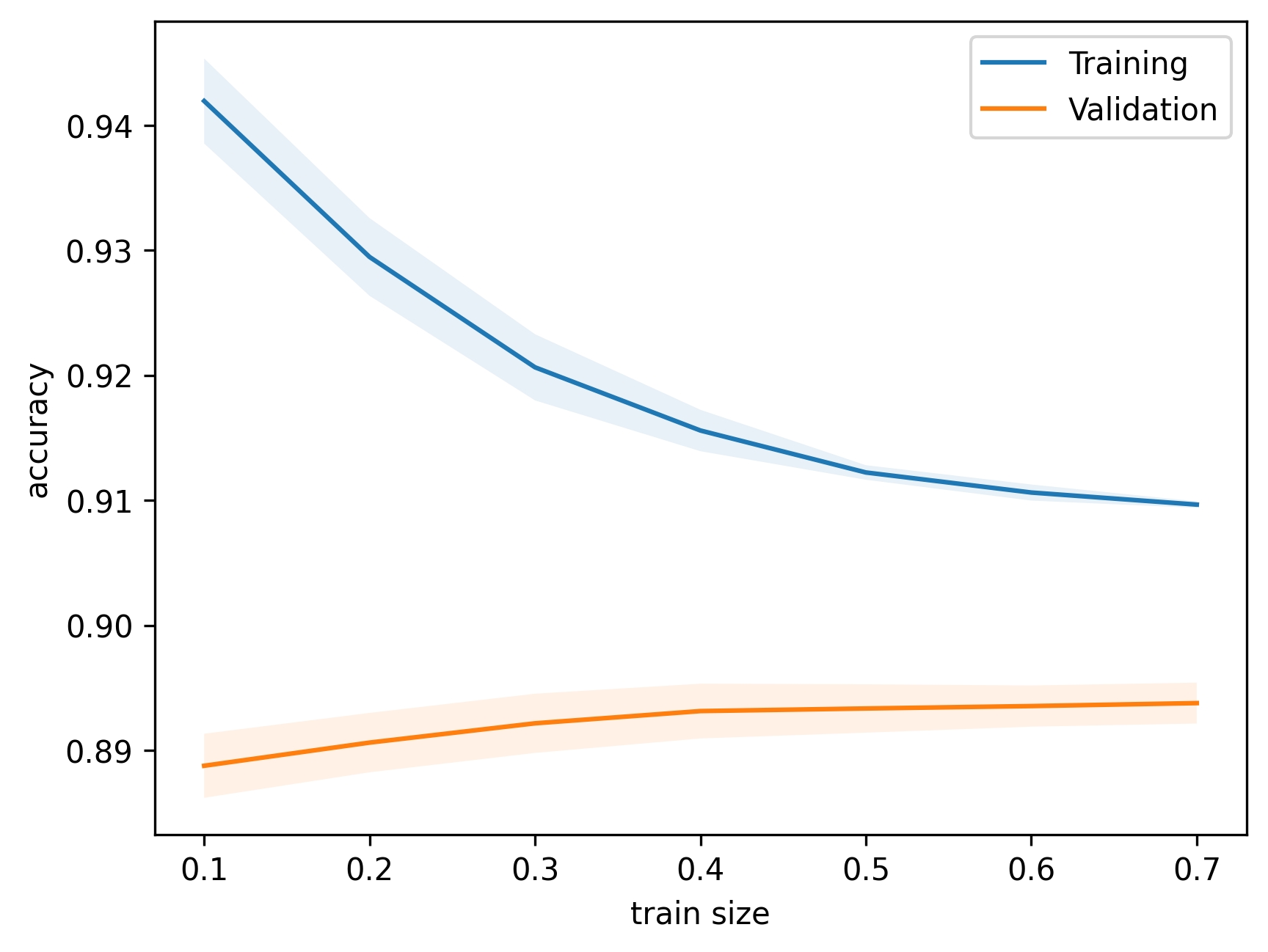}
  \caption{Learning curves for a Random Forest Classifier applied to the dataset of toric varieties of Picard rank two. The plot shows the means of the training and validation accuracies for five different random train--test splits. The shaded regions show the~$1\sigma$ interval, where~$\sigma$ is the standard deviation.}
  \label{fig:learning_curve_rfc_pr2}
\end{figure}

\begin{figure}[tbp]
    \centering
    \begin{subfigure}{.45\textwidth}
      \includegraphics[width=\linewidth]{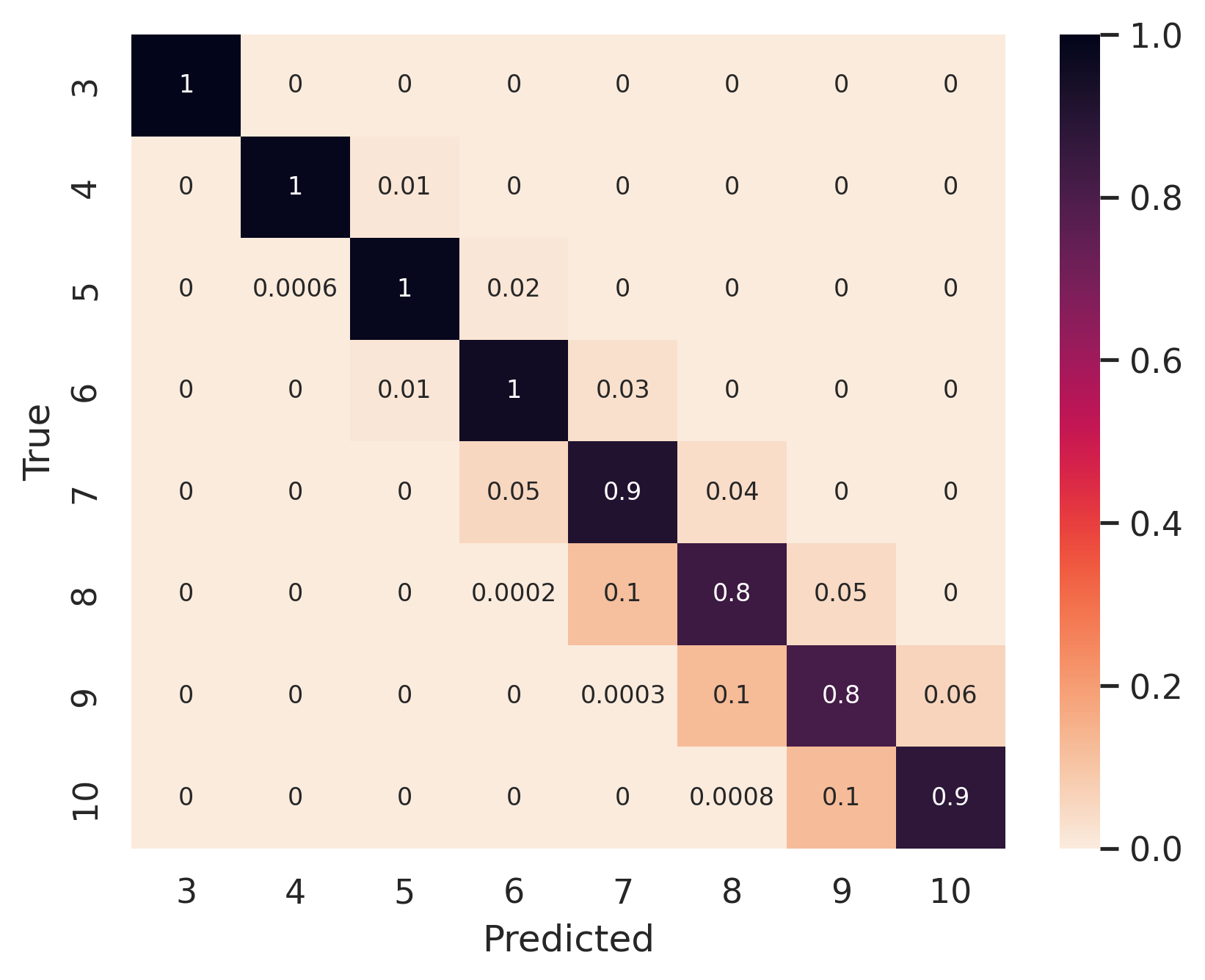}
      \caption{Confusion matrix normalised with respect to the true values.}
    \end{subfigure}
    \qquad
    \begin{subfigure}{.45\textwidth}
      \includegraphics[width=\linewidth]{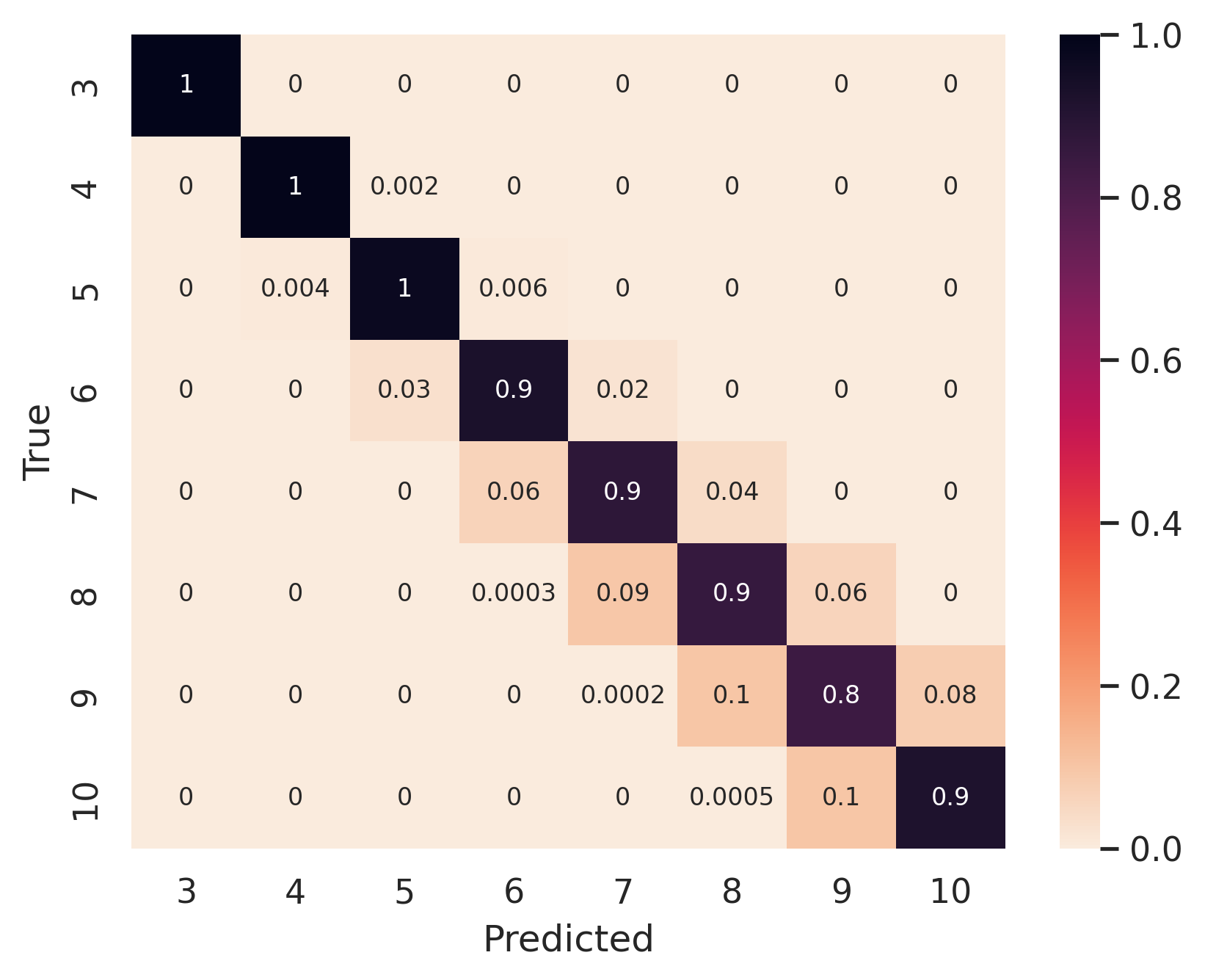}
      \caption{Confusion matrix normalised with respect to the predicted values.}
    \end{subfigure}
    \caption{Confusion matrices for a Random Forest Classifier trained on $70\%$ of the dataset of toric varieties of Picard rank two.}\label{fig:confusion_matrices_rfc}
\end{figure}
%-------------------------------------------------------------------------------
\subsubsection*{Feed-forward neural network}
%-------------------------------------------------------------------------------
As discussed above, neural networks do not handle unbalanced datasets well, and therefore we removed the toric varieties with dimensions three, four, and five from our dataset: see Table~\ref{tab:filtered_rank_2}. We trained a Multilayer Perceptron~(MLP) classifier on the same features, using an~MLP with three hidden layers~$(10,30,10)$, Adam optimiser~\cite{KingmaDiederikBa2014}, and rectified linear activation function~\cite{AgarapAbien2018}. Different train--test splits produced the learning curve in Figure~\ref{fig:learning_curve_nn}; again the model did not overfit. Using~$70\%$ of the data for training, the~MLP gave a validation accuracy of~$88.7\%$. One could further balance the dataset, by randomly undersampling so that there are the same number of representatives in each dimension (8244 representatives: see Table~\ref{tab:filtered_rank_2}). This resulted in a slight decrease in accuracy: the better balance was outweighed by loss of data caused by undersampling.

\begin{figure}[tbp]
  \centering
  \includegraphics[width=0.4\textwidth]{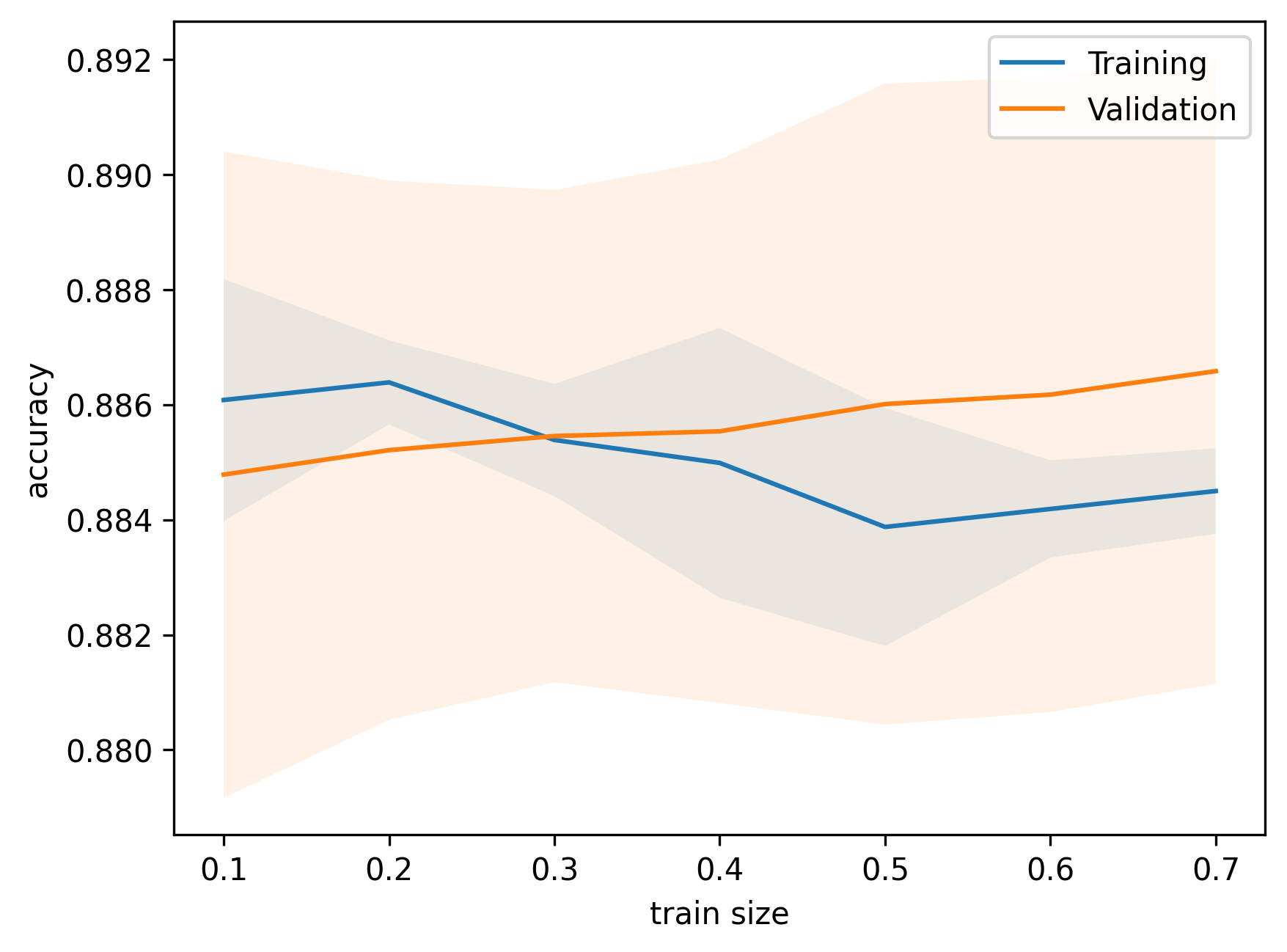}
  \caption{Learning curves for a Multilayer Perceptron classifier $\MLP_2$ applied to the dataset of toric varieties of Picard rank two and dimension at least six, using just the regression data as features. The plot shows the means of the training and validation accuracies for five different random train--test splits. The shaded regions show the~$1\sigma$ interval, where~$\sigma$ is the standard deviation.}\label{fig:learning_curve_nn}
\end{figure}
%-------------------------------------------------------------------------------
\subsubsection*{Feed-forward neural network with many features}
%-------------------------------------------------------------------------------
We trained an~MLP with the same architecture, but supplemented the features by including $\log c_d$ for $1 \leq d \leq 100$
(unless $c_d$ was zero in which case we set that feature to zero), as well as the slope and $y$-intercept as before. We refer to the previous neural network as $\MLP_2$, because it uses 2 features, and refer to this neural network as $\MLP_{102}$, because it uses 102 features. Figure~\ref{fig:learning_curve_nn_more} shows the learning curves obtained for different train--test splits. Using~$70\%$ of the data for training, the $\MLP_{102}$ model gave a validation accuracy of $97.7\%$.

\begin{figure}[tbp]
  \centering
  \includegraphics[width=0.4\textwidth]{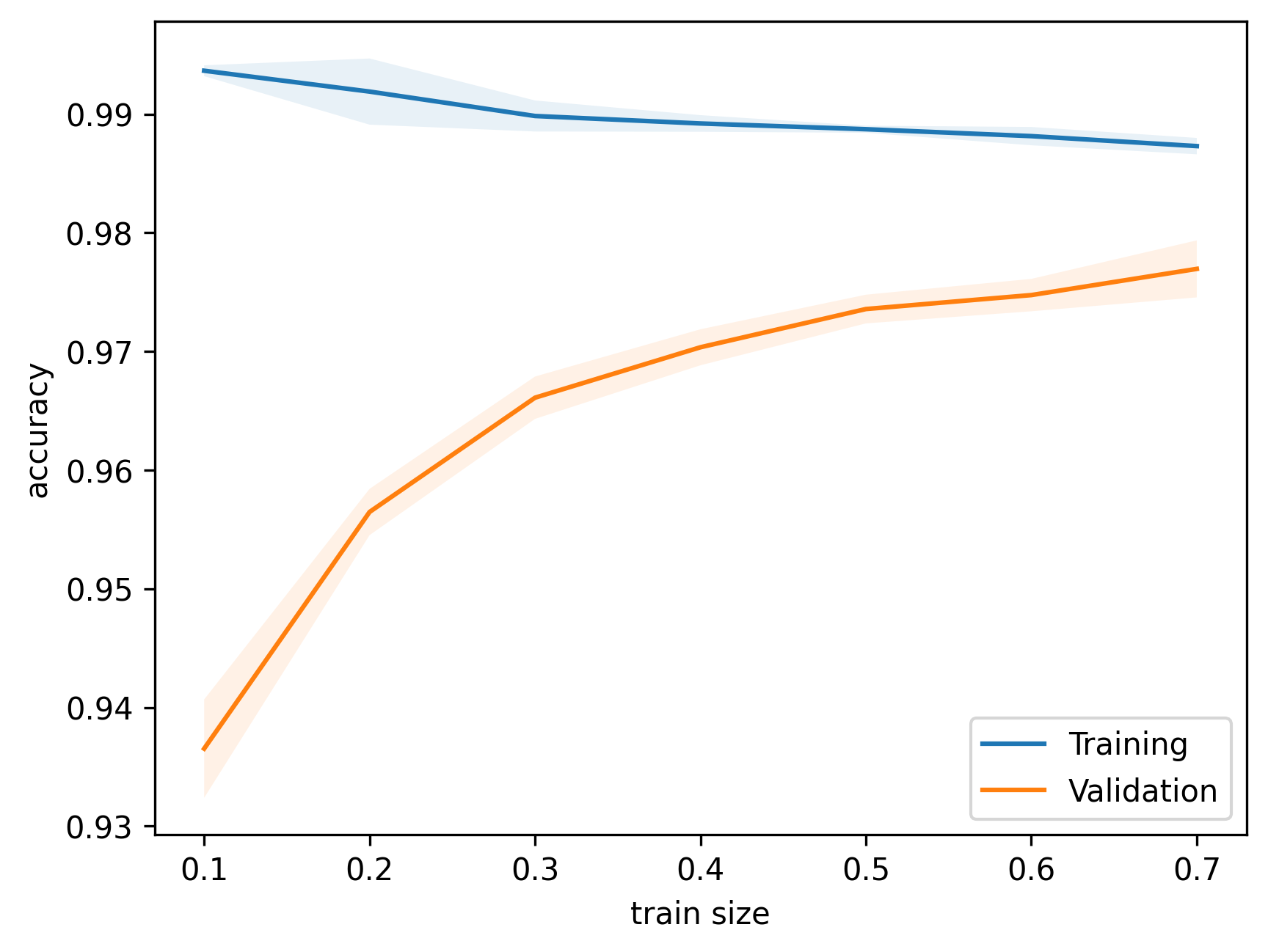}
  \caption{Learning curves for a Multilayer Perceptron classifier $\MLP_{102}$ applied to the dataset of toric varieties of Picard rank two and dimension at least six, using as features the regression data as well as $\log c_d$ for $1 \leq d \leq 100$. The plot shows the means of the training and validation accuracies for five different random train--test splits. The shaded regions show the~$1\sigma$ interval, where~$\sigma$ is the standard deviation.}\label{fig:learning_curve_nn_more}
\end{figure}

We do not understand the reason for the performance improvement between $\MLP_{102}$ and $\MLP_2$. But one possible explanation is the following. Recall that the first 1000 terms of the period sequence were excluded when calculating the slope and intercept, because they exhibit irregular oscillations that decay as $d$ grows. These oscillations reduce the accuracy of the linear regression. The oscillations may, however, carry information about the toric variety, and so including the first few values of $\log(c_d)$ potentially makes more information available to the model. For example, examining the pattern of zeroes at the beginning of the sequence $(c_d)$ sometimes allows one to recover the values of $a$ and $b$ -- see~\eqref{eq:Ghat_rank_two_appendix} for the notation. This information is relevant to estimating the dimension because, as a very crude approximation, larger $a$ and $b$ go along with larger dimension. Omitting the slope and intercept, however, and training on the coefficients $\log c_d$ for $1 \leq d \leq 100$ with the same architecture gave an accuracy of only 62\%.
%-------------------------------------------------------------------------------
\subsubsection*{Comparison of models}
%-------------------------------------------------------------------------------
The validation accuracies of the~SVM,~RFC, and the neural networks $\MLP_2$ and $\MLP_{102}$, on the same data set ($\stderr<0.3$, dimension between six and ten), are compared in Table~\ref{tab:compare_models}. Their confusion matrices are shown in Table~\ref{tab:compare_models_cm}. All models trained on only the regression data performed well, with the~RFC slightly more accurate than the~SVM and the neural network $\MLP_2$ slightly more accurate still. Misclassified examples are generally in higher dimension, which is consistent with the idea that misclassification is due to convergence-related noise. The neural network trained on the supplemented feature set, $\MLP_{102}$, outperforms all other models. However, as discussed above, feature importance analysis using SHAP values showed that the slope and the intercept were the most influential features in the prediction. 

\begin{table}[tbp]
  \centering
    \begin{tabular}{cccc}
        \toprule
        \multicolumn{4}{c}{ML models} \\
        SVM & RFC & $\MLP_2$ & $\MLP_{102}$ \\
        \midrule
        $87.7\%$ & $88.6\%$ & $88.7\%$ & $97.7 \%$ \\
        \bottomrule
    \end{tabular}
    \caption{Comparison of model accuracies. Accuracies for various models applied to the dataset of toric varieties of Picard rank two and dimension at least six: a Support Vector Machine with linear kernel, a Random Forest Classifier, and the neural networks $\MLP_2$ and $\MLP_{102}$.}
    \label{tab:compare_models}
  \end{table}
  
\begin{table}[tbp]
  \centering
  \begin{tabular}{c@{}m{.4\linewidth}@{}m{.4\linewidth}@{}}
      \toprule
      \multicolumn{1}{c}{Model}&\multicolumn{1}{c}{True confusion matrix}&\multicolumn{1}{c}{Predicted confusion matrix}\\
      \midrule
      SVM & \includegraphics[width=\linewidth]{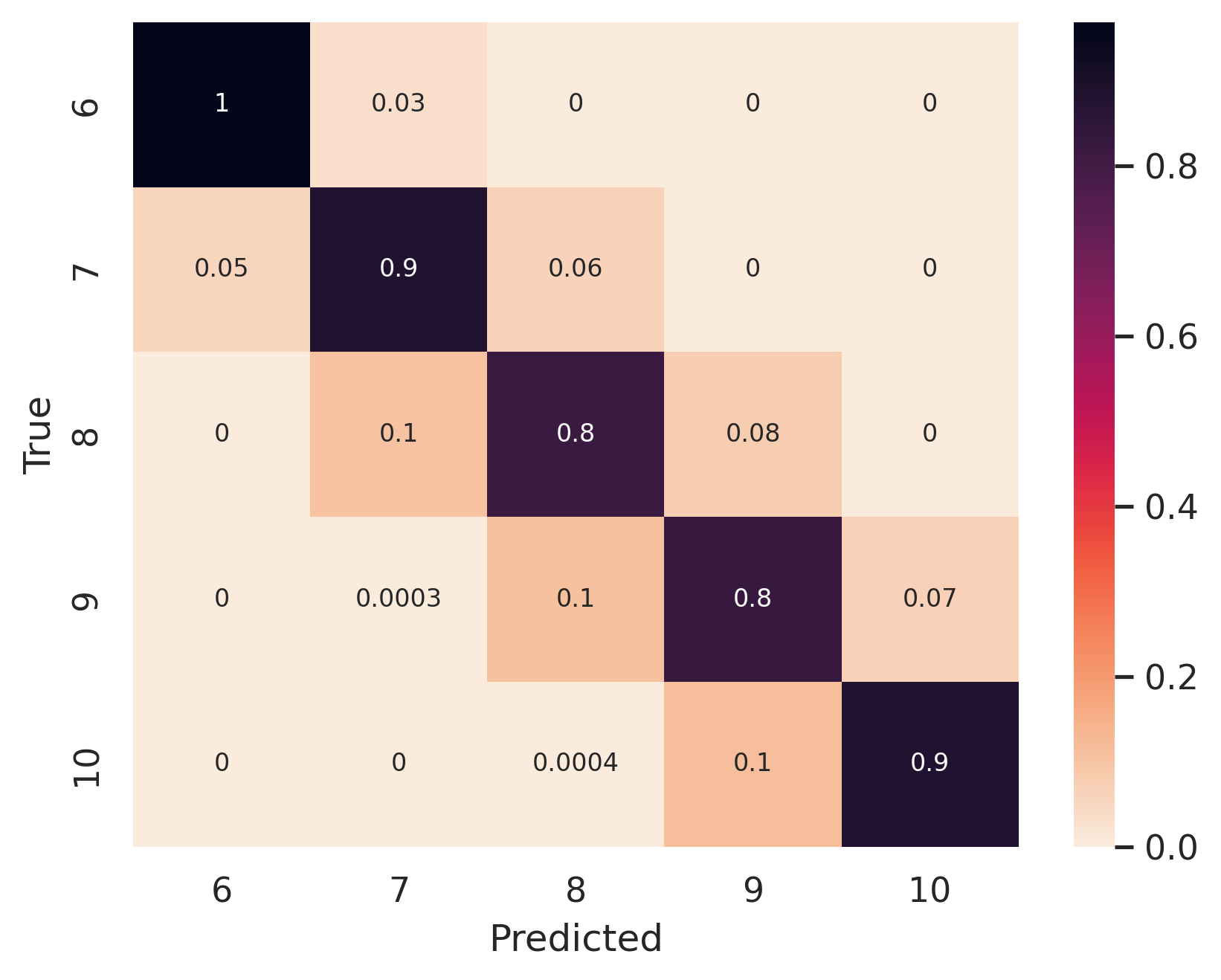} & \includegraphics[width=\linewidth]{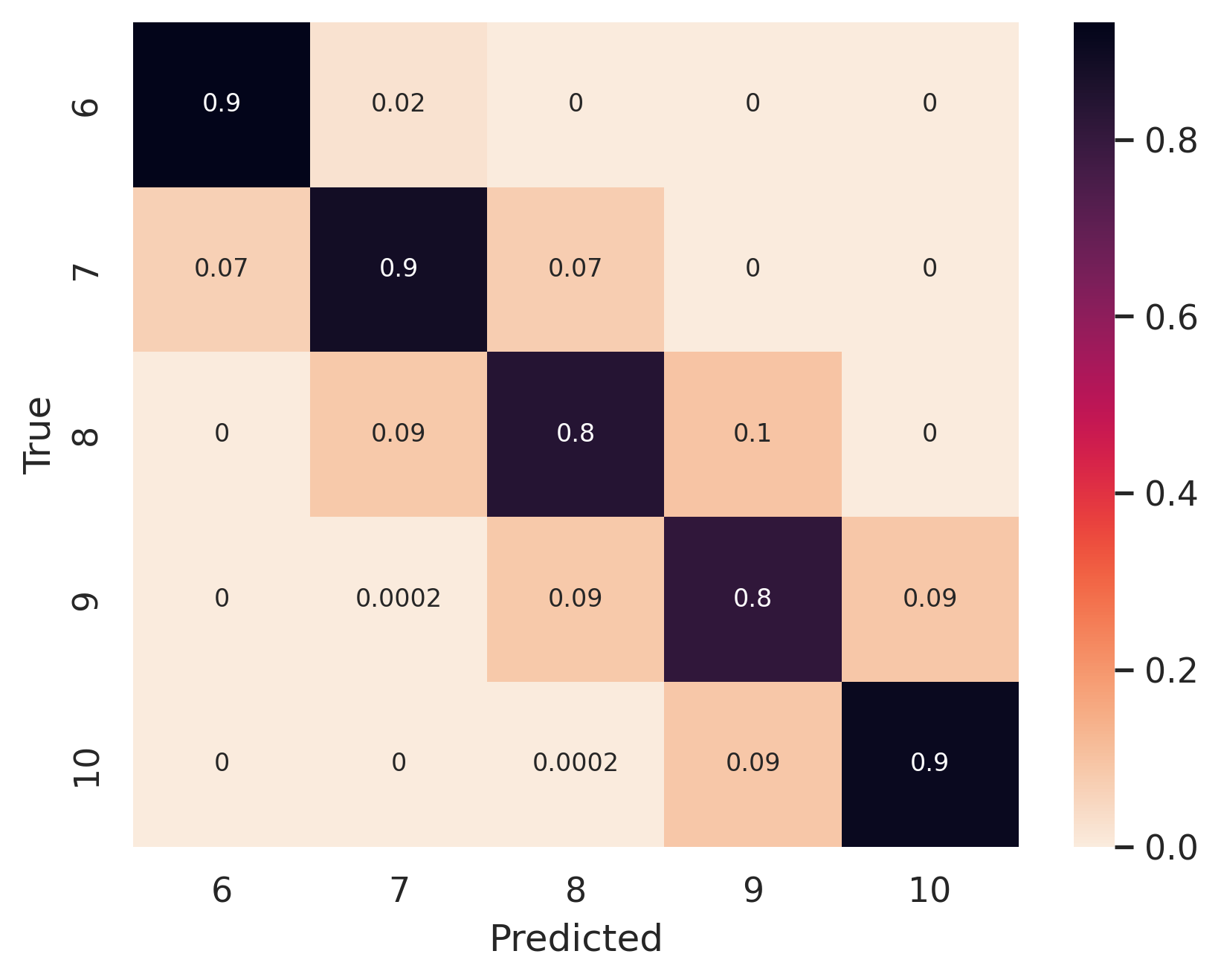} \\
      RFC & \includegraphics[width=\linewidth]{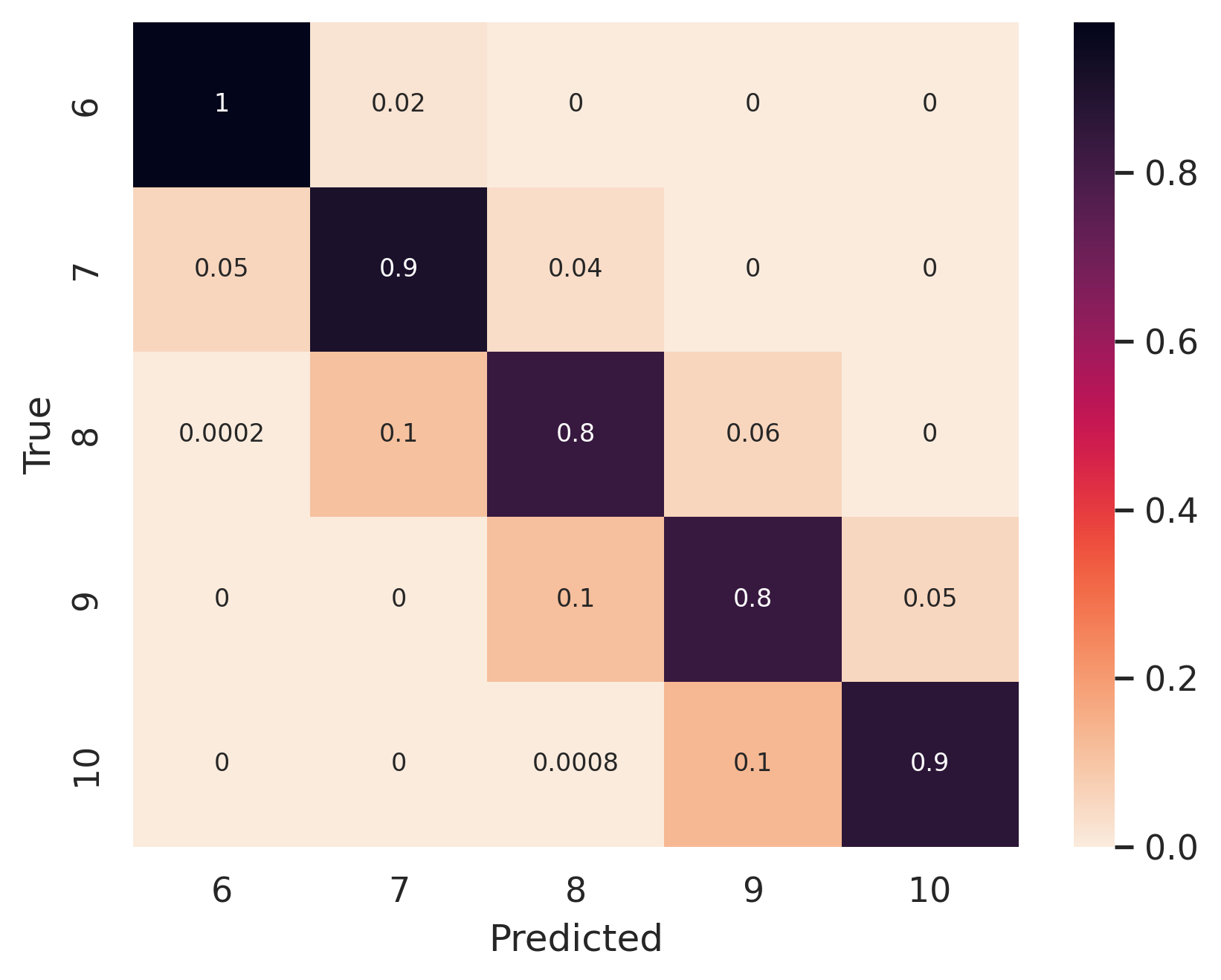}& \includegraphics[width=\linewidth]{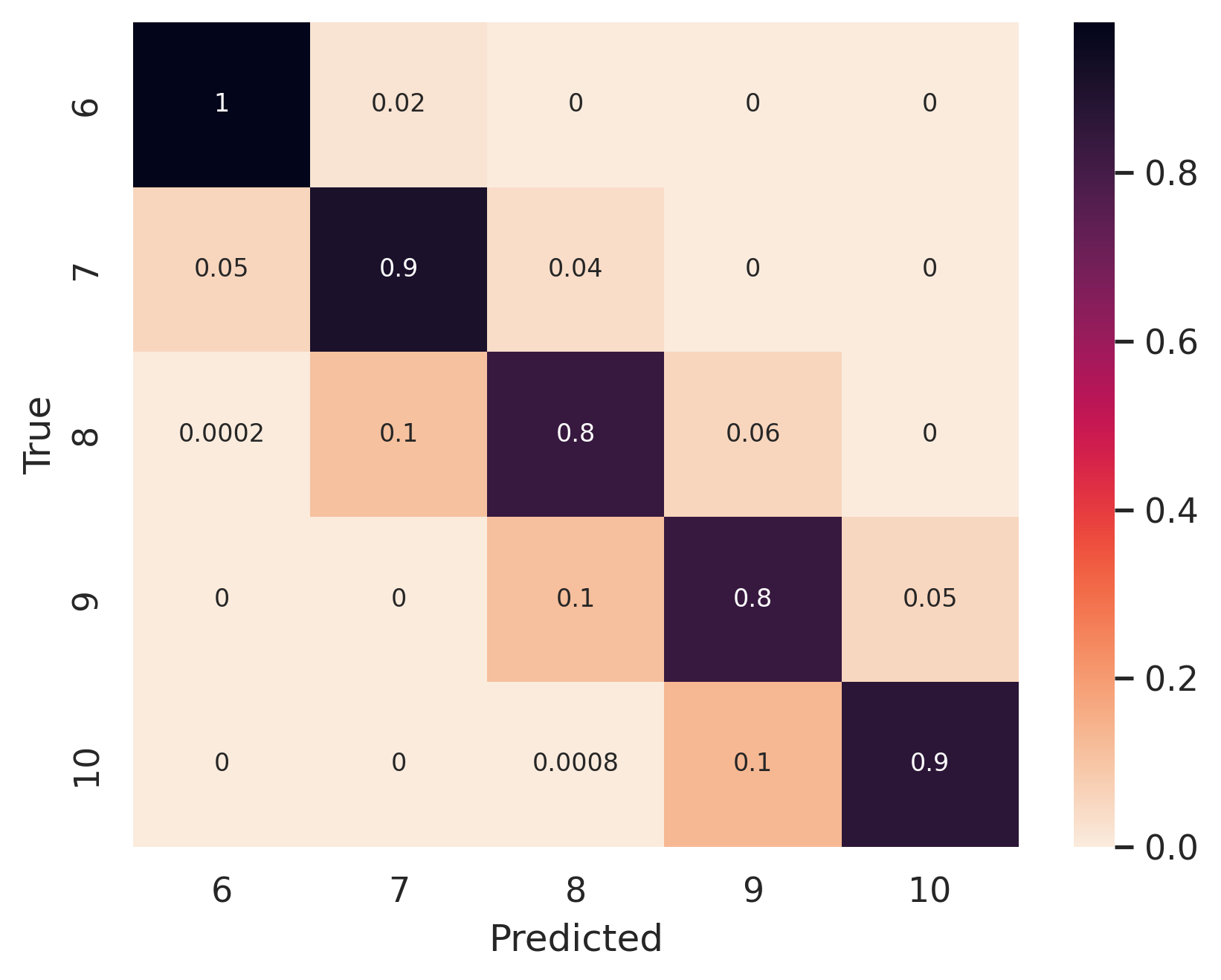} \\
      $\MLP_2$ & \includegraphics[width=\linewidth]{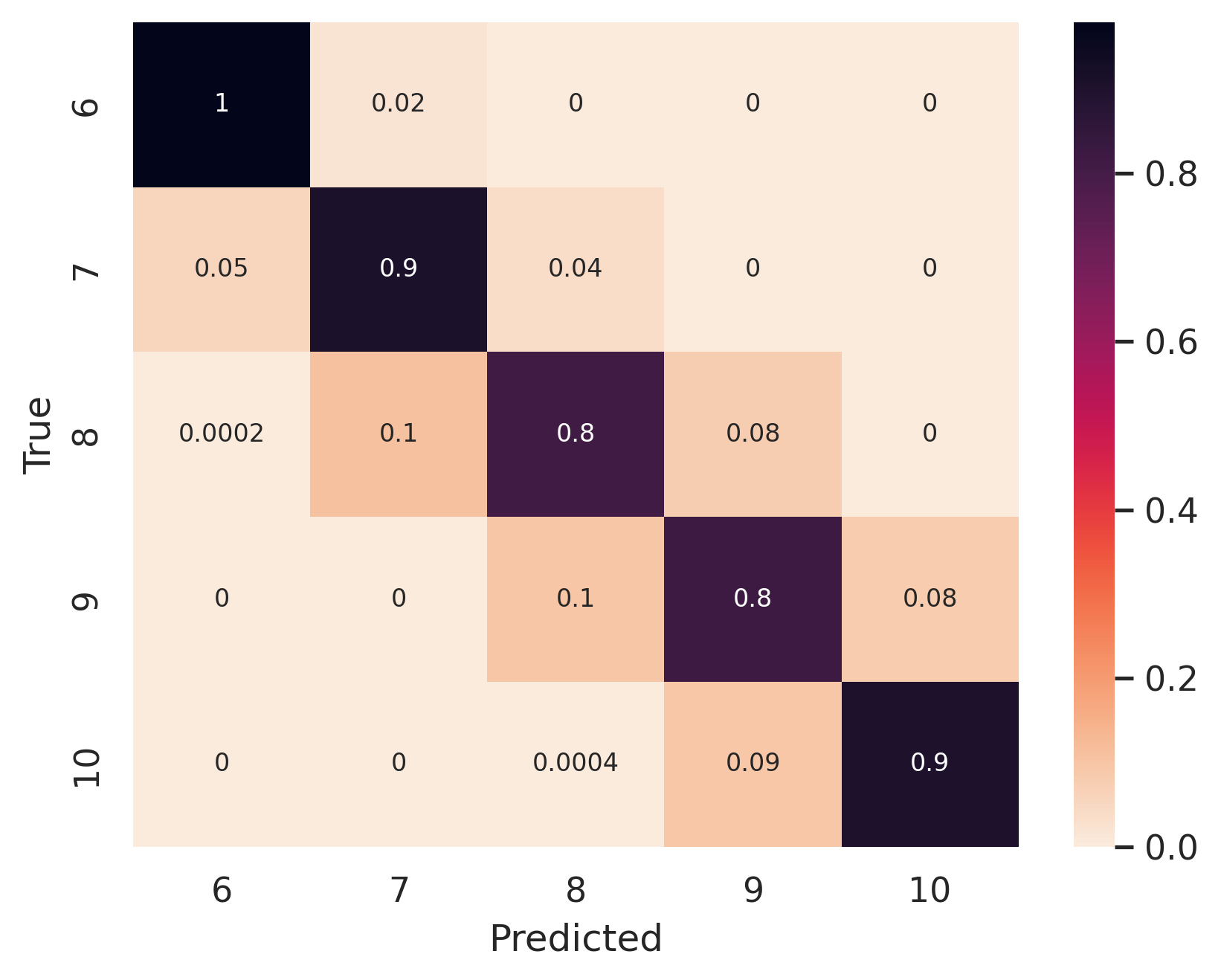} & \includegraphics[width=\linewidth]{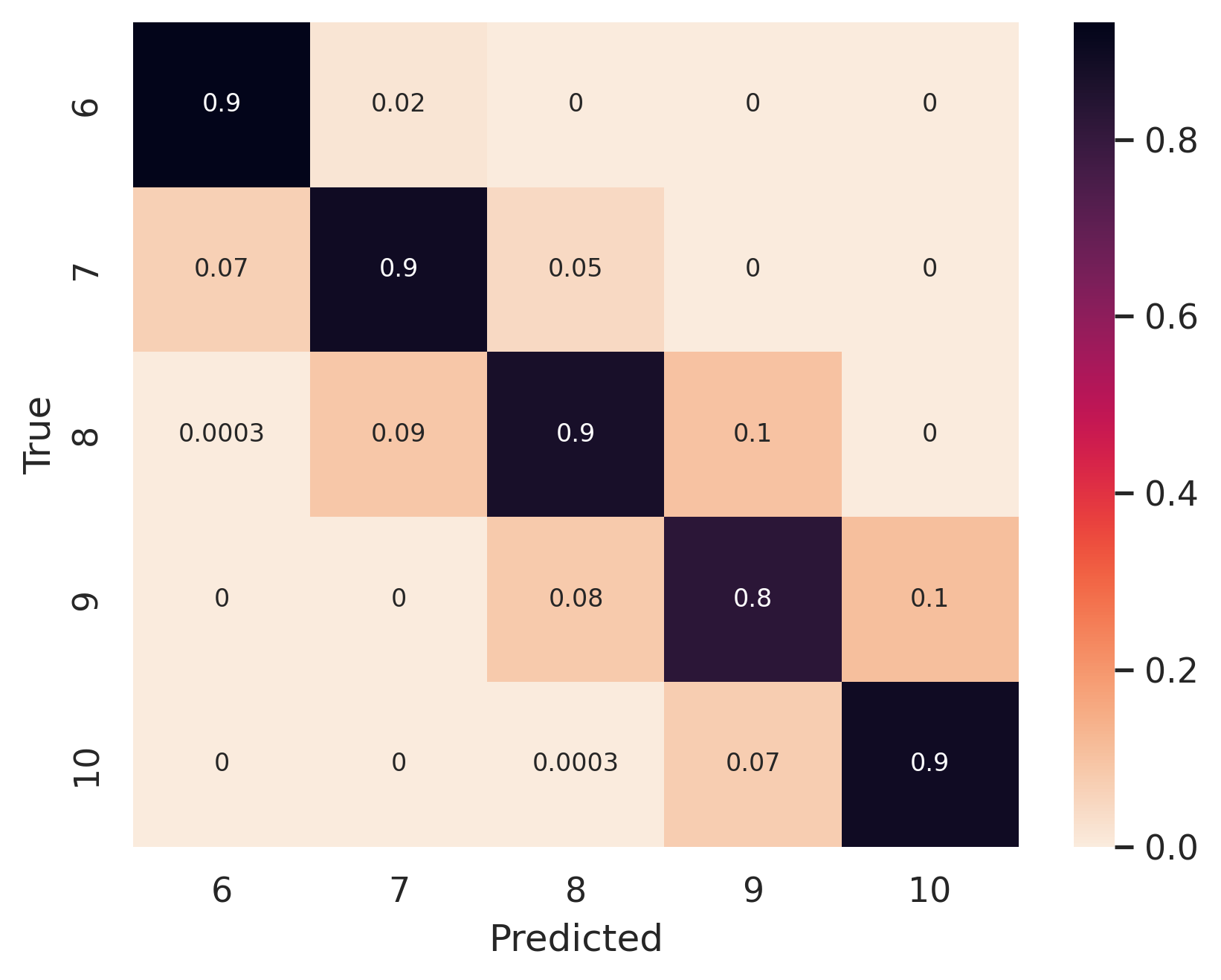}\\
      $\MLP_{102}$ & \includegraphics[width=\linewidth]{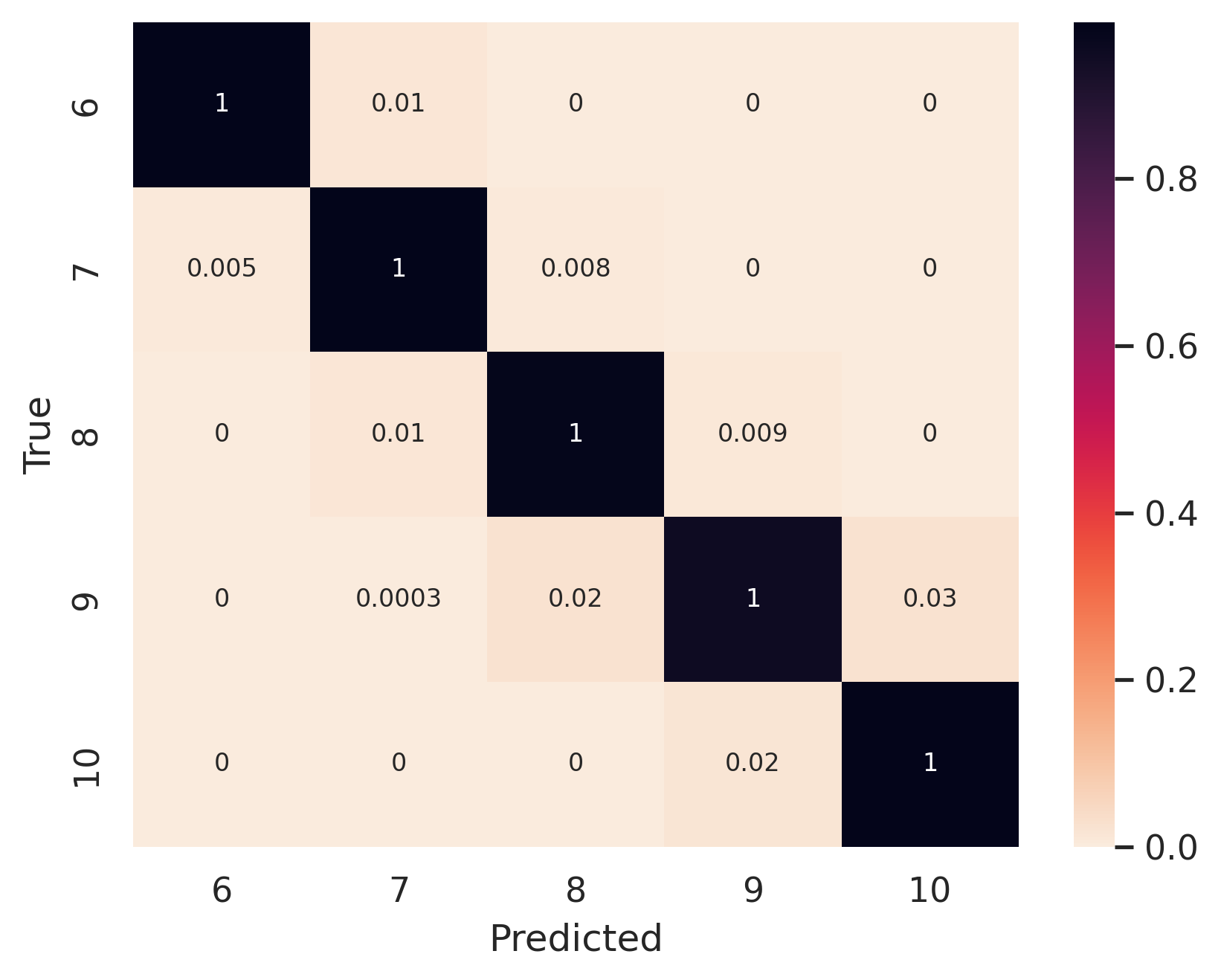} & \includegraphics[width=\linewidth]{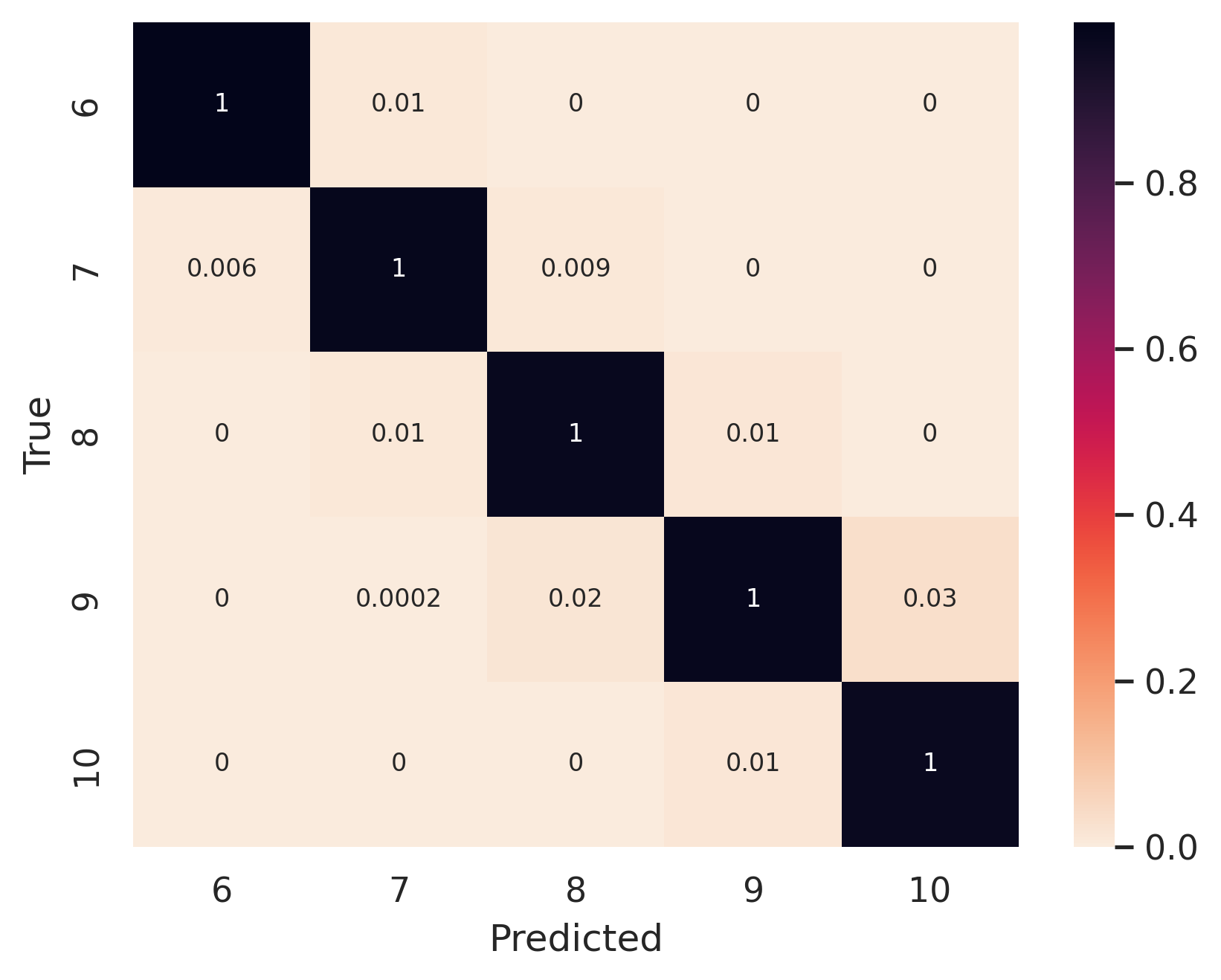}\\
      \bottomrule
  \end{tabular}
\caption{Comparison of confusion matrices. Confusion matrices for various models applied to the dataset of toric varieties of Picard rank two and dimension at least six: a Support Vector Machine with linear kernel, a Random Forest Classifier, and the neural networks $\MLP_2$ and $\MLP_{102}$.}
  \label{tab:compare_models_cm}
\end{table}
%-------------------------------------------------------------------------------
\section{Supplementary Discussion}
%-------------------------------------------------------------------------------
\subsubsection*{Comparison with Principal Component Analysis}
%-------------------------------------------------------------------------------
An alternative approach to dimensionality reduction, rather than fitting a linear model to $\log c_d$, would be to perform Principal Component Analysis (PCA) on this sequence and retain only the first few principal components. Since the vectors $(c_d)$ have different patterns of zeroes -- $c_d$ is non-zero only if $d$ is divisible by the Fano index~$r$ of $X$ -- we need to perform PCA for Fano varieties of each index $r$ separately. We analysed this in the weighted projective space case, finding that for each~$r$ the first two components of PCA are related to the growth coefficients $(A,B)$ from Theorem~\ref{thm:wps} by an invertible affine-linear transformation. That is, our analysis suggests that the coefficients $(A,B)$ contain exactly the same information as the first two components of PCA. Note, however, that the affine-linear transformation that relates PCA to $(A,B)$ varies with the Fano index $r$. Using $A$ and $B$ as features therefore allows for meaningful comparison between Fano varieties of different index. Furthermore, unlike PCA-derived values, the coefficients $(A,B)$ can be computed for a single Fano variety, rather than requiring a sufficiently large collection of Fano varieties of the same index.
%-------------------------------------------------------------------------------
\subsubsection*{Towards more general Fano varieties}
%-------------------------------------------------------------------------------
Weighted projective spaces and toric varieties of Picard rank two are very special among Fano varieties. It is hard to quantify this, because so little is known about Fano classification in the higher-dimensional and non-smooth cases, but for example this class includes only 18\% of the $\QQ$-factorial terminal Fano toric varieties in three dimensions. On the other hand, one can regard weighted projective spaces and toric varieties of Picard rank two as representative of a much broader class of algebraic varieties called toric complete intersections. Toric complete intersections share the key properties that we used to prove Theorems~\ref{thm:wps} and~\ref{thm:rank_2} -- geometry that is tightly controlled by combinatorics, including explicit expressions for genus-zero Gromov--Witten invariants in terms of hypergeometric functions -- and we believe that the rigorous results of this paper will generalise to the toric complete intersection case. All smooth two-dimensional Fano varieties and 92 of the 105 smooth three-dimensional Fano varieties are toric complete intersections~\cite{CoatesCortiGalkinKasprzyk2016}. Many theorems in algebraic geometry were first proved for toric varieties and later extended to toric complete intersections and more general algebraic varieties; cf.~\cite{Givental1996, Givental1998, GrossSiebert2019} and~\cite{Givental2001,Teleman2012}.

The machine learning paradigm presented here, however, applies much more broadly. Since our models take only the regularized quantum period sequence as input, we expect that whenever we can calculate $\Ghat_X$ -- which is the case for almost all known Fano varieties -- we should be able to apply a machine learning pipeline to extract geometric information about $X$.
%-------------------------------------------------------------------------------
\subsection*{Data availability} 
%-------------------------------------------------------------------------------
Our datasets~\cite{wpsData,rank2Data} and the code for the Magma computer algebra system~\cite{BosmaCannonPlayoust1997} that was used to generate them are available from Zenodo~\cite{Zenodo} under a CC0 license. The data was collected using Magma~V2.25-4.
%-------------------------------------------------------------------------------
\subsection*{Code availability} 
%-------------------------------------------------------------------------------
All code required to replicate the results in this paper is available from Bitbucket under an MIT license~\cite{supporting-code}.
%-------------------------------------------------------------------------------
\subsection*{Acknowledgements}
%-------------------------------------------------------------------------------
TC~is funded by ERC Consolidator Grant~682603 and EPSRC Programme Grant~EP/N03189X/1. AK~is funded by EPSRC Fellowship~EP/N022513/1. SV~is funded by the EPSRC Centre for Doctoral Training in Geometry and Number Theory at the Interface, grant number~EP/L015234/1. We thank Giuseppe Pitton for conversations and experiments that began this project, and thank John Aston and Louis Christie for insightful conversations and feedback.  We also thank the anonymous referees for their careful reading of the text and their insightful comments, which substantially improved both the content and the presentation of the paper.
%-------------------------------------------------------------------------------
\bibliographystyle{plain}
\bibliography{bibliography}
%-------------------------------------------------------------------------------
\end{document}